\documentclass[a4paper,10pt]{amsart}
\usepackage[plainpages=false]{hyperref}
\usepackage{amsfonts,amssymb,amsthm, mathrsfs, lscape}
\usepackage{verbatim}

\usepackage[T1]{fontenc}
\usepackage[utf8]{inputenc}
\usepackage{lmodern}
\usepackage{latexsym}
\usepackage{amssymb}
\usepackage{graphics}
\usepackage{graphicx}
\usepackage[space]{cite}

\usepackage{latexsym}
\usepackage{amsfonts}
\usepackage{amsmath}
\usepackage{latexsym}
\usepackage{amsfonts}
\usepackage{amssymb}
\usepackage{rawfonts}

\usepackage[usenames,dvipsnames]{pstricks}
\usepackage{epsfig}
\usepackage{pst-grad} 
\usepackage{pst-plot} 
\usepackage[space]{grffile} 
\usepackage{etoolbox} 
\makeatletter 

\usepackage{bm}

\renewcommand{\Re}{{\operatorname{Re}\,}}
\renewcommand{\Im}{{\operatorname{Im}\,}}

\renewcommand{\epsilon}{\varepsilon}

\newcommand{\kahler}{K\"{a}hler }

\newcommand{\poincare}{Poincar\'{e} }

\newcommand{\R}{{\mathbb R}}
\newcommand{\C}{{\mathbb C}}
\newcommand{\D}{{\mathbb D}}

\newcommand{\Z}{{\mathbb Z}}
\newcommand{\B}{{\mathbb B}}
\newcommand{\CH}{{\mathbb{CH}}}

\newcommand{\dbar}{\bar\partial}

\newcommand{\idzdz}{\sqrt{-1}dz\wedge\bar{dz}}

\newcommand{\iddbar}{\sqrt{-1}\partial \bar{\partial}}

\newcommand{\diag}{{\operatorname{diag}}}

\newcommand{\eqd}{\buildrel {\operatorname{def}}\over =}

\newcommand{\ccal}{\mathcal{C}}

\newcommand{\hcal}{\mathcal{H}}

\newcommand{\ocal}{\mathcal{O}}

\newtheorem{theorem}{{Theorem}}[section]

\newtheorem{lemma}[theorem]{Lemma}
\newtheorem{proposition}[theorem]{Proposition}
\newtheorem{conjecture}[theorem]{Conjecture}

\newenvironment{remark}{\medskip\noindent{\it Remark:\/} }{\medskip}

\newtheorem{definition}[theorem]{Definition}

\newtheorem{defandthm}[theorem]{Definition and Theorem}

\theoremstyle{empty}

\numberwithin{equation}{section}

\def \B {\mathbb B}
\def \C {\mathbb C}
\def \Z {\mathbb Z}
\def \R {\mathbb R}

\newcommand{\val}{{\operatorname{Val}}}

\def \P {\mathbb P}

\def \Aut {\text{Aut}}

\title[Bergman Kernel of complex hyperbolic]{On the Bergman Kernel of complex hyperbolic manifolds}

\author{Jingzhou Sun\textsuperscript{\dag}}
\thanks{\dag\  Department of Mathematics, Shantou University; \url{jzsun@stu.edu.cn}}

\begin{document}
	
	\begin{abstract}
We prove a formula for the Bergman kernel of polarized complex hyperbolic manifolds. The formula expresses the Bergman kernel as a sum over the geodesic loops in the manifold. As an application, we prove a result about the maximum and minimum of the Bergman kernel function. We also prove an estimate of the off-diagonal Bergman kernel.
	\end{abstract}

	\maketitle
	\setcounter{tocdepth}{1}
	 \tableofcontents

\section{Introduction}
Let $(M, \omega)$ be a Kähler manifold of dimension $n$, and let $L \to M$ be a holomorphic line bundle equipped with a Hermitian metric $h$ whose curvature is $-i\omega$. The Bergman space $\mathcal{H}_k$ consists of holomorphic sections of $L^k$ that are $L^2$-integrable, i.e.,
\[
\int_M |s|_h^2 \frac{\omega^n}{n!} < \infty.
\]
$\mathcal{H}_k$ is naturally a Hilbert subspace of $L^2(M, L^k)$. The Bergman kernel $K_k(x, y)$ is the distribution kernel of the orthogonal projection $L^2(M, L^k) \to \mathcal{H}_k$. 

The density of states function, or Bergman kernel function $\rho_k$, is the pointwise norm of $K_k(x, y)$ on the diagonal, i.e., $\rho_k(x) = K_k(x, x) e^{-k\varphi(x)}$, where $e^{-\varphi(x)}$ is a local representation of $h$. Alternatively, $\rho_k$ can be defined as
\begin{equation}\label{e-rhok}
	\rho_k(z) = \sup_{s \in \mathcal{H}_k,\, \|s\| = 1} |s(z)|^2.
\end{equation}
For domains in $\C^n$, the Bergman kernel is similarly defined for trivial bundles with trivial metrics.

The Bergman kernel has been intensively studied and used for domains in $\C^n$ and for \kahler manifolds (see, e.g., \cite{Fefferman1974,boutet1975singularite,kerzman1978cauchy,donaldson2001,Sun2011Expected,Donaldson2014Gromov,berman2011fekete,Donaldson15, dsz1,dsz2,dsz3,Shiffman2008NV,sz1,bsz0,bsz1}). Except in very symmetric cases (for instance complex space forms), explicit formulas are rare. A fundamental substitute is the Tian-Zelditch-Lu expansion (\cite{Tian1990On, Zelditch2000Szego, Lu2000On, Catlin}): when $M$ is compact,
\[
\rho_k(x)\sim\Big(\frac{k}{2\pi}\Big)^n\Big(1+\frac{b_1(x)}{k}+\frac{b_2(x)}{k^2}+\cdots\Big),
\]
where each $b_j(x)$ is a universal polynomial in the curvature of the Riemannian metric corresponding to $\omega$ and its covariant derivatives. In particular, $b_1(x)$ equals half the scalar curvature. There are also extensions of this expansion to symplectic manifolds and orbifolds (see \cite{MM,dailiuma}). Shiffman-Zelditch~\cite{sz2} and Charles~\cite{Charles2003} proved the asymptotic expansion for the off-diagonal Bergman kernel $K_k(x, y)$. The expansion for the Bergman kernel $K_k(x, y)$, especially for the diagonal $\rho_k(x)$, has been extensively used in the literature.

Given the importance of the Bergman kernel expansion, several natural questions arise. For example, under what conditions does the series $\sum_{i=1}^{\infty} \frac{b_i(x)}{k^i}$ converge? If it does converge, what is the difference between its sum and the actual value of $\rho_k$? More generally, even without convergence, one can consider the truncation
\[
\left(\frac{k}{2\pi}\right)^n \left(1 + \sum_{i=1}^{N(k)} \frac{b_i(x)}{k^i}\right),
\]
where $N(k) \to \infty$ as $k \to \infty$. In this case, the remainder
\[
\rho_k(x) - \left(\frac{k}{2\pi}\right)^n \left(1 + \sum_{i=1}^{N(k)} \frac{b_i(x)}{k^i}\right)
\]
is expected to be $O(k^{-\infty})$. When the local potential of $\omega$ is real analytic, Zelditch conjectured that the remainder of $K_k(x, y)$, with $N(k) = k/C$, should decay exponentially in $k$. This conjecture was confirmed by Rouby-Sjöstrand-Vu Ngoc in~\cite{Rouby-Sj}. Further new proofs were given by Deleporte~\cite{Deleporte2021} and Hezari-Kelleher-Seto-Xu~\cite{hezari2021property}.

\

When $X$ is a Riemann surface with hyperbolic metric, Berman in \cite{berman2012sharp} proved, prior to \cite{Rouby-Sj}, that $K_k(x,y)$ decays exponentially. The proof uses the fact that the pullback of the hyperbolic metric to the unit disk is the standard \poincare metric:
\[
\omega_0 = \frac{2\sqrt{-1}\,dz\wedge d\bar{z}}{(1-|z|^2)^2}.
\]
In particular, he showed that around a point $x$, the remainder is $O(e^{-\delta k})$ with $\delta$ arbitrarily close to $2\log(\cosh \frac{r(x)}{2})$, where $r(x)$ is the injectivity radius at $x$.

\medskip

In the process of exploring the remainder of the asymptotic expansion of the Bergman kernel on hyperbolic Riemann surfaces, we found an exact formula for the Bergman kernel on hyperbolic Riemann surfaces in the first version, and subsequently extended the formula to complex hyperbolic manifolds.
After the first version of this article, the author in \cite{sun-ber-abelian} further extended the result to the case of polarized abelian varieties. 

\

Let $(X,\omega)$ be a \kahler manifold equipped with a complete complex hyperbolic metric, whose Kähler form is $\omega$. Let $(L,h)$ be a Hermitian holomorphic line bundle over $X$, whose curvature is $-i\omega$. For each point $p \in X$, let $\mathfrak{G}_p$ be the set of oriented geodesic loops based at $p$. More explicitly, an element of $\mathfrak{G}_p$ is a geodesic: \[\gamma: [0,l]\to X, l>0, \gamma(0)=p,\gamma(l)=p, \] parametrized by arc length and oriented by its tangent vector. We remind the readers that the tangent vectors at the two endpoints may differ.
For each $\gamma \in \mathfrak{G}_p$, let $\ell(\gamma)$ denote its length, and let $e^{2\pi \alpha_\gamma \sqrt{-1}}$ be the holonomy the Chern connection of $h$ around $\gamma$.

Our first and main result is the following:

\begin{theorem}\label{thm-main}
Let $(X,\omega)$ be a \kahler manifold with complete complex hyperbolic metric. Let $L \to X$ be a polarization, equipped with a Hermitian metric $h$ on $L$ with curvature $- i \omega$. 
  Let $k_0=2n+[\frac{n}{2}]+1$ ($k_0=n+[\frac{n}{2}]+1$ when $X$ is compact).
  Then for all $k \geq k_0$, the Bergman kernel function $\rho_k(p)$ satisfies the following equality:
\[
\rho_k(p)
= \frac{1}{(4\pi)^n} \frac{(2k-1)!}{(2k-n-1)!}\left(1+\sum_{\gamma\in\mathfrak{G}_p}
\cosh^{-2k}\!\Big(\frac{\ell(\gamma)}{2}\Big)\cos\!\big(2\pi k\alpha_\gamma\big)\right).
\]

\end{theorem}
\begin{remark}
    \begin{itemize}
        \item We should remark here that the main novelty of this result lies in the explicit formula for the Bergman kernel function, which involves global data like lengths of geodesic loops and the holonomy of the line bundle, although Lu-Zelditch \cite{lu2016szegHo} and Ma-Marinescu \cite{ma2015exponential} have proved formulas expressing the Bergman kernel in terms of the Bergman kernel of the universal cover (see Definition and Theorem \ref{def-and-thm-k-gamma}). This type of results for Bergman kernel have never been seen before.
		\item It is easy to notice that there is a striking analogy between this result and the Selberg trace formula (see \cite{patterson1978dennis}). We will be excited to see a deeper connection to the Selberg trace formula.
    \end{itemize}
\end{remark}

When $\dim X=1$ and $X$ has cusps, it was first proved in \cite{AMM} that, under some extra condition on the Hermitian metric on the line bundle, the supremum of $\rho_k$ is approximately $(\frac{k}{2\pi})^{3/2}$. And it was noticed in \cite{Punctured} that $\rho_k$ exhibits a quantum phenomenon around the cusps. In view of Theorem \ref{thm-main}, this can be understood as following. For a point close to the cusps, there are many short geodesic loops. So if the holonomies make the terms $\cosh^{-2k}(\frac{\gamma}{2})\cos(2\pi k\alpha_\gamma)$ cancel out each other, $\rho_k$ becomes small, otherwise $\rho_k$ can be large. It is also worth mentioning that this quantum phenomenon was also proved for Cheng-Yau metrics in \cite{sun2024ChengYau}.

Given the importance of complex hyperbolic manifolds and Bergman kernels,
there are certainly many potential applications of our main result. We show one application of this result.
Recall that the systole length of $X$ is defined as minimum of the lengths of all closed geodesics in $X$. 

For simplicity, let $l_1 = \mathrm{sys}(X)$, and we call a closed geodesic of systole length a systole.
\begin{theorem}\label{thm-max-min}
	Let $X$ be a hyperbolic Riemann surface of finite topology with no cusps. Assume further that no two systoles intersect. Then, there exists $l'>l_1$ and $k_0 > 0$ such that for all $k \geq k_0$, if the holonomies of $(kL, h^k)$ around all systoles are $1$ ($-1$, respectively), the maximum (minimum, respectively) of $\rho_k$ is attained in the exponentially small region $$\left\{ x \in X ,\Big| d(x,\gamma) < \left(\frac{\cosh\left(\frac{l_1}{2}\right)}{\cosh\left(\frac{l'}{2}\right)}\right)^{k} \right\}$$ for some systole $\gamma$.
\end{theorem}
\begin{remark}
	By Proposition~\ref{prop-kx}, when $L = K_X$, the holonomy around any systole $\gamma$ is always $1$. So the maximum part of the above theorem applies to the canonical bundle.
\end{remark}

\medskip
It is known that, for a positive line bundle over a compact \kahler manifold, when $y$ is within a $O(k^{-1/2})$ small neighborhood of $x$, the off-diagonal Bergman kernel decays exponentially like $e^{-\frac{k}{2}d^2(x,y)}$ as $d(x,y)$ increases(see for example, \cite{Shiffman2008}). For positive distances, when the metric is real analytic, when $d(x,y)\ge\varepsilon>0$ one has the uniform bound
\[
\big|K_k(x,y)\big|_{h^k}\le C_\varepsilon e^{-c_\varepsilon k},
\]
i.e. the kernel is exponentially small in $k$ off the diagonal. In the case of hyperbolic Riemann surfaces of finite volume and $L=K_X$, Aryasomayajula-Majumder (\cite{aryasomayajula2018off,aryasomayajula2020estimates}) proved similar results, which involves the distance and the injectivity radius of $X$, for the off-diagonal kernel when the distance between $x$ and $y$ is smaller than the injectivity radius of $X$.
Our result is the first result that shows that decaying rate of the off-diagonal Bergman kernel depends explicitly on the distance between $x$ and $y$, globally. Notice that our result also allow infinite volume and cusps.

We also have the following estimate for off-diagonal Bergman kernel:
\begin{theorem}\label{thm-off-diag}
	Let $(X,\omega,L,h)$ be a poloarized complex hyperbolic manifold as in the setting of Theorem \ref{thm-main}. Let $k_0=2n+[\frac{n}{2}]+1$ ($k_0=n+[\frac{n}{2}]+1$ when $X$ is compact).
  Then for all $k \geq k_0$, and $x\neq y$, we have 
	\[|K_k(x,y)|_{h^k}\leq \frac{1}{(4\pi)^n} \frac{(2k-1)!}{(2k-n-1)!}\sum_{\gamma\in \mathfrak{G}_{x,y}}\cosh^{-k}\frac{\ell(\gamma)}{2} .
	\]
	And if there is a unique geodesic minimizing $d(x,y)$, there exist $k_1>0, C>0$ such that for all $k \geq k_1$,
	\[
	\big| |K_k(x,y)|_{h^k}-\frac{1}{(4\pi)^n} \frac{(2k-1)!}{(2k-n-1)!}\cosh^{-k}\frac{d(x,y)}{2}\big| \leq C k^n \cosh^{-k}\frac{l_2}{2},
	\]
	where $l_2=\min_{\gamma\in \mathfrak{G}_{x,y}, \ell(\gamma)>d(x,y)} \ell(\gamma)$.

\end{theorem}

As a corollary of Theorem \ref{thm-f-b}, we have the following result, which is of independent interest:
\begin{theorem}\label{thm-jak}
Let $J(a,k) = \int_{-\pi/2}^{\pi/2} e^{2at} \cos^{2k-2} t\, dt$. Then for $k\geq 2$, $J(z,k)$ has no zeros in the strip $\{|\Im z| < k\} \subset \C$.
\end{theorem}
We propose the following rigidity conjecture:

\begin{conjecture}
	Let $(X,L)$ and $(X',L')$ be polarized complex hyperbolic manifolds of dimension $n$, that satisfy the conditions of Theorem \ref{thm-main}. Suppose there is a diffeomorphism $\Phi:X\to X'$ and an integer $k\geq k_0$ ($k_0$ is the same as that in Theorem \ref{thm-main}) such that
	\[
	\rho_{L,k}=\Phi^*\rho_{L',k}.
	\]
	Then $\Phi$ is either a biholomorphism or an anti-biholomorphism, and moreover
	\[
	\Phi^*((L')^{\!k})\cong L^{\!k}\qquad\text{or}\qquad \Phi^*((L')^{\!k})\cong\overline{L^{\!k}},
	\]
	respectively. Here $\overline{L}$ denotes the complex-conjugate line bundle of $L$.
\end{conjecture}

A weaker version of this conjecture is the following conjecture:
\begin{conjecture}\label{conj-second}
	Let $(X,L)$ be a polarized complex hyperbolic manifold of dimension $n$, that satisfy the conditions of Theorem \ref{thm-main}. Let $L'$ be another polarization of $X$ with respect to a Hermitian metric $h'$ on $L'$ such that $\Theta_{h'}=\Theta_h$. Suppose that $$\rho_{L',k}=\rho_{L,k}$$ for some $k\geq k_0$ ($k_0$ is the same as that in Theorem \ref{thm-main}). Then $kL'$ is isomorphic to $kL$.
\end{conjecture}
In the case of abelian varieties, we have proved the corresponding conjecture in \cite{sun-ber-abelian}.

\medskip

We briefly discuss the proof of these theorems. To simplify the notations, we will refer to dimension 1 complex hyperbolic cylinders as hyperbolic cylinders.
The main strategy for proving Theorem \ref{thm-main} consists of four steps: first, prove the theorem on hyperbolic cylinders; second, prove the theorem on higher dimensional complex hyperbolic cylinders; third, prove the theorem on complex hyperbolic cusps; and fourth,
by using a formula for the pullback of the Bergman kernel to the universal covering, we combine the results obtained in the first three steps to obtain the formula for general complex hyperbolic manifolds. This approach was inspired by the proof of the Selberg trace formula (see \cite{patterson1978dennis}). The primary challenge in the first step is the absence of a direct formula for the Fourier transform of $1/J(a,k)$. Our method relies on precise estimates for the Bergman kernel in the disk model and a technique involving the enlargement of the hyperbolic cylinder. After the proof of the formula for the Fourier transform of $1/J(a,k)$, we compute the holonomy along any geodesic loop in a hyperbolic cylinder. Then we do not directly prove the main theorem for hyperbolic cylinders, but compress the proof to the case of general dimensions, as it is only a special case.
The primary challenge in the second step is the construction of the complex hyperbolic cylinder and the calculations of the holonomy of higher dimensional complex hyperbolic cylinders. In fact, both the construction of the complex hyperbolic cylinder and the calculation of the holonomy are non-trivial.

 In the third step, we use complex hyperbolic cylinders to approximate complex hyperbolic cusps. Compared to the dimension one case, the classification of complex hyperbolic cusps of higher dimension is much more complicated. Interestingly, following our strategy of proof, we do not need to study general complex hyperbolic cusps. In fact, we only need to study those complex hyperbolic cusps that are quotients of the complex hyperbolic space form by the discrete groups genereted by one parabolic element.
 As will be seen, there are two types of such complex hyperbolic cusps. For the first type, the approximation is similar to the one-dimensional case, as seen in \cite{sun2024apde}. Geometrically, the complex hyperbolic half cylinders approximate the complex hyperbolic cusps of type one. The approximation of the complex hyperbolic cusps of type two is a major challenge. The approximation maps we used in this article is out of brute force trying. It remains a very interesting problem to understand the geometric meaning of this approximation. This approximation also provides a very good model for the study of degenerating families of \kahler-Einstein manifolds. Once we have constructed the approximations of complex hyperbolic cusps by complex hyperbolic cylinders, we go on to prove that the Bergman kernel of the complex hyperbolic cylinders approximate that of the complex hyperbolic cusps. It is interesting to see that the proofs for the two types of cusps are basically parallel, despite the tremendous difference in the construction of the approximation maps. It is also worth noticing that in proving the approximations of the Bergman kernels, we do not need to use H\"{o}mander's $L^2$ estimates.

 Finally, in the fourth step, we prove that the pull back of the off-diagonal Bergman kernel to the unit ball model can be expressed as a summation of the translations of the Bergman kernel of the unit ball over the whole Fuchsian group. Theorem \ref{thm-off-diag} is a direct consequence of this summation formula. Theorem \ref{thm-max-min} is a careful analization of the Bergman kernel near the systoles.

 It is worth mentioning that one possible way to prove the claim in the proof of Lemma \ref{lem-int-s-sum-g-s-p} is to show that the Bergman kernel function on $X$ is bounded. In dimension 1, under the assumption of finite topology, or in general dimension, under the assumption of positive injectivity radius, this is indeed the case. However, we do not know if this is true in general. It is an interesting question of independent interest. 

Since complex hyperbolic cusps are important models for singular \kahler-Einstein metrics (see for example \cite{datar2023kahler}), the two cusp models, especially the second model, also provides good playground for this subject.

\medskip

\noindent\textbf{Organization of the paper.}
 In Section~\ref{sec-space-form}, we first recall the basics of the Bergman kernel, holonomy and the complex hyperbolic space form, including the ball model and the Siegel domain model. Then we prove Theorem \ref{thm-parabolic-conjugacy}, which classifies the parabolic elements. Section~\ref{sec-disk} provides estimates for the Bergman kernel using geodesic balls (the disk model) and proves Theorem~\ref{thm-disk-model}, which is used in the proof of Theorem \ref{thm-f-b}. In Section~\ref{sec-hyper-cyl}, we first study the Bergman kernel of the hyperbolic cylinder, hence introducing the function $J(a,k)$; then we prove Theorem~\ref{thm-f-b}. We also prove Theorem~\ref{thm-jak} as an application of Theorem~\ref{thm-f-b}.
 In Section~\ref{sec-comp-hyper-cyl}, we first construct the complex hyperbolic cylinders, then study the Bergman kernel of the complex hyperbolic cylinders. In Section~\ref{sec-holonomy}, we calculate the holonomy of the Chern connections associated to the standard Hermitian metrics on the trivial line bundles on complex hyperbolic cylinders and prove that Theorem~\ref{thm-main} holds in this case. In Section~\ref{sec-cusp-I}, we first construct the complex hyperbolic cusp of type one, then prove that we can use complex hyperbolic cylinders to approximate complex hyperbolic cusp of type one. Then we prove that Theorem~\ref{thm-main} holds in this case. Similarly, in Section~\ref{sec-cusp-II}, we first construct the complex hyperbolic cusp of type two, then prove that we can use complex hyperbolic cylinders to approximate complex hyperbolic cusp of type two. Then we prove that Theorem~\ref{thm-main} holds in this case. 
 Section~\ref{sec-proof-main} proves Theorems~\ref{thm-main}, \ref{thm-max-min} and \ref{thm-off-diag}.

\medskip

\noindent\textbf{Acknowledgements.} The author would like to thank Professor Song Sun for many very helpful discussions. The author also thanks the IASM of Zhejiang University for their hospitality during the preparation of this manuscript.

\section{Backgrounds and preparations}\label{sec-space-form}

\subsection{Bergman kernel and holonomy}

\subsubsection*{The Bergman kernel}

Let $(M,L)$ be a polarized manifold of dimension $n$. Let $h$ be a Hermitian metric whose curvature is $-i\omega$, where $\omega$ is a \kahler form. 
The Bergman kernel $K_k(x,y)$ is the integral kernel of the orthogonal projection $L^2(M,L^k)\to\hcal_k$, characterized by
\[
s(y)=\int_M \langle s(x),K_k(x,y)\rangle_{h^k}\,\frac{\omega^n(x)}{n!}\qquad(\forall s\in\hcal_k).
\]
The construction of $K_k(x,y)$ is as follows.
For each point $y\in X$, valuation map $\val_y:\hcal_k\to L^k_y$ is a bounded functional. So there is a unique $L^2$-integrable holomorphic section $K_y$ of $L^k$ with values in $\overline{L^k_y}$ such that \begin{equation}
	\langle s,K_y \rangle=s(y), \quad \forall s\in \hcal_k.
\end{equation}
The Bergman kernel is then defined as $K(x,y)\eqd K_y(x)$. If $\{s_i\}_{1\leq i<\infty}$ is an orthonormal basis of $\hcal_k$, then \[K(x,y)=\sum_{i=1}^{\infty}s_i(x)\otimes \bar{s}_i(y).\]
From the definition of $K(x,y)$, this summation is independent of the choice of the orthonormal basis.

\subsubsection*{Holonomy of line bundles}
Let $(L,h)\to M$ be a Hermitian line bundle endowed with a unitary connection $\nabla$. Fix a base point $x\in M$ and let $\gamma:[0,1]\to M$ be a oriented loop with $\gamma(0)=\gamma(1)=x$.

If $\gamma$ is piecewise $C^1$, parallel transport along $\gamma$ is defined by solving the ordinary differential equation
\[
\nabla_{\dot\gamma(t)}s(t)=0,\qquad s(0)=v\in L_x,
\]
which yields a unique section $s(t)$ along $\gamma$. The parallel transport map
\[
P_\gamma:\;L_x\to L_x,\qquad P_\gamma(v)\eqd s(1),
\]
is a unitary linear map of the fibre $L_x$, hence multiplication by a unit complex number. The holonomy of $(L,\nabla)$ along $\gamma$ is the element
\[
\operatorname{Hol}_\nabla(\gamma)\in\mathrm{U}(1)
\]
characterising $P_\gamma$, i.e. $P_\gamma=\operatorname{Hol}_\nabla(\gamma)\cdot\mathrm{Id}_{L_x}$. Since the Chern connection of a holomorphic line bundle is determined by the Hermitian metric $h$, we also use $\operatorname{Hol}_h(\gamma)$ or $\operatorname{Hol}_L(\gamma)$ to denote the holonomy of the Chern connection of $(L,h)$ along $\gamma$.

\subsubsection*{Poisson summation formula}

\begin{theorem}[Poisson summation formula]
Let $\hat{g}$ be the Fourier transform of $g$. Suppose $g(x) = \int_{\R} \hat{g}(y) e^{2\pi i x y} dy$ with $|g(x)| \leq A(1+|x|)^{-1-\delta}$ and $|\hat{g}(y)| \leq A(1+|y|)^{-1-\delta}$ for some $\delta > 0$. Then
\[
\sum_{c=-\infty}^{\infty} g(x + c) = \sum_{\xi=-\infty}^{\infty} \hat{g}(\xi) e^{2\pi i \xi x}.
\]
\end{theorem}

\subsection{Complex hyperbolic space form}
We recall the construction of complex hyperbolic space $\CH^n$, following \cite{chen1974hyperbolic} and \cite{chg-Goldman1999}.

Let $Z = (Z_1, \ldots, Z_{n+1})$ be coordinates on $\C^{n+1}$. Define $\C^{n,1}$ as $\C^{n+1}$ equipped with the Hermitian form
\[
\langle Z, W \rangle = Z_1 \bar{W}_1 + \cdots + Z_n \bar{W}_n - Z_{n+1} \bar{W}_{n+1}.
\]
Let $U(n,1)$ denote the group of unitary automorphisms of $\C^{n,1}$. A vector $Z \in \C^{n,1}$ is called negative if $\langle Z, Z \rangle < 0$. The complex hyperbolic $n$-space $\CH^n\subset \P \C^{n,1}$ is defined as the subset consisting of negative lines. The biholomorphism group of $\CH^n$ is the image of $U(n,1)$ in $\mathbf{PGL}(\C^{n,1})$, denoted by $PU(n,1)$.

The unit ball model $\B^n \subset \C^n$ is identified with $\CH^n$ via the map $\C^n \to \P \C^{n,1}$, $z \mapsto [z,1]$. Under this identification, the closure $\overline{\B^n}$ is identified with $\overline{\CH^n}$. By the Brouwer fixed point theorem, every $g \in PU(n,1)$ has a fixed point in $\overline{\CH^n}$. There are three types of elements in $PU(n,1)$:
\begin{itemize}
	\item \emph{Elliptic}: one fixed point in $\CH^n$,
	\item \emph{Parabolic}: one fixed point on $\partial\CH^n$,
	\item \emph{Loxodromic}: two fixed points on $\partial\CH^n$.
\end{itemize}

\subsection{Ball model}

Let $x = (x_1, \ldots, x_n)$ be coordinates on $\C^n$.

The Bergman metric (complex hyperbolic metric) on $\B^n$ is given by $\omega_B = -\sqrt{-1} \partial \bar{\partial} h_B$, where $h_B = 2 \log(1 - |x|^2)$. The corresponding volume form is
\[
\frac{\omega_B^n}{n!} = \frac{(2\sqrt{-1})^n}{(1 - |x|^2)^{n+1}} dx_1 \wedge d\bar{x}_1 \cdots dx_n \wedge d\bar{x}_n.
\]
The group $PU(n,1)$ acts on $\B^n$ by biholomorphisms and isometries. The following proposition is easy to verify.
\begin{proposition}\label{ball-geodesic}
	For any $p \in \B^n$, there is a unique geodesic $\gamma_{0,p}$ from the origin to $p$, which is a straight line segment in $\C^n$.
\end{proposition}
 The distance between $x$ and $y$ is given by
 \begin{equation}
	\cosh \frac{d(x, y)}{2} = \frac{1 - |x \cdot y|^2}{\sqrt{1 - |x|^2} \sqrt{1 - |y|^2}}.
 \end{equation}

 \begin{proposition}\label{prop-dist-line}
	Let $p=(p_1,p')\in\B^n$ with $r^2=|p_1|^2+\|p'\|^2$. Let $D_1=\{z\in\B^n\big| z_j=0, j>1\}$ be the one-dimensional disk. The point $q_0\in D_1$ realizing the distance from $p$ to $D_1$ is
	\[
	q_0=(p_1,0),
	\]
	and the distance is given by
	\[
	\cosh^2\!\Big(\frac{d(p,L)}{2}\Big)=\frac{1-|p_1|^2}{1-r^2},
	\qquad
	d(p,L)=2\operatorname{arcosh}\!\Big(\sqrt{\frac{1-|p_1|^2}{1-r^2}}\Big).
	\]
	In particular, if $p_1=0$ then $\cosh^2\!\big(\tfrac{d(p,L)}{2}\big)=\dfrac{1}{1-r^2}$.
	\end{proposition}

	\begin{proof}
		For $q=(t,0)\in L$ with $|t|<1$ the ball distance formula gives
	\[
	\cosh^2\!\Big(\frac{d(p,q)}{2}\Big)=\frac{|1-p_1\overline{t}|^2}{(1-r^2)(1-|t|^2)}.
	\]
	Write $p_1=|p_1|e^{i\phi}$ and $t=\rho e^{i\psi}$. For fixed $\rho$ the numerator $|1-p_1\overline{t}|^2=1-2|p_1|\rho\cos(\phi+\psi)+|p_1|^2\rho^2$ is minimized when $\psi=-\phi$, so it suffices to consider real $\rho\in[0,1)$. Setting $f(\rho)=\dfrac{1-2|p_1|\rho+|p_1|^2\rho^2}{1-\rho^2}$ and differentiating yields the unique minimizer $\rho=|p_1|$. Hence $\overline{t}= \overline{p_1}$ and $t=p_1$, so $q_0=(p_1,0)$. Substituting $t=p_1$ gives
	\[
	\cosh^2\!\Big(\frac{d(p,q_0)}{2}\Big)=\frac{(1-|p_1|^2)^2}{(1-r^2)(1-|p_1|^2)}=\frac{1-|p_1|^2}{1-r^2},
	\]
	which yields the stated distance formula.
	\end{proof}

The Bergman space $\hcal_k$ on $\B^n$ consists of holomorphic functions that are $L^2$-integrable with respect to the Hermitian metric $h_B^k = (1 - |x|^2)^{2k}$ and the volume form above. The inner product is
\[
\langle f, g \rangle = \int_{\B^n} f \bar{g} (1 - |x|^2)^{2k} \frac{\omega_B^n}{n!}.
\]

It is easy to show that the Bergman kernel function on $\B^n$ is
\[
\rho_k(z) = \frac{1}{(4\pi)^n} \frac{(2k-1)!}{(2k-n-1)!}.
\]

In particular, the peak section at $0$ is the constant section $s_0(z) = \sqrt{\frac{1}{(4\pi)^n} \frac{(2k-1)!}{(2k-n-1)!}}$. The off-diagonal Bergman kernel at $(z, 0)$ is
\[
K_k(z, 0) = \frac{1}{(4\pi)^n} \frac{(2k-1)!}{(2k-n-1)!}.
\]
Thus,
\[
|K_k(z, 0)|_{h_B} = \frac{1}{(4\pi)^n} \frac{(2k-1)!}{(2k-n-1)!} (1 - |z|^2)^{2k}.
\]
\subsection{Siegel domain model}
Let $z = (z_1, \ldots, z_n)$ be the standard complex coordinates on $\C^n$, and set $z' = (z_1, \ldots, z_{n-1})$. The Siegel domain $\mathfrak{H}^n \subset \C^n$ is defined by
\[
2\Re z_n > |z'|^2.
\]
This domain provides another model for $\CH^n$, via the map $\mathfrak{H}^n \to \P \C^{n,1}$ given by
\begin{equation}\label{e-hn-chn}
	(z', z_n) \mapsto [z', \tfrac{1}{2} - z_n, \tfrac{1}{2} + z_n].
\end{equation}

The identification between the ball model $(\B^n, x)$ and the Siegel domain $(\mathfrak{H}^n, z)$ is given by the Cayley transform $\ccal_{HB} : \mathfrak{H}^n \to \B^n$:
\begin{align*}
	x_j &= \frac{2z_j}{1 + 2z_n}, \qquad & z_j = \frac{x_j}{1 + x_n} & (1 \leq j < n), \\
	x_n &= \frac{1 - 2z_n}{1 + 2z_n}, \qquad & z_n = \frac{1}{2} \frac{1 - x_n}{1 + x_n}.&
\end{align*}

The complex hyperbolic metric $\omega_H$ on $\mathfrak{H}^n$ is defined by $\omega_H = -\sqrt{-1} \partial \bar{\partial} \log h_H$, where $h_H = (2\Re z_n - |z'|^2)^2$. A straightforward computation yields
\begin{align*}
	\omega_H = \frac{2\sqrt{-1}}{(2\Re z_n - |z'|^2)^2} \Big[ & (2\Re z_n - |z'|^2) \sum_{j=1}^{n-1} dz_j \wedge d\bar{z}_j \\
	& + (dz_n - \sum_{j=1}^{n-1} \bar{z}_j dz_j) \wedge (d\bar{z}_n - \sum_{j=1}^{n-1} z_j d\bar{z}_j) \Big].
\end{align*}

The corresponding volume form is
\[
\frac{\omega_H^n}{n!} = \frac{(2\sqrt{-1})^n}{(2\Re z_n - |z'|^2)^{n+1}}\, dz_1 \wedge d\bar{z}_1 \cdots dz_n \wedge d\bar{z}_n,
\]
which can also be obtained by pulling back the volume form from the ball model.

The distance between two points $x = (x', x_n)$ and $y = (y', y_n)$ in $\mathfrak{H}^n$ is given by
\begin{equation}
	\cosh \frac{d(x, y)}{2} = \frac{|x_n + \bar{y}_n - \langle x', y' \rangle|}{\sqrt{2\Re x_n - |x'|^2}\, \sqrt{2\Re y_n - |y'|^2}}.
\end{equation}

\begin{lemma}\label{lem-dist-line}
	Let $R_n=\{(0,t):t\in\R_{>0}\}$ be the real line (so $w'=0$, $w_n=t\in\R_{>0}$).
For a fixed $p=(z',z_n)$, the distance from $p$ to the real line $R_n$ is given by
	\[
	d(p,R_n) = 2\operatorname{arccosh}\!\left(\sqrt{\frac{\big(|z_n|+\Re z_n\big)}{2\Re z_n-\sum_{j=1}^{n-1}|z_j|^2}}\right)\]
\end{lemma}
\begin{proof}
	Let $z_n=x+iy$.
Since \begin{align*}
	\cosh^2\frac{d(p,(0,t))}{2} &=\frac{|z_n+t|^2}{2\big(2\Re z_n-\sum_{j}|z_j|^2\big)\,t}\\
	&=\frac{(x+t)^2+y^2}{2\big(2x-\sum_{j}|z_j|^2\big)\,t}=\frac{t^2+2tx+x^2+y^2}{2\big(2x-\sum_{j}|z_j|^2\big)\,t},
\end{align*}
we are minimizing
\[
\frac{t^2+2tx+x^2+y^2}{t}=2x+t+\frac{x^2+y^2}{t}.
\]
Clearly, the minimizer is $t=|z_n|$ and the minimal value is
\[
\frac{\big(|z_n|+\Re z_n\big)}{2\Re z_n-\sum_{j=1}^{n-1}|z_j|^2}.
\]
Hence the distance from $p$ to $R_n$ is
\[
d\big(p,R_n\big)
=2\operatorname{arccosh}\!\left(\sqrt{\frac{\big(|z_n|+\Re z_n\big)}
{2\Re z_n-\sum_{j=1}^{n-1}|z_j|^2}}\right).
\]
\end{proof}
The following proposition is straightforward but handy for our later arguments.
\begin{proposition}\label{prop-hn-normal}
	Let $N_{R_n}$ denote the normal bundle of $R_n$ in $\mathfrak{H}^n$. Then the exponential map $\exp_{R_n/\mathfrak{H}^n}:N_{R_n}\to \mathfrak{H}^n$ is a bijection.
\end{proposition}
\begin{proof}
	Let $p=(z',z_n)\in\mathfrak{H}^n$. From the proof of Lemma \ref{lem-dist-line}, we see that the distance function from $p$ to $(0,t)$ has exactly one critical point, which is a minimum. So there is exactly one geodesic from $p$ to ${R_n}$ that is perpendicular to $D_1$. This shows that $\exp_{{R_n}/\mathfrak{H}^n}$ is bijective.

\end{proof}

\subsection{Elements of $PU(n,1)$}
We consider $U(n-1)$ as a subgroup of $PU(n,1)$ acting on $\mathfrak{H}^n$ by
\[(z', z_n) \mapsto (Uz', z_n), \quad U \in U(n-1).\]
Or in the form of matrices,
\[U\mapsto \begin{bmatrix}U & 0 & 0 \\ 0 & 1 & 0 \\ 0 & 0 & 1\end{bmatrix} \in PU(n,1).\]
For each $t\in \R$, let $\iota_t \in PU(n,1)$ be defined by
\[\iota_t: (z', z_n) \mapsto (e^{-t}z', e^{-2t}z_n ).\]
Or in the form of matrices,
\[\iota_t = \begin{bmatrix}I_{n-1} & 0 & 0 \\ 0 & \cosh (t) & \sinh (t) \\ 0 & \sinh (t) & \cosh (t)\end{bmatrix} \in PU(n,1).\]
Let $\mathfrak{N} \subset \P \C^{n,1}$ be the Heisenberg group consisting of elements of the form
\[\begin{bmatrix}I_{n-1} & v & v \\ -v^* & 1-\frac12(\|v\|^2-is) & -\frac12(\|v\|^2-is) \\ v^* & \frac12(\|v\|^2-is) &  1+\frac12(\|v\|^2-is)\end{bmatrix} \in PU(n,1),\]
for $(v,s) \in \C^{n-1} \times \R$. For simplicity, we denote such an element by $N(v,s)$, or $(v,s)$ for simplicity. 
The multiplication of $\mathfrak{N}$ is given by \[(v,s) \cdot (v',s') = (v + v', s + s' + 2\Im \langle v, v' \rangle).\]

The point at infinity $p_\infty$ in this model corresponds to the point $[0,-1,1] \in \P \C^{n,1}$. The stabilizer of $p_\infty$ in $PU(n,1)$ is the subgroup $\mathfrak{P} = \mathfrak{N} \rtimes (U(n-1) \times \{\iota_t\}_{t\in \R})$.

A parabolic element in $\mathfrak{P}$ is contained in $\mathfrak{N} \rtimes U(n-1)$.

\begin{theorem}\label{thm-parabolic-conjugacy}
	Any parabolic element in $\mathfrak{P}$ is conjugate to one of the following two types:
	\begin{itemize}
		\item \emph{Type I}: an element of the form $ N(0,s)U$ with $U\in U(n-1)$ and $s \neq 0$.
		\item \emph{Type II}: an element of the form $N(v,0)U $ with \[U=\begin{bmatrix}
	U'& 0 \\
	0 & I_m
  \end{bmatrix},\quad U'\in U(n-m-1),\]for some $m\geq 1$, and $ v=(a_1,\cdots,a_{n-1}) \text{ with } a_j=0\text{ for }j\leq n-m-1$.
	\end{itemize}
\end{theorem}

\begin{proof}
	Write elements $N(v,s)U$ of the stabiliser \(\mathfrak{N} \rtimes U(n-1)\) as triples
\((U,v,s)\) with \(U\in U(n-1)\) and \((v,s)\in \mathfrak{N}\cong\C^{\,n-1}\times\R\),
and use the identity \[U N(v,s) U^{-1}=N(Uv,s),\]
we get the multiplication law
\[
(U,v,s)\cdot(U',v',s')=(UU',\;v+U v',\; s+s'+2\Im\langle v, U v'\rangle).
\]
So the inverse is \((U,v,s)^{-1}=(U^{-1},-U^{-1}v,-s)\).  Hence for
\(H=(R,w,t)\in \mathfrak{P}\) the conjugation action on \(G=(U,v,s)\) is
\[
\begin{aligned}
&HGH^{-1}=(R,w,t)(U,v,s)(R^{-1},-R^{-1}w,-t)\\
&=\Big(RUR^{-1},\; Rv+(I-RUR^{-1})w,\; s
+2\Im\langle w,Rv\rangle-2\Im\langle v,UR^{-1}w\rangle\\ &\quad -2\Im\langle w,RUR^{-1}w\rangle\Big).
\end{aligned}
\]
So we have the following two cases:
\begin{enumerate}
  \item 1 is not an eigenvalue of \(U\). Then we let $R=I_{n-1}$, and $t=0$. Then, one can solve \((I_{n-1}-U)w=-v\) and conjugate \(G\) by \(H\) to \((U,0,s')\). 
  \item  \(U\) has \(1\)-eigenspace of positive dimension. We can first choose \(R\in U(n-1)\) such that 
  \[RUR^{-1}=\begin{bmatrix}
	U'& 0 \\
	0 & I_m
  \end{bmatrix},\] where $U'\in U(n-m-1)$ for some $m\geq 1$, and 1 is not an eigenvalue of \(U'\). So we can simply assume that \(U\) is in this block diagonal form. 
 Write
  \(\C^{\,n-1}=\ker(I_{n-1}-U)\oplus\operatorname{im}(I_{n-1}-U)\).  By choosing \(w\) one may kill the
  \(\operatorname{im}(I_{n-1}-U)\)-component of \(v\); the residual is the projection
  of \(v\) to \(\ker(I_{n-1}-U)\). So, we may assume that \(v\) lies in \(\ker(I_{n-1}-U)\). If $v=0$, this is case (1). If $v\neq 0$, we let $R=I_{n-1}$, $t=0$ and $w=\lambda v$ for some $\lambda\in i\R$ to be determined. Then we can solve $s+2\Im\langle \lambda v,v\rangle-2\Im\langle v,\lambda v\rangle=0$ to get $\lambda$ so that \[(I_{n-1},\lambda v,0)(U,v,s)(I_{n-1},\lambda v,0)^{-1}=(U,v,0). \]
  \end{enumerate}
\end{proof}

The following proposition is easy to see.
\begin{proposition}\label{prop-correspondece}
	Let $(X, \omega)$ be a manifold equipped with a complete complex hyperbolic metric, whose Kähler form is $\omega$.
	Let $\Gamma\subset\Aut(\B^n)$ be a Fuchsian group corresponding to a covering map $\pi:\B^n\to X$. For a fixed point $p\in X$, let $\tilde{p}$ be a lift of $p$ in $\B^n$. 
	Then, there exists a one to one correspondece between the geodesic loops based at $p$ and the set $\Gamma\backslash \{1\}$.
	More precisely, \begin{itemize}
		\item For each $g\in \Gamma\backslash \{1\}$, there exists a unique geodesic loop $\gamma_{\tilde{p},g}\in \mathfrak{G}_p$ such that the lift of $\gamma_{\tilde{p},g}$ starting at $\tilde{p}$ ends at $g(\tilde{p})$. Furthermore, the length of $\gamma_{\tilde{p},g}$ is given by the distance from $\tilde{p}$ to $\gamma_{\tilde{p},g}(\tilde{p})$.
		\item Conversely, for each geodesic loop $\gamma\in \mathfrak{G}_p$, there exists a unique lift
		$\tilde{\gamma}$ of $\gamma$ starting at $\tilde{p}$. So there is a unique element $g\in \Gamma\backslash \{1\}$ such that $g(\tilde{p})$ is the other end point of $\tilde{\gamma}$.
	\end{itemize}

\end{proposition}

\section{Disk model}\label{sec-disk}
\subsection{Estimates for the Bergman kernel using geodesic balls}

On the unit disk $\D$, the standard hyperbolic metric with Ricci curvature $-1$ is given by
\[
\omega_0 = \frac{2\sqrt{-1}\,dz\wedge d\bar{z}}{(1 - |z|^2)^2}.
\]
The Hermitian metric on the trivial line bundle is
\[
h_0 = (1 - |z|^2)^2.
\]
Fix $R \in (0,1)$, and let $\Delta_R = \{z \in \D : |z| < R\}$.
Consider the Bergman space $\mathcal{H}_{R,k}$ for $(\Delta_R, h_0^k, \omega_0)$.
Define
\[
\tilde{I}(a, k) = \int_{\Delta_R} |z|^{2a} (1 - |z|^2)^{2k-2} \omega_0,
\]
so that $\left\{\frac{z^a}{\sqrt{\tilde{I}(a, k)}}\right\}_{a \geq 0}$ forms an orthonormal basis of $\mathcal{H}_{R,k}$.
In particular, the Bergman kernel at the origin is
\[
\rho_{R, k}(0) = \frac{1}{\tilde{I}(0, k)}.
\]

\begin{proposition}\label{prop-prep}
	If $(1 - R^2)^{2k-2} < \frac{1}{2}$, then
	\[
	\rho_{R, k}(0) < \frac{k - 1/2}{\pi}.
	\]
\end{proposition}
\begin{proof}
	We compute
	\begin{align*}
		\tilde{I}(0, k) = 4\pi \int_0^{R^2} (1 - x)^{2k-2} dx = \frac{4\pi}{2k - 1} \left(1 - (1 - R^2)^{2k-2}\right),
	\end{align*}
	so
	\[
	\rho_{R, k}(0) = \frac{k - 1/2}{2\pi} \left(1 - (1 - R^2)^{2k-2}\right)^{-1}.
	\]
	The result follows immediately.
\end{proof}

We use H\"ormander's $L^2$ estimate to construct global sections from local holomorphic functions. The following lemma is standard (see, for example, \cite{demailly1997complex}):

\begin{lemma}\label{lem-horm}
	Let $(X, \hat{\omega})$ be a complete K\"ahler manifold of complex dimension $n$, $\omega$ another K\"ahler form, and $E \to X$ a semi-positive vector bundle. Let $g \in L^2_{n,q}(X, E)$ satisfy $\dbar g = 0$ and
	\[
	\int_X \langle A_q^{-1} g, g \rangle\, dV < +\infty
	\]
	with respect to $\omega$, where $A_q$ denotes the operator $\sqrt{-1}\Theta(E)\wedge \Gamma$ (with $\Gamma$ the adjoint of wedging with $\omega$) in bidegree $(n, q)$ and $q \geq 1$.
	Then there exists $f \in L^2_{n, q-1}(X, E)$ such that $\dbar f = g$ and
	\[
	\|f\|^2 \leq \int_X \langle A_q^{-1} g, g \rangle\, dV.
	\]
\end{lemma}

\begin{theorem}\label{thm-disk-model} 
	Let $(X,\omega)$ be a hyperbolic Riemann surface. Let $(L,h)\to X$ be a holomorphic Hermitian line bundle with $i\Theta_h=\omega$. Let $p\in X$ have injectivity radius $r(p)$. Then for $k\geq 2$ and $(\cosh \frac{r(p)}{2})^{2k-2}>7\sqrt{k}$, we have
	\[
	\left|\rho_k(p)-\frac{1}{2\pi}(k-\frac12)\right| < 14k^{3/2}e^{-2(k-1)\log\cosh \frac{r(p)}{2}}.
	\]
\end{theorem}

\begin{proof}
	There exists an isometry $G:\Delta_R\to X$, where $R=\tanh^{-1}(\frac{r(p)}{2})$, such that $G(0)=p$. Composing with $z\mapsto \bar{z}$ if necessary, we may assume $G$ is holomorphic. Thus $G^*\omega=\omega_0$. We still denote by $h$ the pull-back of $h$ to $G^*L$. Let $e_L$ be a frame of $G^*L$. Then $|e_L|_h^2=e^{-\varphi}$ for some smooth function $\varphi$ with $\iddbar \varphi=\omega_0$. Thus $u=\varphi-\varphi_0$ is harmonic, so $u=\Re f$ for some holomorphic function $f$. The new frame $e_L'=e^f e_L$ then satisfies
	\[
	|e_L'|_h^2 = h_0.
	\]
	For simplicity, we continue to use $e_L'$ to denote $(G^{-1})^*e_L'$ on $X$.

	For any $0<R'<R$, let $\epsilon=R-R'>0$. Pick a smooth cut-off function $\chi(|z|)$ such that:
	\begin{itemize}
		\item $\chi(|z|)=1$ for $|z|\leq R'$;
		\item $\chi(|z|)=0$ for $|z|\geq R$;
		\item $-\frac{2}{\epsilon}<\chi'\leq 0$.
	\end{itemize}
	Then $\dbar\chi=\chi'(|z|)\frac{z\,d\bar{z}}{2|z|}$. So
	\[
	|\dbar\chi\otimes (e_L')^{\otimes k}|_{h}^2 = (\chi')^2 (1-|z|^2)^{2k+2} \frac{1}{8}.
	\]
	Let $\alpha' = \sqrt{\frac{1}{2\pi}(k-\frac12)}\, \dbar\chi \otimes (e_L')^{\otimes k}$. Then $\alpha'$ is a smooth section of $\Omega^{0,1}_X\otimes (kL)$ satisfying
	\begin{align*}
		\int_X |\alpha'|_{h}^2\omega
		&\leq \left(\frac{2}{\epsilon}\right)^2 \frac{1}{8\pi}(k-\frac12) \int_{\Delta_R\setminus \Delta_{R'}} (1-|z|^2)^{2k} \idzdz \\
		&\leq \frac{1}{\epsilon^2} (k-\frac12) (1-R'^2)^{2k} \frac{R^2-R'^2}{2} \\
		&\leq \frac{1}{\epsilon} (k-\frac12) (1-R'^2)^{2k}.
	\end{align*}

Using the notations in Lemma \ref{lem-horm}, set $q=1$, $g = \alpha' = \sqrt{\frac{1}{2\pi}(k-\frac12)} (dz)^{-1} \otimes (e_L')^{\otimes k} dz \wedge \dbar\chi$, and $E = kL - K_X$. Thus, $i\Theta_E = (k-1)\omega$, so $A_1^{-1}\alpha' = \frac{1}{k-1}\alpha'$. By Lemma \ref{lem-horm}, we can solve $\dbar v = \alpha'$ and obtain a solution $v \in C^{\infty}(X, kL)$ satisfying
\[
\int_X |v|^2_h \omega \leq \frac{1}{\epsilon} \frac{k-\frac12}{k-1} (1-R'^2)^{2k} \leq \frac{2}{\epsilon} (1-R'^2)^{2k}.
\]
Define
\[
\tilde{s}_p = \sqrt{\frac{1}{2\pi}(k-\frac12)}\, \chi(|z|) (e_L')^{\otimes k} - v,
\]
which is a holomorphic section of $kL$ on $X$.

Since
\[
1 > \left\| \sqrt{\frac{1}{2\pi}(k-\frac12)}\, \chi(|z|) (e_L')^{\otimes k} \right\|^2 > 1 - \frac{1}{2\pi}(k-\frac12) \int_{\D \setminus \Delta_{R'}} (1 - |z|^2)^{2k} \omega_0,
\]
and 
\[
\frac{1}{2\pi}(k-\frac12) \int_{\D \setminus \Delta_{R'}} (1 - |z|^2)^{2k} \omega_0< (1-R'^2)^{2k-1},
\]
we have,
\[
\left| \left\| \sqrt{\frac{1}{2\pi}(k-\frac12)}\, \chi(|z|) (e_L')^{\otimes k} \right\| - 1 \right| < (1-R'^2)^{2k-1}.
\]
Therefore,
\[
\left| \|\tilde{s}_p\| - 1 \right| < (1-R'^2)^{2k-1} + \|v\|.
\]

For any $s \in \hcal_k$ with $s(p) = 0$, by $S^1$-symmetry, we have $\langle \tilde{s}_p, s \rangle_X = \langle -v, s \rangle_X$. Decompose $\tilde{s}_p = a s_p + s$, where $s_p$ is a peak section at $p$ and $s(p) = 0$. By multiplying by a suitable $e^{i\theta}$, we may assume $a > 0$. Then, $-\langle v, s \rangle_X = \langle \tilde{s}_p, s \rangle_X = \|s\|^2$, so $\|s\| \leq \|v\|$.

Since $a^2 + \|s\|^2 = \|\tilde{s}_p\|^2$, we have
\begin{align*}
|a^2 - 1| &\leq |\|\tilde{s}_p\|^2 - 1| + \|s\|^2 \\
&\leq \frac{5}{2} |\|\tilde{s}_p\| - 1| + \|s\|^2 \\
&\leq 3\left( (1-R'^2)^{2k-1} + \|v\| \right) \\
&\leq 3\left( (1-R'^2)^{2k-1} + \sqrt{\frac{2}{\epsilon}} (1-R'^2)^k \right).
\end{align*}
Thus,
\[
|a - 1| < 3\left( (1-R'^2)^{2k-1} + \sqrt{\frac{2}{\epsilon}} (1-R'^2)^k \right).
\]

If $s_p = f_p (e_L')^{\otimes k}$, then
\[
a f_p(0) = \sqrt{\frac{1}{2\pi}(k-\frac12)} - \frac{v}{(e_L')^{\otimes k}}(0).
\]
By Proposition \ref{prop-prep}, when $(1-R'^2)^{2k-2} < \frac{1}{2}$, we have
\[
\left| \frac{v}{(e_L')^{\otimes k}}(0) \right| < \sqrt{\frac{k-\frac12}{\epsilon}} (1-R'^2)^k.
\]
Therefore,
\begin{align*}
|f_p(0) - \sqrt{\frac{1}{2\pi}(k-\frac12)}| &< \sqrt{\frac{k-\frac12}{2\pi}} |a^{-1} - 1| + 2\sqrt{\frac{k-\frac12}{\epsilon}} (1-R'^2)^k \\
&< 6\sqrt{\frac{k-\frac12}{2\pi}} \left( (1-R'^2)^{2k-1} + \sqrt{\frac{2}{\epsilon}} (1-R'^2)^k \right) \\
&= \left( 1+ \sqrt{\frac{2}{\epsilon}} \right) 6\sqrt{\frac{k-\frac12}{2\pi}} (1-R'^2)^k.
\end{align*}

Now, pick $\epsilon$ so that $\frac{1-R'^2}{1-R^2} = e^{1/k}$. Thus, $\epsilon = R - \sqrt{1 - e^{1/k} + e^{1/k} R^2}$, and $(1-R'^2)^{2k-2} = e^{\frac{2k-2}{k}} (1-R^2)^{2k-2} < \frac{1}{2}$. Since $1 - e^{1/k}(1-R^2) < 1 - (1 + \frac{1}{k})(1-R^2) = R^2 - \frac{1-R^2}{k}$,
\begin{align*}
\epsilon &> R\left(1 - \left(1 - \frac{1-R^2}{kR^2}\right)^{1/2}\right) \\
&> R\left(1 - \left(1 - \frac{1-R^2}{2kR^2}\right)\right) = \frac{1-R^2}{2kR}.
\end{align*}

Therefore,
\begin{align}
|f_p(0) - \sqrt{\frac{1}{2\pi}(k-\frac12)}| &< \left( 1+ \sqrt{\frac{4kR}{1-R^2}} \right) 6\sqrt{\frac{k-\frac12}{2\pi}} (1-R^2)^k \\
&< 7k(1-R^2)^{k-1}. \label{e-fp}
\end{align}

Therefore, when $(1-R^2)^{k-1} < (7\sqrt{k})^{-1}$,
\[
|f_p^2(0) - \frac{1}{2\pi}(k-\frac12)| < 14k^{3/2} e^{-2(k-1)\log\cosh \frac{r(p)}{2}}.
\]
This completes the proof of the theorem.
\end{proof}
Another application of the argument in the proof of this theorem is the following off-diagonal estimate:
\begin{theorem}\label{thm-disk-kk}
Let $(X,\omega)$ be a hyperbolic Riemann surface, and $(L,h)\to X$ a holomorphic Hermitian line bundle with $i\Theta_h=\omega$. Let $p\in X$ have injectivity radius $r(p)$, and let $q\in X$ with $d(p,q)<\frac{r(p)}{2}$. Then, when $k\geq 2$ satisfies $(\cosh \frac{r(p)}{4})^{2k-2}>16k^{1/2}$, we have
\[
\left|\,|K_k(q,p)|_h - \frac{1}{2\pi}(k-\tfrac12) \cosh^{-2k}\frac{d(p,q)}{2}\,\right| < 16k^{3/2}\cosh^{2-2k}\left(\frac{r(p)}{2}\right).
\]
\end{theorem}
\begin{proof}
We use the notations from the proof of Theorem~\ref{thm-disk-model}. By construction, $K_k(p,q) = s_p(q)\otimes \overline{s_p(p)}$, where $s_p$ is a peak section at $p$.

Recall that $s\in \hcal_k$ satisfies $s(p) = 0$ and $\tilde{s}_p = a s_p + s$.
By Theorem~\ref{thm-disk-model}, $|v(q)|_h$ and $|s(q)|_h$ are both bounded by $\sqrt{\frac{k-\frac12}{\epsilon}}(1-R'^2)^k$. Since $s_p = \frac{1}{a}(\tilde{s}_p - s) = \frac{1}{a}(\sqrt{\frac{k-1/2}{2\pi}} - v - s)$, we obtain
\[
\left|\,|s_p(q)|_h - \frac{1}{a}\sqrt{\frac{k-1/2}{2\pi}}(1-|z|^2)^k\,\right| < 4\sqrt{\frac{k-\frac12}{\epsilon}}(1-R'^2)^k < 6k(1-R^2)^{k-1/2}.
\]
Since $|a-1| < 4\frac{\sqrt{k}}{1-R^2}(1-R^2)^k$, it follows that
\[
\left|\,|s_p(q)|_h - \sqrt{\frac{k-1/2}{2\pi}}(1-|z|^2)^k\,\right| <  8k(1-R^2)^{k-1}.
\]
Similarly, by \eqref{e-fp},
\[
\left|\,|s_p(p)|_h - \sqrt{\frac{k-1/2}{2\pi}}\,\right| <  8k(1-R^2)^{k-1}.
\]
Using the equality $AB - CD = B(A-C) + C(B-D)$, we have
\[
\left|\,|s_p(p)|_h\,|s_p(q)|_h - \frac{k-1/2}{2\pi}(1-|z|^2)^k\,\right| < 16k^{3/2}(1-R^2)^{k-1}.
\]
Since $1-|z|^2 = \cosh^{-2}\frac{d(p,q)}{2}$, the result follows.
\end{proof}

\section{Hyperbolic cylinder model}\label{sec-hyper-cyl}

\subsection{Setting up}

The hyperbolic cylinder $\C_\eta^* \subset \C$ is defined as
\[
\C_\eta^* = \left\{ z \in \C \;\middle|\; -\frac{\pi}{2\eta} < \log|z| < \frac{\pi}{2\eta} \right\}
\]
equipped with the metric
\[
\omega_\eta = \frac{\eta^2 \sqrt{-1}\,dz \wedge d\bar{z}}{2 \cos^2(\eta t) |z|^2}
\]
where $t = \log|z|$. The Gaussian curvature of $\omega_\eta$ is $-1$. The central circle $\gamma = \{ |z| = 1 \}$ is a smooth geodesic loop of length $2\pi\eta$, and the injectivity radius at points on $\gamma$ is $\pi\eta$; at other points, it is strictly greater than $\pi\eta$.

The Hermitian line bundle under consideration is the trivial bundle, equipped with the Hermitian metric $h_\eta = \cos^2(\eta t)$.

Through each point $p \in \gamma$, there is a unique geodesic $\sigma_p(u)$ perpendicular to $\gamma$, parametrized by arc length $u$ with $\sigma_p(0) = p$ and $\frac{du}{dt} > 0$. A straightforward calculation shows that
\[
\cosh u = \frac{1}{\cos(\eta t)},
\]
so
\[
t = \frac{1}{\eta} \tan^{-1}(\sinh u), \qquad \frac{dt}{du} = \frac{1}{\eta \cosh u}.
\]

For each $a$, we have
\[
\int_{\C_\eta^*} |z|^{2a} h_\eta^{k+1} \omega_\eta = 2\pi \eta^2 \int_{-\pi/(2\eta)}^{\pi/(2\eta)} e^{2a t} \cos^{2k}(\eta t)\, dt.
\]
We define
\[
J(a, k+1, \eta) = \eta^2 \int_{-\pi/(2\eta)}^{\pi/(2\eta)} e^{2a t} \cos^{2k}(\eta t)\, dt.
\]

The Bergman kernel function at $\log|z|=t$ is then
\[
\rho_{k+1}(z) = \frac{\cos^{2k+2}(\eta t)}{2\pi} \sum_{a=-\infty}^{\infty} \frac{e^{a t}}{J(a, k+1, \eta)}.
\]
In particular, at $t = 0$ (i.e., $z = 1$),
\[
\rho_{k+1}(1) = \frac{1}{2\pi} \sum_{a=-\infty}^{\infty} \frac{1}{J(a, k+1, \eta)}.
\]

When $\eta = 1$, we have
\[
J(a, k+1, 1) = \int_{-\pi/2}^{\pi/2} e^{2a t} \cos^{2k} t\, dt.
\]
For simplicity, we will write $J_a = J(a, k+1, 1)$.

\begin{proposition}
For $a > 0$ and $k\geq 1$, we have
\begin{equation}
J_a > \frac{1}{4a}\left(e^{\frac{a}{\sqrt{k}}} - 1\right).
\end{equation}
\end{proposition}
\begin{proof}
When $|t| < \frac{1}{2\sqrt{k}}$, $\cos^{2k} t > (1 - \frac{1}{8k})^{2k} > \frac{1}{2}$. Thus, for $a > 0$,
\begin{align*}
J_a > \int_{0}^{\frac{1}{2\sqrt{k}}} e^{2a t} \cdot \frac{dt}{2} = \frac{1}{4a}\left(e^{\frac{a}{\sqrt{k}}} - 1\right).
\end{align*}
\end{proof}

Thus, $1/J_a$ decays exponentially. Since
\[
J(a, k+1, \eta) = \eta \int_{-\pi/2}^{\pi/2} e^{2a t/\eta} \cos^{2k} t\, dt = \eta J_{a/\eta},
\]
it follows that $1/J(a, k+1, \eta)$, as a function of $a$, also decays exponentially.

Applying the Poisson summation formula, we obtain
\[
\sum_{a=-\infty}^{\infty} \frac{1}{J(a, k+1, \eta)} = \sum_{\xi=-\infty}^{\infty} F(\xi, \eta),
\]
where
\[
F(\xi, \eta) = \int_{-\infty}^{\infty} e^{-2\pi i \xi a} \frac{da}{J(a, k+1, \eta)}.
\]
Since $J(a, k+1, \eta)$ is an even function of $a$, it follows that 
\[
F(\xi, \eta) = \int_{-\infty}^{\infty} \cos(2\pi \xi a) \frac{da}{J(a, k+1, \eta)}
\]
is an even function of $\xi$.

Moreover,
\begin{align*}
F(\xi, \eta) &= \int_{-\infty}^{\infty} e^{-2\pi i \eta \xi \frac{a}{\eta}} \frac{d(a/\eta)}{J_{a/\eta}} \\
&= \int_{-\infty}^{\infty} e^{-2\pi i \eta \xi a} \frac{da}{J_a} = F(\eta \xi, 1).
\end{align*}
For simplicity, we write $F(\xi) = F(\xi, 1)$. Therefore,
\begin{equation}
\sum_{a=-\infty}^{\infty} \frac{1}{J(a, k+1, \eta)} = \sum_{\xi=-\infty}^{\infty} F(\eta \xi).
\end{equation}

\begin{theorem}\label{thm-F0}
For $k\geq 1$, we have $F(0) = k + \frac{1}{2}$.
\end{theorem}

We now prove this theorem.

First, we establish the following estimate:

\begin{proposition}
For $\xi \neq 0$,
\begin{equation}
|F(\xi)| \leq \frac{3}{4\xi^2} F(0).
\end{equation}
\end{proposition}

\begin{proof}
Let $\dot{J}_a \triangleq \frac{\partial J_a}{\partial a} = \int_{-\pi/2}^{\pi/2} 2t\, e^{2a t} \cos^{2k} t\, dt$. Clearly, $|\frac{\dot{J}_a}{J_a}| < \pi$. Similarly, $|\frac{\ddot{J}_a}{J_a}| < \pi^2$. For $\xi > 0$, integration by parts shows
\begin{align*}
F(\xi) &= \frac{1}{2\pi\xi} \int_{-\infty}^{\infty} \sin(2\pi\xi a) \frac{\dot{J}_a}{J_a^2} da \\
&= \frac{1}{(2\pi\xi)^2} \int_{-\infty}^{\infty} \cos(2\pi\xi a) \left( \frac{\ddot{J}_a}{J_a^2} - 2\frac{\dot{J}_a^2}{J_a^3} \right) da.
\end{align*}
Therefore,
\[
|F(\xi)| \leq \frac{1}{(2\pi\xi)^2} \cdot 3\pi^2 \int_{-\infty}^{\infty} \frac{da}{J_a} = \frac{3}{4\xi^2} F(0).
\]
\end{proof}

As a direct corollary, $\lim_{\eta \to +\infty} \sum_{|\xi| \geq 1} |F(\eta \xi)| = 0$. Therefore,
\[
\lim_{\eta \to +\infty} \rho_{\eta, k+1}(1) = \frac{1}{2\pi} F(0).
\]
On the other hand, since the injectivity radius of $1 \in \C^*_\eta$ is $\pi\eta$, by Theorem~\ref{thm-disk-model},
\[
\lim_{\eta \to +\infty} \rho_{\eta, k+1}(1) = \frac{1}{2\pi} \left(k + \frac{1}{2}\right).
\]
Thus, we have proved Theorem~\ref{thm-F0}.

\subsection{Main result}
The main result of this section is the following theorem:
\begin{theorem}\label{thm-f-b}
	For the Fourier transform of $\frac{1}{J(a,k+1,1)}$, for $k\geq 1$, we have
	\[
	F(b) = \left(k+\frac{1}{2}\right) \cosh^{-2(k+1)}(\pi b) \quad \text{for any } b \text{ such that }|\Im b| < \frac12.
	\]
\end{theorem}

As a direct corollary, we can prove Theorem \ref{thm-jak}:

\begin{proof}[Proof of Theorem \ref{thm-jak}]
	By Theorem \ref{thm-f-b}, $e^{|\lambda\xi|}F(\xi)\in L^2(\R)$ for any $|\lambda|<k+1$. By the Paley-Wiener theorem, $\frac{1}{J(a,k+1,1)}$ extends holomorphically to the strip $\{z\in\C: |\Im a|<k+1\}$. Thus, $J(a,k+1,1)$ has no zeros for $|\Im a|<k+1$.
\end{proof}

The following two subsections are devoted to the proof of this main result.
\subsection{Exponential Decay} 
Denote $d(\eta) = 2\log(\cosh(\frac{\pi\eta}{2}))$.

\begin{theorem}\label{thm-exp-decay}
	When $k\geq 1 $ and $\eta>0$ satisfies $k\tanh(\pi\eta) \geq 1/2$ and \[(\cosh(\pi\eta))^{2k}>7\sqrt{k+1},\] we have
	\[
	|F(\eta)| < 7(k+1)^2\left(e^{-d(\eta)k} + 3e^{-d(2\eta)k}\right).
	\]
\end{theorem}

We divide the proof into two steps.

\subsubsection*{Step 1: Exponential Decay of $F(\xi)$}

\begin{proposition}\label{prop-exp-decay}
	For any $\delta \in (0,1)$, the function $e^{\delta|\xi|}F(\xi)$ is bounded for $\xi \in \mathbb{R}$.
\end{proposition}

To prove this, we recall the Paley-Wiener theorem (see, e.g., \cite{krantz2002primer}):

\begin{theorem}[Paley-Wiener]\label{thm-p-w}
	Let $f \in L^2(\mathbb{R})$. The following are equivalent:
	\begin{itemize}
		\item[(1)] There exists a function $F$ and constants $a, C > 0$ such that $F$ is holomorphic in the strip $\{z \in \mathbb{C} : |\Im z| < a\}$, $F(x+i0) = f(x)$ for all real $x$, and
		\[
		\int_{\mathbb{R}} |F(x+iy)|^2 dx \leq C, \quad \forall |y| < a.
		\]
		\item[(2)] The function $e^{a|\xi|}\hat{f}(\xi)$ lies in $L^2(\mathbb{R})$.
	\end{itemize}
\end{theorem}

For any $\delta \in (0,1)$, if $|b| \leq \delta$, then $\cos(bt) \geq \cos(\frac{\delta\pi}{2}) =: C_\delta > 0$. Thus,
\[
\Re J_{a+bi} = \int_{-\pi/2}^{\pi/2} \cos(bt) e^{a t} \cos^{2k} t\, dt \geq C_\delta J_a,
\]
so $|J_{a+bi}| \geq C_\delta J_a$. Therefore,
\[
\int_{\mathbb{R}} \left| \frac{1}{J_{a+ib}} \right|^2 da \leq \frac{1}{C_\delta^2} \int_{\mathbb{R}} \frac{1}{J_a^2} da, \quad \forall |b| \leq \delta.
\]
By the Paley-Wiener theorem, $e^{\delta|\xi|}F(\xi) \in L^2(\mathbb{R})$.

Similarly, since
\[
F'(\xi) = -2\pi i \int_{-\infty}^{\infty} e^{-2\pi i a \xi} \frac{a\, da}{J(a)} = -2\pi i\, \widehat{\left(\frac{a}{J(a)}\right)}(\xi),
\]
and
\[
\int_{\mathbb{R}} \left| \frac{a+bi}{J_{a+ib}} \right|^2 da \leq \frac{1}{C_\delta^2} \int_{\mathbb{R}} \frac{a^2 + \delta^2}{J_a^2} da, \quad \forall |b| \leq 1-\delta,
\]
we also have $e^{\delta|\xi|}F'(\xi) \in L^2(\mathbb{R})$.

It remains to show that $e^{\delta \xi} F(\xi)$ is bounded on $[0, +\infty)$. Let $H(\xi) = e^{\delta \xi} F(\xi)$ and $K(\xi) = e^{\delta \xi} F'(\xi)$. Then both $H(\xi)$ and $K(\xi)$ are in $L^2([0, +\infty))$, and
\[
H' - \delta H = K.
\]
Solving this ODE gives
\[
H(\xi) = e^{\delta \xi} \left( H(0) + \int_0^\xi e^{-\delta x} K(x) dx \right).
\]
Since $K(\xi) \in L^2([0, +\infty))$, the integral $\int_0^{+\infty} e^{-\delta x} |K(x)| dx$ converges. Let $A = \lim_{\xi \to \infty} \int_0^\xi e^{-\delta x} K(x) dx$. Since $H(\xi) \in L^2([0, +\infty))$, we must have $H(0) = -A$. Thus,
\[
H(\xi) = -e^{\delta \xi} \int_\xi^{\infty} e^{-\delta x} K(x) dx = -\int_0^{\infty} e^{-\delta x} K(x+\xi) dx.
\]
By Cauchy-Schwarz inequality,
\[
|H(\xi)| \leq \|K\|_{L^2([0, +\infty))} \left( \int_0^{\infty} e^{-2\delta x} dx \right)^{1/2}.
\]
Therefore, $e^{\delta|\xi|} F(\xi)$ is bounded, as claimed.

\
\subsubsection*{Step 2: Estimating the Remainder}

Let 
\[
S = \left\{ \lambda \in \mathbb{Z} \;\middle|\; \lambda > 1,\; p^2 \nmid \lambda\ \forall\ p > 1 \right\}
\]
be the set of square-free integers greater than $1$. For each positive $\lambda \in \mathbb{Z}^*$, let $\phi(\lambda)$ denote the number of distinct prime factors of $\lambda$. Clearly,
\begin{equation}\label{e-phi}
	\phi(\lambda) \leq \log_2 \lambda.
\end{equation}

Consider the alternating series
\begin{equation}\label{e-alt}
	\sum_{\lambda \in S} (-1)^{\phi(\lambda)-1} \sum_{n \in \mathbb{Z}^*} F(n\lambda\eta).
\end{equation}
By \eqref{e-phi}, each $F(m\eta)$ with integer $|m| > 1$ appears in \eqref{e-alt} at most $2^{\phi(m)} \leq m$ times. By Proposition~\ref{prop-exp-decay}, the series \eqref{e-alt} converges absolutely.

Recall that $\sum_{l=0}^n \binom{n}{l} (-1)^l = 0$ for $n \geq 1$, so $\sum_{l=1}^n \binom{n}{l} (-1)^{l-1} = 1$. Thus,
\[
\sum_{|\xi| \geq 2} F(\xi\eta) = \sum_{\lambda \in S} (-1)^{\phi(\lambda)-1} \sum_{n \in \mathbb{Z}^*} F(n\lambda\eta).
\]
By Theorems~\ref{thm-disk-model} and~\ref{thm-F0}, we have
\[
\left| \sum_{n \in \mathbb{Z}^*} F(n\lambda\eta) \right| \leq 14 (k+1)^2  \, e^{-d(\lambda\eta)k}.
\]
Therefore,
\begin{align*}
	|2F(\eta)| &\leq 14(k+1)^2  \left( e^{-d(\eta)k} + \sum_{\lambda \in S} e^{-d(\lambda\eta)k} \right) \\
	&\leq 14(k+1)^2  \left( e^{-d(\eta)k} + \sum_{\lambda \geq 2} e^{-d(\lambda\eta)k} \right).
\end{align*}

To estimate the sum, we use the following technical lemma (see \cite{sun2024apde}):

\begin{lemma}\label{lem-concave}
	Let $f(x)$ be a concave function. Suppose $f'(x_0) < 0$, then
	\[
	\int_{x_0}^\infty e^{f(x)} dx \leq \frac{e^{f(x_0)}}{-f'(x_0)}.
	\]
\end{lemma}

Since $d(x)$ is convex and $d'(x) = \pi \tanh \frac{x\pi}{2}$, we have
\begin{align*}
	\sum_{\lambda \geq 2} e^{-d(\lambda\eta)k} 
	&< e^{-d(2\eta)k} + \int_{2}^{\infty} e^{-d(\lambda\eta)k} d\lambda \\
	&< \left( 1 + \frac{1}{k\tanh(\pi\eta)} \right) e^{-d(2\eta)k} \\
	&\leq 3 e^{-d(2\eta)k}
\end{align*}
for $k\tanh(\pi\eta) \geq 1/2$.

Therefore, we have proved Theorem~\ref{thm-exp-decay}.

\
\subsection{Exact formula}

We now prove Theorem \ref{thm-f-b}.

Let
\[
s_1(z) = \frac{1}{2\pi\sqrt{\rho_{k+1}(1)}} \sum_{a=-\infty}^{\infty} \frac{z^a}{J(a,k+1,\eta)}.
\]
Then $s_1$ is a peak section at $z_0=1$, since $\|s_1\|_{h_\eta}=1$ and $|s_1(1)|_{h_\eta} = \sqrt{\rho_{k+1}(1)}$.

For $z = e^{-i\pi\theta}$,
\[
s_1(e^{-i\pi\theta}) = \frac{1}{2\pi\sqrt{\rho_{k+1}(1)}} \sum_{a=-\infty}^{\infty} \frac{e^{-i\pi \theta a}}{J(a,k+1,\eta)}.
\]
The Fourier transform of $\frac{e^{-i\pi \theta a}}{J(a,k+1,\eta)}$ is
\[
\int_{-\infty}^{\infty} \frac{e^{-i\pi \theta a - 2\pi \xi a}}{J(a,k+1,\eta)}\, da = F(\eta(\xi + \tfrac{\theta}{2})).
\]
By the Poisson summation formula,
\[
2\pi\sqrt{\rho_{k+1}(1)}\, s_1(e^{-i\pi\theta}) = \sum_{\xi=-\infty}^{\infty} F(\eta(\xi + \tfrac{\theta}{2})).
\]
For a fixed $b>0$, let $\theta = \frac{2b}{\eta}$. Then
\[
2\pi\sqrt{\rho_{k+1}(1)}\, s_1(e^{-i\pi\theta}) = F(b) + \sum_{|\xi|\geq 1} F(\eta\xi + b).
\]
By Theorem \ref{thm-exp-decay}, when $k\geq 1$, $\lim_{\eta\to+\infty} \sum_{|\xi|\geq 1} F(\eta\xi + b) = 0$, so
\[
\lim_{\eta\to+\infty} s_1(e^{-i\pi\theta}) = \frac{1}{\sqrt{(k+1/2)2\pi}} F(b).
\]

On the other hand, since $\cos t = 1$ at $z = e^{-i\pi\theta}$, the norm $|s_1(e^{-i\pi\theta})|_{h_\eta}$ equals the absolute value of $s_1(e^{-i\pi\theta})$. Since $d(1, e^{-i\pi\theta}) = \eta\pi\theta = 2b\pi$, by Theorem \ref{thm-disk-kk},
\[
\lim_{\eta\to+\infty} |s_1(e^{-i\pi\theta})| = \sqrt{\frac{k+1/2}{2\pi}}\,  \cosh^{-2(k+1)}(\pi b) .
\]
Therefore,
\[
|F(b)| = (k+1/2)\,  \cosh^{-2(k+1)}(\pi b) .
\]
So $F(\xi)$ does not vanish.
Since $F(0)>0$, we have
\[
F(b) = (k+1/2)\,  \cosh^{-2(k+1)}(\pi b) .
\]
This proves Theorem \ref{thm-f-b} for $b>0$.

The function $\log J_a$ is convex in $a$. To see this, let $d\mu = \cos^{2k} t\, dt$; then 
\[
\frac{\partial \log J_a}{\partial a} = \frac{\int_{-\pi/2}^{\pi/2} 2t\, e^{2a t} d\mu}{\int_{-\pi/2}^{\pi/2} e^{2a t} d\mu} = \frac{\int_{-\pi/2}^{\pi/2} 2t\, e^{2a t} d\mu}{J_a},
\]
and by the Cauchy-Schwarz inequality,
\begin{align*}
\frac{\partial^2 \log J_a}{\partial a^2}
= \frac{
    \int_{-\pi/2}^{\pi/2} e^{2a t} d\mu \int_{-\pi/2}^{\pi/2} (2t)^2 e^{2a t} d\mu
    - \left( \int_{-\pi/2}^{\pi/2} 2t\, e^{2a t} d\mu \right)^2
}{
    \left( \int_{-\pi/2}^{\pi/2} e^{2a t} d\mu \right)^2
} > 0
\end{align*}
Since $\lim_{a\to\infty}J_a=+\infty$, it is not hard to see that \[ \lim_{a\to+\infty} \frac{\int_{-\pi/2}^{\pi/2} 2t\, e^{2a t} d\mu}{\int_{-\pi/2}^{\pi/2} e^{2a t} d\mu} = \pi. \]
So \[ \lim_{a\to+\infty}\frac{\partial \log J_a}{\partial a} = \pi, \lim_{a\to-\infty}\frac{\partial \log J_a}{\partial a} =- \pi \]
Therefore, $F(b)=\int_{-\infty}^{\infty}e^{-2\pi b\sqrt{-1}}\frac{da}{J_a}$ is holomorphic for $|\Im b|<\frac12$, since $\dbar F(b) = 0$.
Theorem \ref{thm-f-b} follows from this.

The following proposition is a preparation for the proof of the main theorem in the case of complex hyperbolic cylinder.
\begin{proposition}\label{prop-f-twist} 
	Let \(t_0=\log|z_0|\) and \(u=\sinh^{-1}(\tan(\eta t_0))\). For \(k\geq 1\) one has
	\[
	\cos^{2k+2}(\eta t_0)\,
	F\!\bigl(\eta\xi + i\eta\tfrac{t_0}{\pi}\bigr)
	= \frac{\bigl(k+\tfrac12\bigr)\,
	e^{-i(k+1)\vartheta_\xi}}
	{\bigl(\cosh^2(\eta\pi\xi)\cosh^2 u-\sinh^2 u\bigr)^{\,k+1}}.
	\]
	Here the phase \(\vartheta_\xi\) is determined (up to \(2\pi\)) by
	\[
	\cos\vartheta_\xi
	=1-\frac{2\sinh^2(\eta\pi\xi)\sinh^2 u}
	{\cosh^2(\eta\pi\xi)+\sinh^2(\eta\pi\xi)\sinh^2 u},
	\qquad
	\operatorname{sign}\!\bigl(\sin\vartheta_\xi\bigr)=\operatorname{sign}(\xi u).
	\]

\end{proposition}
\begin{proof}
Since
\begin{align*}
\cosh^2(x+iy)
&= \cos^2 y\,\cosh^2 x - \sin^2 y\,\sinh^2 x + i\sin(2y)\sinh x\cosh x,
\end{align*}
Take \(x=\eta\pi\xi\) and \(y=\eta t_0\), using
\(\cos(\eta t_0)=\frac{1}{\cosh u}\) and \(\sin(\eta t_0)=\tanh u\), we obtain
\begin{align*}
\Re\!\bigl(\cosh^2(\eta\pi\xi + i\eta t_0)\bigr)
&= \frac{\cosh^2(\eta\pi\xi)-\sinh^2(\eta\pi\xi)\sinh^2 u}{\cosh^2 u},\\[4pt]
\bigl|\cosh^2(\eta\pi\xi + i\eta t_0)\bigr|
&= \frac{\cosh^2(\eta\pi\xi)+\sinh^2(\eta\pi\xi)\sinh^2 u}{\cosh^2 u}.
\end{align*}
Hence
\begin{align*}
\frac{\Re\cosh^2(\eta\pi\xi + i\eta t_0)}
{\bigl|\cosh^2(\eta\pi\xi + i\eta t_0)\bigr|}
&= \frac{\cosh^2(\eta\pi\xi)-\sinh^2(\eta\pi\xi)\sinh^2 u}
{\cosh^2(\eta\pi\xi)+\sinh^2(\eta\pi\xi)\sinh^2 u}\\
&= 1-\frac{2\sinh^2(\eta\pi\xi)\sinh^2 u}
{\cosh^2(\eta\pi\xi)+\sinh^2(\eta\pi\xi)\sinh^2 u}.
\end{align*}
Therefore there exists an angle \(\vartheta_\xi\) with
\[
\cos\vartheta_\xi
=1-\frac{2\sinh^2(\eta\pi\xi)\sinh^2 u}
{\cosh^2(\eta\pi\xi)+\sinh^2(\eta\pi\xi)\sinh^2 u},
\]
and we may write the polar form
\[
\cosh^2(\eta\pi\xi + i\eta t_0)
= \bigl(\cosh^2(\eta\pi\xi)-\tanh^2 u\bigr)\,e^{i\vartheta_\xi}.
\] The sign of $\sin \vartheta_\xi$ is the same as the sign of $\sinh(\eta\pi\xi)\sinh u$.
Consequently
\[
\frac{\cos^{2}(\eta t_0)}{\cosh^2(\eta\pi\xi + i\eta t_0)}
= \frac{e^{-i\vartheta_\xi}}
{\cosh^2(\eta\pi\xi)\cosh^2 u-\sinh^2 u}\,.
\]
Since \[
F(\eta\pi\xi+i\eta \tfrac{t_0}{\pi}) = \left(k + \tfrac{1}{2}\right) \left(\cosh(\eta\pi\xi+i\eta t_0)\right)^{-2(k+1)} ,
\] we get the result.
\end{proof}

\subsection{Holonomy along a geodesic loop}
Let $h$ be a Hermitian metric on the trivial bundle on $\C_\eta^*$ whose curvature is $-i\omega_\eta$. Let $\nabla$ be the Chern connection of $h$.
\begin{theorem}\label{thm-holonomy-cylinder}
	Let \(x\in\C_\eta^*\) and set \(t=\log|x|\), \(u=\operatorname{arcsinh}(\tan(\eta t))\).
	Give the central circle \(\gamma_0=\{|z|=1\}\) the counterclockwise orientation.
	Let \(\gamma:[0,\ell]\to\C_\eta^*\) be a geodesic loop based at \(x\) and denote its winding number by
	\[
	m=\frac{1}{2\pi i}\int_\gamma\frac{dz}{z}\in\mathbb{Z}.
	\]
	If the holonomy of the Chern connection \(\nabla\) along \(\gamma_0\) equals \(e^{2\pi i\alpha_0}\),
	then the holonomy along \(\gamma\) is
	\[
	\operatorname{Hol}_\gamma
	= e^{2\pi i\alpha_\gamma}
	= e^{ i\bigl(2\pi m\alpha_0+\vartheta_\gamma\bigr)},
	\]
	where the phase \(\vartheta_\gamma\in\mathbb{R}\) (well-defined modulo \(\mathbb{Z}\)) satisfies
	\[
	\cos\vartheta_\gamma
	=1-\frac{2\sinh^2(m\pi\eta)\sinh^2 u}
	{\cosh^2(m\pi\eta)+\sinh^2(m\pi\eta)\sinh^2 u},
	\qquad
	\operatorname{sign}\bigl(\sin\vartheta_\gamma\bigr)=\operatorname{sign}(m u).
	\]
	Moreover, the length \(\ell\) of \(\gamma\) is given by
	\begin{equation}\label{e-length-gamma}
		\cosh^2\frac{\ell}{2}
	=\cosh^2(m\pi\eta)\cosh^2 u-\sinh^2 u.
	\end{equation}
\end{theorem}
\begin{proof}
	$|u|$ is the distance from $x$ to the circle $\{|z|=1\}$. We can assume that $u\geq 0$.
	We can also assume that $m\geq 1$. 
Let $\pi:\D\to \C_\eta$ be a covering map. We can assume that the real line is mapped to $\gamma_0$. 
Let $\gamma_1$ be the geodesic segment from $x$ to a point in $\gamma_0$ that is orthogonal to $\gamma_0$. Then $\gamma_1$ is lifted to a geodesic segment $\tilde{\gamma}_1$ in $\D$ that is orthogonal to the real line. Let $q$ and $p$ be the two ends of $\tilde{\gamma}_1$. Let $\tilde{\gamma}$ be the lift of $\gamma$, that starts from $q$. Let $q'$ be the other end of $\tilde{\gamma}$. Let $\tilde{\gamma}'_1$ in $\D$ be the lift of $\gamma_1$ that has $q'$ as one endpoint. Let $p'$ be the other endpoint of $\tilde{\gamma}'_1$. Clearly, $p'$ is also in the real line.

By a hyperbolic isometry, we can assume that $0$ is the midpoint of the line segment $\overline{pq}$. Then Figure \ref{fig:quadrilateral} is an illutration of the quadrilateral $Q$ formed by $\tilde{\gamma}_1$, $\tilde{\gamma}$, $\tilde{\gamma}'_1$ and the real line segment $\overline{pp'}$. 

When $m=1$, by Stokes formula, we have \[\int_{\gamma_1}\partial \varphi- \int_{\gamma}\partial \varphi=i\text{area}(Q). \] By the Gauss-Bonnet theorem, the area of the right half of the quadrilateral $Q$ is \[\frac{\pi}{2}-\phi,\] where $\phi$ is the acute angle in the figure. 
So $\text{area}(Q)=\pi-2\phi$, $\cos\vartheta=-\cos(2\phi)$ and $\sin\vartheta=\sin(2\phi)>0$.
	\begin{figure}[ht]
    \centering
    \includegraphics[width=0.65\textwidth]{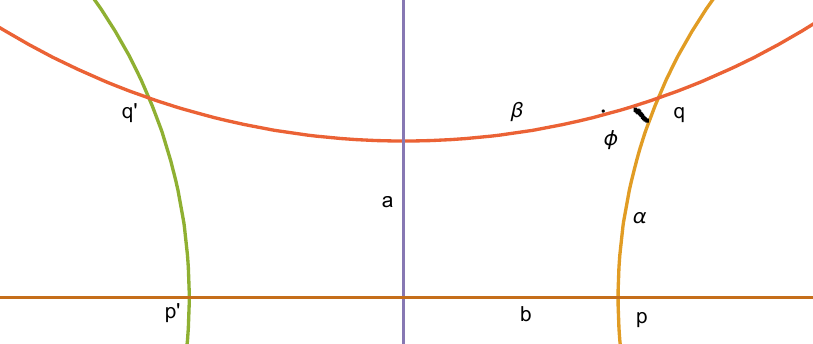}
    \caption{Quadrilateral used in the discussion of the twisted hyperbolic cylinder.}
    \label{fig:quadrilateral}
\end{figure}
\begin{align*}
\cos\phi &= \sinh a\,\sinh b
	\;=\; \tanh\alpha\,\tanh\beta,\\
\cosh a &= \tanh\beta\coth b.
\end{align*}

Hence
\[
\sinh^2 a\,\sinh^2 b = \tanh^2\alpha\,\tanh^2\beta.
\]

It follows that
\begin{align*}
\tanh^2\alpha\,\tanh^2\beta
&= \bigl(\tanh^2\beta\,\coth^2 b - 1\bigr)\sinh^2 b\\
&= \tanh^2\beta\,\cosh^2 b - \sinh^2 b.
\end{align*}

Therefore
\[
\tanh^2\beta=\frac{\sinh^2 b}{\cosh^2 b-\tanh^2\alpha},
\qquad
\cosh^2\beta=\cosh^2 b\,\cosh^2\alpha-\sinh^2\alpha.
\]

Consequently,
\begin{align*}
\cos 2\phi &= 2\cos^2\phi-1 = 2\tanh^2\alpha\,\tanh^2\beta -1\\
&= 2\tanh^2\alpha\frac{\sinh^2 b}{\cosh^2 b-\tanh^2\alpha}-1\\
&= \frac{2\sinh^2\alpha\,\sinh^2 b}{\cosh^2 b\,\cosh^2\alpha-\sinh^2\alpha}-1\\
&= \frac{2\sinh^2\alpha\,\sinh^2 b}{\sinh^2 b\,\sinh^2\alpha+\cosh^2 b}-1.
\end{align*}
Let $\vartheta_\gamma=\text{area}(Q)$, then $\cos(\vartheta_\gamma)=1- \frac{2\sinh^2\alpha\,\sinh^2 b}{\sinh^2 b\,\sinh^2\alpha+\cosh^2 b}$ and $ \sin(\vartheta_\gamma)=\sin(2\phi)>0$. 
Since $\alpha=u$ and $b=\eta\pi$, this proves the case \(m=1\).

When \(m>1\) we pull back \(\partial\varphi\) to the quadrilateral \(Q\). Then the diffence $\int_{\gamma_1}\partial \varphi- m\int_{\gamma}\partial \varphi$ is still equal to $i\text{area}(Q)$.  The same computation with \(\alpha=u\) and \(b=m\eta\pi\) yields the general case.

	\end{proof}

\section{Complex hyperbolic cylinder model}\label{sec-comp-hyper-cyl}
\subsection{Complex hyperbolic cylinder}
Given $\lambda > 1$ and $\vartheta=(\vartheta_1, \ldots, \vartheta_{n-1}) $ with $-\pi\leq \vartheta_j\leq\pi$, we define $A_{\lambda,\vartheta} \in \mathrm{Aut}(\mathfrak{H}^n)$ by
\[
A_{\lambda,\vartheta}(z', z_n) = (\lambda e^{i \vartheta}z', \lambda^2 z_n),
\]
where $e^{i \vartheta} z' = (e^{i \vartheta_1} z_1, \ldots, e^{i \vartheta_{n-1}} z_{n-1})$. 
Let $S_{\lambda,\vartheta} \subset \mathrm{Aut}(\mathfrak{H}^n)$ be the cyclic subgroup generated by $A_{\lambda,\vartheta}$. Consider the quotient space
\[
X_{\lambda,\vartheta} \triangleq \mathfrak{H}^n / S_{\lambda,\vartheta}.
\]
The action of $S_{\lambda,\vartheta}$ is free and properly discontinuous, so $X_{\lambda,\vartheta}$ is a complex manifold, and the quotient map \[Q: \mathfrak{H}^n \to X_{\lambda,\vartheta}\] is a covering map. Since the metric $\omega_H$ is invariant under $S_{\lambda,\vartheta}$, $X_{\lambda,\vartheta}$ inherits a Kähler form $\omega_{\lambda,\vartheta}$ such that $q^* \omega_{\lambda,\vartheta} = \omega_H$.

Following section~\ref{sec-hyper-cyl}, for a positive number $\eta$, let $\C_\eta^* \subset \C^*$ denote the annulus defined by $|\log |w_n|| < \frac{\pi}{2\eta}$. Set \begin{equation}
\eta = \frac{\log \lambda}{\pi}.
\end{equation}

Define a holomorphic map
\[
\Phi_{\lambda,\vartheta}: \mathfrak{H}^n \to \C^{n-1} \times \C_\eta^*
\]
by
\begin{equation}
	(z', z_n) \mapsto \left( z_1z_n^{-1/2-\frac{\vartheta_1 \sqrt{-1}}{2\log \lambda}},\cdots, z_{n-1}z_n^{-1/2-\frac{\vartheta_{n-1} \sqrt{-1}}{2\log \lambda}}, z_n^{-\frac{\pi \sqrt{-1}}{\log \lambda}} \right),
\end{equation}
where $z_n^{-\frac{\pi i}{\log \lambda}}$ is defined using the principal branch of the logarithm.
\begin{proposition}
	If $x,y\in \mathfrak{H}^n$, then $\Phi_{\lambda,\vartheta}(x) = \Phi_{\lambda,\vartheta}(y)$ if and only if $y = A_{\lambda,\vartheta}^k(x)$ for some $k \in \Z$.
\end{proposition}
\begin{proof}
	Suppose $y = A_{\lambda,\vartheta}(x)$. Then we have $y_n=\lambda^2 x_n $ and \[y_j=\lambda x_j e^{i \vartheta_j}, j<n.\] So $\frac{y_j}{\sqrt{y_n}}  = \frac{x_j}{\sqrt{x_n}},$ and $y_n^{\frac{\pi i}{\log \lambda}}=x_n^{\frac{\pi i}{\log \lambda}}$. And since $y_n^{-\frac{\vartheta_j \sqrt{-1}}{2\log \lambda}}=x_n^{-\frac{\vartheta_j \sqrt{-1}}{2\log \lambda}}e^{-\frac{\vartheta_j \sqrt{-1}}{2\log \lambda}\log \lambda^2}$ for $j<n$, we have $\Phi_{\lambda,\vartheta}(x) = \Phi_{\lambda,\vartheta}(y)$.

Conversely, suppose $\Phi_{\lambda,\vartheta}(x) = \Phi_{\lambda,\vartheta}(y)$. Then $y_n^{\frac{\pi i}{\log \lambda}}=x_n^{\frac{\pi i}{\log \lambda}}$, so $y_n = \lambda^{2k} x_n$ for some $k \in \Z$. Moreover,
	\[
	y_jy_n^{-1/2-\frac{\vartheta_j \sqrt{-1}}{2\log \lambda}} = x_jx_n^{-1/2-\frac{\vartheta_j \sqrt{-1}}{2\log \lambda}}, \quad j < n,
	\]implies \[y_j=\lambda^k x_j e^{k i \vartheta_j}, j<n.\] Thus, $y = A_{\lambda,\vartheta}^k(x)$.
\end{proof}
So $\Phi_{\lambda,\vartheta}$ descends to a holomorphic injective map $\iota_{\lambda,\vartheta}: X_{\lambda,\vartheta} \to \C^{n-1} \times \C_\eta^*$.

Let $(w_1, \ldots, w_{n-1}, w_n)$ be the coordinates on $\C^{n-1} \times \C_\eta^*$. Then
we have the following proposition.
\begin{proposition}\label{prop-biholog}
	$\Phi_{\lambda,\vartheta}$ is a local biholomorphism.
\end{proposition}
\begin{proof}
	We only need to check the Jacobian of $\Phi_{\lambda,\vartheta}$. It is straightforward to get that
	\[
 \frac{dw_n}{w_n}=-\frac{\pi i}{\log \lambda} \frac{dz_n}{z_n},
\]
and for $j < n$,
\begin{align*}
\frac{dw_j}{w_j} =  \frac{dz_j}{z_j}-(\frac12+ \frac{\vartheta_j i}{2\log \lambda} )\frac{dz_n}{z_n}.
\end{align*} 

Therefore, \begin{align*}
	dw_1 \wedge \cdots \wedge dw_n &= -(\prod_{j=1}^{n}\frac{w_j}{z_j})\frac{\pi i}{\log \lambda} dz_1 \wedge \cdots \wedge dz_n\\
	&= -\frac{\pi i}{\log \lambda}z_n^{-(n+1)/2-\frac{\sum_{j=1}^{n-1}\vartheta_j i}{2\log \lambda}-\frac{\pi i}{\log \lambda}} dz_1 \wedge \cdots \wedge dz_n.
\end{align*}
Thus, $\Phi_{\lambda,\vartheta}$ is a local biholomorphism.
\end{proof}
Therefore, $\iota_{\lambda,\vartheta}$ is a biholomorphism from $X_{\lambda,\vartheta}$ onto its image $\Phi_{\lambda,\vartheta}(\mathfrak{H}^n)$. In the following, we identify $X_{\lambda,\vartheta}$ with $\Phi_{\lambda,\vartheta}(\mathfrak{H}^n) \subset \C^{n-1} \times \C_\eta^*$.

\subsection{Setting-up on $X_{\lambda,\vartheta}$}
Write $z_n = |z_n| e^{i\theta_n}$, so $t_n \triangleq \log |w_n| = \frac{\pi}{\log \lambda} \theta_n$ and $\Re z_n = |z_n| \cos \theta_n$. Therefore,
\[
\Re z_n = |z_n| \cos\left(\frac{\log \lambda}{\pi} t_n\right).
\]

Since $|w_j|^2 = |z_j|^2 |z_n|^{-1} e^{\frac{\vartheta_j}{\log \lambda} \theta_n}$, we have
\begin{align}
	2\Re z_n - |z'|^2 &= |z_n| \left( 2\cos\left(\frac{\log \lambda}{\pi} t_n\right) - \sum_{j=1}^{n-1} |w_j|^2 e^{-\frac{\vartheta_j}{\log \lambda} \theta_n} \right) \\
	&= |z_n| \left( 2\cos\left(\frac{\log \lambda}{\pi} t_n\right) - \sum_{j=1}^{n-1} |w_j|^2 e^{-\frac{\vartheta_j}{\pi} t_n} \right).
\end{align}
Define \begin{equation}
	\kappa = 2\cos\left(\frac{\log \lambda}{\pi} t_n\right) - \sum_{j=1}^{n-1} |w_j|^2 e^{-\frac{\vartheta_j}{\pi} t_n}.
\end{equation}

Then
\[
\omega_{\lambda,\vartheta} = -2\sqrt{-1} \partial \bar{\partial} \log \kappa,
\]
\begin{proposition}\label{prop-X-s-psh}
	$-\log \kappa$ is a strictly plurisubharmonic exhaustion function of $X_{\lambda,\vartheta}$. Therefore, $X_{\lambda,\vartheta}$ is a Stein manifold.
\end{proposition}
\begin{proof}
	We only need to show that for each $0<c<1$, the set \[E_c=\{w\in X_{\lambda,\vartheta}\big| \kappa\geq c \} \] is compact. 

	Recall that $R_n=\{(0,t):t\in\R_{>0}\}$ is a real ine in $\mathfrak{H}^n$. Let $\Sigma_1$ and $\Sigma_{\lambda^2}$ be the images of the normal bundles of $R_n$ at $t=1$ and $t=\lambda^2$, respectively, under the exponential map. Let $F$ be the region bounded by $\Sigma_1$ and $\Sigma_{\lambda^2}$. Then, by Proposition \ref{prop-hn-normal}, the restriction of $\Phi_{\lambda,\vartheta}$ to $F$ is surjective onto $X_{\lambda,\vartheta}$.  Let $F_c=\Phi_{\lambda,\vartheta}^{-1}(E_c)\cap F$. Then \[F_c=\{z\in F\big|2\Re z_n - |z'|^2\geq c|z_n| \}. \] Since $\Re z_n\leq \lambda^2$ on $F$, we immediately get that $\Im z_n$ is bounded. We claim that $\exists \delta>0$ such that $\Re z_n\geq \delta$ for $z\in F_c$. Otherwise $|z|$ can be arbitrarily small, and would not be contained in $F$. Therefore, $F_c$ is compact. And so $E_c$ is compact.
\end{proof}

 $X_{\lambda,\vartheta}$ is clearly homotopic  to $S^1$, so every holomorphic line bundle over $X_{\lambda,\vartheta}$ is trivial. Since the map $\mathfrak{H}^1\to\C^*_\eta$ is surjective, it is easy to check that $X_{\lambda,\vartheta}$ is the open subset of $\C^{n-1} \times \C_\eta^*$ defined by $\kappa > 0$.

Denote by $|dz|^2=i^n dz_1\wedge d\bar{z}_1\wedge \cdots \wedge dz_n\wedge d\bar{z}_n$. Then the computation in the proof of Proposition \ref{prop-biholog} shows that \[|dz|^2=(\frac{\log \lambda}{\pi})^2|z_n|^{n+1}e^{(2\pi-\sum_{j=1}^{n-1} \vartheta_j)\frac{2t_n}{\pi}}|dw|^2. \]
So the volume form is 
\begin{equation}\label{e-omega-lambda-n}
	\frac{\omega_{\lambda,\vartheta}^n}{n!}=\frac{2^{n}\eta ^2}{\kappa^{n+1}|w_n|^2}e^{-\sum_{j=1}^{n-1} \vartheta_j\frac{t_n}{\pi}} |dw|^2.
\end{equation}
\begin{definition}
	We call the circle $\beta_1=\{(0,w_n)\in \C^{n-1}\times \C_\eta^*| |w_n|=1\}$, oriented counterclockwise, the central geodesic of $X_{\lambda,\vartheta}$.
\end{definition}

Let $h_{\lambda,\vartheta}=(\frac{\kappa}{2})^2$. Let $\hcal_{\text{cyl},k}$ be the Bergman space of holomorphic functions $f$ on $X_{\lambda,\vartheta}$ such that  \[\|f\|^2_{\text{cyl},k}\eqd\int_{X_{\lambda,\vartheta}}|f|^2h_{\lambda,\vartheta}^k\frac{\omega_{\lambda,\vartheta}^n}{n!}<\infty. \]
Let $\rho_{\text{cyl},k}$ be the Bergman kernel function of $\hcal_{\text{cyl},k}$. More generally, for a fixed number $\alpha\in [0,1)$, let $\mathfrak{m}_k=k\alpha-[k\alpha]$ be the decimal part. We can consider the Bergman space $\hcal_{\text{cyl},\alpha,k}$ consisting of holomorphic functions $f$ on $X_{\lambda,\vartheta}$ such that  \[\|f\|^2_{\text{cyl},\alpha,k}\eqd\int_{X_{\lambda,\vartheta}}|f|^2h_{\lambda,\vartheta}^k|w_n|^{-2\mathfrak{m}_k}\frac{\omega_{\lambda,\vartheta}^n}{n!}<\infty. \]
We call the Bergman kernel function $\rho_{\text{cyl},\alpha,k}$ of $\hcal_{\text{cyl},\alpha,k}$ the twisted Bergman kernel function. We should remark here that the effect of twisting $h_{\lambda,\vartheta}^k$ with $|w_n|^{-2\mathfrak{m}_k}$ is to change the holonomy around $\beta_1$ to be $e^{2\pi k\alpha i}$.
\begin{lemma}\label{lem-twist-metric}
	Let $h_{\alpha}=e^{-\phi}$ be a smooth function on $X_{\lambda,\vartheta}$, considered as a metric on the trivial line bundle, such that $\iddbar \phi=\omega_{\lambda,\vartheta}$ and the holonomy around $\beta_1$ is $e^{2\pi \alpha i}$. Then there exists a holomorphic function $g$ on $X_{\lambda,\vartheta}$ such that \[|g|^2h_{\alpha}=h_{\lambda,\vartheta}|w_n|^{-2\alpha}.\]
\end{lemma}
\begin{proof}
	By twisting both sides with $|w_n|^{2\alpha}$, we can assume that $\alpha=0$. So there exists a pluriharmonic function $u$ on $X_{\lambda,\vartheta}$ such that $\phi=-\log h_{\lambda,\vartheta}+u$. The holonomy condition implies that $\int_{\beta_1} \partial u=0$. Thus $u(0,\cdot)$ is the real part of a holomorphic function $f$ on $\C_\eta^*$. Since $Y_{w_n}\eqd (\C^{n-1}\times \{w_n\})\cap X_{\lambda,\vartheta}$ is a ellipsoid for each fixed $w_n$, we can extend $f$ to a holomorphic function $F_{w_n}$ on $Y_{w_n}$, such that $\Re F_{w_n}=u$. So we get a function $G(\cdot,w_n)\eqd F_{w_n}(\cdot)$ on $X_{\lambda,\vartheta}$. $G$ is clearly continuous. On the other hand, in a simply connected neighborhood $U$ of any point $(0,w_n)$, there is a unique holomorphic function $\tilde{G}$ such that $\Re \tilde{G}=u$ and $\tilde{G}\big|_{U\cap( \{0\}\times \C^*_\eta)}=f$. Since the conjugate pluriharmonic function $v$ of $u$ is unique up to an imaginary constant, we have $G=\tilde{G}$ on $U$, hence holomorphic on $U$. The same argument can be repeated to show that $G$ is holomorphic on $X_{\lambda,\vartheta}$.
\end{proof}
This lemma implies that the Bergman kernel function for $(X_{\lambda,\vartheta},h_{\alpha}^k,\frac{\omega_{\lambda,\vartheta}^n}{n!})$ is the same as $(X_{\lambda,\vartheta},h_{\lambda,\vartheta}^k|w_n|^{-2\mathfrak{m}_k},\frac{\omega_{\lambda,\vartheta}^n}{n!})$.

The main result of this section and Section~\ref{sec-holonomy} is the following theorem.
\begin{theorem}\label{thm-cyl-model-main}
	Theorem \ref{thm-main} holds for $(X_{\lambda,\vartheta},\omega_{\lambda,\vartheta})$, with $k_0=n+1$.
\end{theorem}

\subsection{Computation of the Bergman kernel function} To avoid complication of notations in the process of our computations, we first show the computations for the case when $\alpha=0$, namely the Hermitian metric for the line bundle being $h_{\lambda,\vartheta}^k$. It is then easy to get the general case from there.

Let
\[
\Z^n_H \eqd \{a=(a_1,\dots,a_n)\mid a_j\in\Z_{\ge 0}\ (1\le j\le n-1),\ a_n\in\Z\}.
\]
For \(w\in X_{\lambda,\vartheta}\) and \(a\in\Z^n_H\) write the multi-index power
\[
w^a \eqd \prod_{j=1}^n w_j^{a_j}.
\]
Since the measure \(h_{\lambda,\vartheta}^{k}\,\omega_{\lambda,\vartheta}^n\) is invariant under the standard \((S^1)^n\)-action, the monomials \(w^a\) are orthogonal:
\[
\langle w^a,w^b\rangle_{L^2(h_{\lambda,\vartheta}^k)}=0\qquad\text{whenever }a\neq b.
\]

For \(a=(a',a_n)\) with \(a'=(a_1,\dots,a_{n-1})\) we set
\[
I(a)\eqd \int_{X_{\lambda,\vartheta}} |w^a|^2 \,h_{\lambda,\vartheta}^{k+1}\,\frac{\omega_{\lambda,\vartheta}^n}{n!}.
\]

Using \(h_{\lambda,\vartheta}=(\kappa/2)^2\) and the expression for \(\omega_{\lambda,\vartheta}^n/n!\) obtained above, this integral can be written as
\begin{multline*}
	I(a',a_n)
= \int_{X_{\lambda,\vartheta}}
|w'^{a'}|^2|w_n^{a_n}|^2
\Big(\cos\!\Big(\frac{\log\lambda}{\pi}t_n\Big)
-\tfrac{1}{2}\sum_{j=1}^{n-1}|w_j|^2 e^{-\frac{\vartheta_j}{\pi}t_n}\Big)^{2k-n+1} \\ \times
\frac{\eta^2}{2|w_n|^2}\,e^{-\sum_{j=1}^{n-1} \vartheta_j\frac{t_n}{\pi}}\,|dw|^2.
\end{multline*}

The integral can be rearranged as
\begin{multline*}
\frac{\eta^2}{2}\int_{\C^*_\eta}
\Bigg[
\int_{\substack{\sum_{j=1}^{n-1}|w_j|^2 e^{-\frac{\vartheta_j}{\pi}t_n}<2\cos\!\big(\eta t_n\big)}}
|w'^{a'}|^2\left(1-\frac{\sum_{j=1}^{n-1}|w_j|^2 e^{-\frac{\vartheta_j}{\pi}t_n}}
{2\cos\!\big(\eta t_n\big)}\right)^{2k-n+1}\\
\times
e^{-\sum_{j=1}^{n-1}\vartheta_j\frac{t_n}{\pi}}\,|dw'|^2
\Bigg]
\cos^{2k-n+1}\!\bigg(\eta t_n\bigg)\;
|w_n|^{2a_n}\;\frac{|dw_n|^2}{|w_n|^2}.
\end{multline*}
Denote $|a'|=\sum_{j=1}^{n-1}a_j$.
Let $y_j=(2\cos\!\big(\eta t_n\big))^{-1/2}e^{-\frac{\vartheta_j}{2\pi}t_n}w_j $, then the inner integral becomes \[(2\cos\!\big(\eta t_n\big))^{n-1+|a'|}e^{\frac{t_n}{\pi}\sum_{j=1}^{n-1}a_j\vartheta_j}\int_{\sum_{j=1}^{n-1}|y_j|^2<1}\prod_{j=1}^{n-1}|y_j|^{2a_j}
(1-\sum_{j=1}^{n-1}|y_j|^2)^{2k-n+1}\,|dy'|^2.\] 
We denote\[G_{2k+1}(a')=\int_{y'\in \B_{n-1}}\prod_{j=1}^{n-1}|y_j|^{2a_j}(1-\sum_{j=1}^{n-1}|y_j|^2)^{2k-n+1}|dy'|^2 \]
which, when $a'=\mathbf{0}$, is equal to
\begin{align*}(\cos\!\big(\eta t_n\big))^{n-1+|a'|}\frac{(4\pi)^{n-1}(2k-n+1)!}{(2k)!}.
\end{align*}

Therefore,
\begin{eqnarray*}
\frac{I(a',a_n)}{2^{n-1+|a'|}}&=&\frac{\eta^2}{2}G_{2k+1}(a')\int_{\C^*_\eta}\cos^{2k+|a'|}\!\big(\eta t_n\big)\;|w_n|^{2a_n+\sum_{j=1}^{n-1}a_j\frac{\vartheta_j}{\pi}}\;\frac{|dw_n|^2}{|w_n|^2}\\
&=&2\pi G_{2k+1}(a') J(a_n+\sum_{j=1}^{n-1}a_j\frac{\vartheta_j}{2\pi}, k+\frac{|a'|}{2}+1,\eta)\\
\end{eqnarray*}

Consider the Bergman space  \[(\B^{n-1},(1-|z'|^2)^{2k+1},\frac{2^{n-1}}{(1-|z'|^2)^n}|dz'|^2 ).\]
On one hand, we know that the Bergman kernel equals $\frac{(2k)!}{(4\pi)^{n-1}(2k+1-n)!}$. On the other hand, this can be computed by using $\{(z')^{a'}\}_{a'\in\Z_{\geq 0}^{n-1}}$ as orthogonal basis:
\[(1-|z'|^2)^{2k+1}\sum_{a'\in\Z_{\geq 0}^{n-1}}\frac{|(z')^{a'}|^2}{2^{n-1}G_{2k+1}(a')}  \]

By using the Taylor expansion of $(1-|z'|^2)^{-2k-1}$, we get
\[\frac{1}{G_{2k+1}(a')} =\frac{(2k)!}{(2\pi)^{n-1}(2k+1-n)!}\binom{2k+|a'|}{|a'|} \frac{|a'|!}{a'!}, \]
where
  $a'!=\prod_{j=1}^{n-1}(a_j!)$.

So the Bergman kernel function of $(X_{\lambda,\vartheta},h_{\lambda,\vartheta}^k,\frac{\omega_{\lambda,\vartheta}^n}{n!})$ at $w=(w',w_n)$ is given by
\begin{eqnarray*}
\rho_{\text{cyl},k+1}(w)&=&(\frac{\kappa}{2})^{2k+2}\sum_{a'\in\Z_{\geq 0}^{n-1}}2^{1-n-|a'|}|w'^{a'}|^2\sum_{a_n\in\Z}\frac{|w_n|^{2a_n}}{I(a',a_n)}\\
=(\frac{\kappa}{2})^{2k+2}&\sum_{a'\in\Z_{\geq 0}^{n-1}}&\frac{2^{1-n-|a'|}|w'^{a'}|^2}{2\pi G_{2k+1}(a')}\sum_{a_n\in\Z}\frac{|w_n|^{2a_n}}{J(a_n+\sum_{j=1}^{n-1}a_j\frac{\vartheta_j}{2\pi}, k+\frac{|a'|}{2}+1,\eta)}.
\end{eqnarray*}


By Theorem \ref{thm-f-b} \[F_{k+1, \eta}(\xi)=F_{k+1, 1}(\eta\xi)=(k+\frac12)\cosh^{-2(k+1)}(\pi\eta\xi ). \]

So \begin{eqnarray*}
	&&\sum_{a_n\in\Z}\frac{|w_n|^{2a_n}}{J(a_n+\sum_{j=1}^{n-1}a_j\frac{\vartheta_j}{2\pi}, k+\frac{|a'|}{2}+1,\eta)}\\
	&=&\sum_{\xi\in \Z}e^{i\xi \sum_{j=1}^{n-1}a_j\vartheta_j}|w_n|^{\sum_{j=1}^{n-1}a_j\frac{\vartheta_j}{\pi}}F_{k+\frac{|a'|}{2}+1,\eta}(\xi+\frac{t_n}{\pi}i)\\
	&=&\sum_{\xi\in \Z}e^{i\xi \sum_{j=1}^{n-1}a_j\vartheta_j}|w_n|^{-\sum_{j=1}^{n-1}a_j\frac{\vartheta_j}{\pi}}(k+\frac{|a'|+1}{2})\cosh^{-2(k+1+\frac{|a'|}{2})}(\pi\eta\xi+\eta t_n i ) 
\end{eqnarray*}
Since \[(k+\frac{|a'|+1}{2})\binom{2k+|a'|}{|a'|} =(k+\frac12)\binom{2k+1+|a'|}{|a'|} , \]
we get that 
\begin{multline*}
	\sum_{a'\in\Z_{\geq 0}^{n-1}}\frac{e^{i\xi \sum_{j=1}^{n-1}a_j\vartheta_j}}{|w_n|^{\sum_{j=1}^{n-1}a_j\frac{\vartheta_j}{\pi}}}\frac{2^{1-n-|a'|}|w'^{a'}|^2}{ G_{2k+1}(a')}(k+\frac{|a'|+1}{2})\cosh^{-|a'|}(\pi\eta\xi+\eta t_n i ) \\
	=(k+\frac12)\frac{(2k)!}{ (4\pi)^{n-1}(2k+1-n)!}(1-\frac{\sum_{j=1}^{n-1}e^{i\xi \vartheta_j}|w_j|^2|w_n|^{-\frac{\vartheta_j}{\pi}}}{2\cosh(\pi\eta\xi+\eta t_n i )})^{-2k-2}
\end{multline*}

Therefore, the Bergman kernel function of $(X_{\lambda,\vartheta},h_{\lambda,\vartheta}^k,\frac{\omega_{\lambda,\vartheta}^n}{n!})$ at $w=(w',w_n)$ is given by
\begin{eqnarray*}
\rho_{\text{cyl},k+1}(w)&=&(\frac{\kappa}{2})^{2k+2}(k+\frac12)\frac{(2k)!}{ 2\pi(4\pi)^{n-1}(2k+1-n)!}\\
&\times&\sum_{\xi\in \Z}(\cosh(\pi\eta\xi+\eta t_n i )-\frac12\sum_{j=1}^{n-1}e^{i\xi \vartheta_j}|w_j|^2|w_n|^{-\frac{\vartheta_j}{\pi}})^{-2k-2}.
\end{eqnarray*}
When $\xi=0$, \[(\frac{\kappa}{2})^{2k+2}(\cosh(\eta t_n i )-\frac12\sum_{j=1}^{n-1}|w_j|^2|w_n|^{-\frac{\vartheta_j}{\pi}})^{-2k-2}=1.\] 
So the main term of the Bergman kernel function is \[(k+\frac12)\frac{(2k)!}{ 2\pi(4\pi)^{n-1}(2k+1-n)!}=\frac{1}{(4\pi)^n} \frac{(2k+1)!}{(2k-n+1)!}.\]

For general $\alpha\in [0,1)$,  we set
\[
I_\alpha(a)\eqd \int_{X_{\lambda,\vartheta}} |w^a|^2 \,h_{\lambda,\vartheta}^{k+1}|w_n|^{-2\mathfrak{m}_{k+1}}\,\frac{\omega_{\lambda,\vartheta}^n}{n!}.
\]
Then, clearly, $I_\alpha(a',a_n)=I(a',a_n-\mathfrak{m}_{k+1})$. Thus, $\rho_{\text{cyl},\alpha,k+1}(w)$ equals
\begin{eqnarray*}
&=&(\frac{\kappa}{2})^{2k+2}|w_n|^{-2\mathfrak{m}_{k+1}}\sum_{a'\in\Z_{\geq 0}^{n-1}}2^{1-n-|a'|}|w'^{a'}|^2\sum_{a_n\in\Z}\frac{|w_n|^{2a_n}}{I(a',a_n-\mathfrak{m}_{k+1})}\\
&=&(\frac{\kappa}{2})^{2k+2}\sum_{a'\in\Z_{\geq 0}^{n-1}}\frac{2^{1-n-|a'|}|w'^{a'}|^2}{2\pi G_{2k+1}(a')}\sum_{a_n\in\Z}\frac{|w_n|^{2(a_n-\mathfrak{m}_{k+1})}}{J(a_n-\mathfrak{m}_{k+1}+\sum_{j=1}^{n-1}a_j\frac{\vartheta_j}{2\pi}, k+\frac{|a'|}{2}+1,\eta)}.
\end{eqnarray*}
Consequently, we get \begin{multline}\label{e-cyl-twist-bergman}
	\rho_{\text{cyl},\alpha,k+1}(w)=(\frac{\kappa}{2})^{2k+2}(k+\frac12)\frac{(2k)!}{ 2\pi(4\pi)^{n-1}(2k+1-n)!}\\
\times\sum_{\xi\in \Z}(\cosh(\pi\eta\xi+\eta t_n i )-\frac12\sum_{j=1}^{n-1}e^{i\xi \vartheta_j}|w_j|^2|w_n|^{-\frac{\vartheta_j}{\pi}})^{-2k-2}e^{-2\pi\xi\mathfrak{m}_{k+1} i }.
\end{multline}

\section{Computation of the holonomy}\label{sec-holonomy}
\subsection{Formulas on the ball model}

	Let $z=(z_1,z_2,\dots,z_n)$ be the coordinates of $\B^n$. Write $|z|^2=\sum_{i=1}^n |z_i|^2$.
	Let $\varphi_0=-2\log (1-|z|^2)$. Then, \[\partial \varphi_0 =\frac{2\bar{z}\cdot dz}{1-|z|^2}.\]
\begin{lemma}\label{lem-ball-holonomy}
	Let $x,y\in \B^n$. Let $[x,y]$ be the oriented geodesic from $x$ to $y$. Then
	\[\int_{[x,y]} \partial \varphi_0=\log \frac{1-|x|^2}{1-|y|^2}+2\sqrt{-1}\arg(1-x\cdot \bar{y}).\]
\end{lemma}
\begin{proof}
	Recall that the identification map from $\B^n$ to $\CH^n$ is given by $z\mapsto [z,1]$. We denote by $A:\CH^n\to \B^n$ the inverse map. Let $\tilde{x}=\sqrt{\frac{1}{1-|x|^2}}[x,1]$, $\tilde{y}=\sqrt{\frac{1}{1-|y|^2}}[y,1]$. Then we have $\langle \tilde{x},\tilde{x}\rangle=-1$, $\langle \tilde{y},\tilde{y}\rangle=-1$ and
	 \[\langle \tilde{x},\tilde{y}\rangle=\frac{x\cdot \bar{y}-1}{(1-|x|^2)(1-|y|^2)}.\]
	 Let $\phi=\arg(1-x\cdot \bar{y})$. Then \[\langle \tilde{x},e^{i\phi}\tilde{y}\rangle=\frac{(x\cdot \bar{y}-1)e^{-i\phi}}{(1-|x|^2)(1-|y|^2)}<0.\]
	 Let $P=\tilde{x}$, $Q=e^{i\phi}\tilde{y}$, then the real line segment $[P,Q]$, given by $X(t)=(\frac12-t)P+(\frac12+t)Q$, $t\in [-\frac12,\frac12]$, is mapped to the geodesic $[x,y]$, with $t$ not the length parameter. Write $P=[P_1,\cdots,P_{n+1}]$, $Q=[Q_1,\cdots,Q_{n+1}]$ and $X(t)=[X_1,\cdots,X_{n+1}]$. Then \[AX(t)=(\frac{X_1(t)}{X_{n+1}(t)},\cdots,\frac{X_n(t)}{X_{n+1}(t)}).\]
	So $A^*(dz_i)=\frac{dX_i(t)}{X_{n+1}(t)}-\frac{X_idX_{n+1}(t)}{X^2_{n+1}(t)}$ and \[A^*(\bar{z}_idz_i)=\frac{\overline{X_i(t)}dX_i(t)}{|X_{n+1}(t)|^2}-\frac{|X_{i}(t)|^2dX_{n+1}(t)}{|X_{n+1}(t)|^2X_{n+1}(t)}.\]
	Thus,
	\begin{eqnarray*}
		A^*(\partial \varphi_0 )&=&\frac{2|X_{n+1}(t)|^2}{|X_{n+1}(t)|^2-\sum_{j=1}^{n}|X_{i}(t)|^2}\sum_{j=1}^{n}\left(\frac{\overline{X_i(t)}dX_i(t)}{|X_{n+1}(t)|^2}-\frac{|X_{i}(t)|^2dX_{n+1}(t)}{|X_{n+1}(t)|^2X_{n+1}(t)}\right)\\
		&=&\frac{2}{|X_{n+1}(t)|^2-\sum_{j=1}^{n}|X_{i}(t)|^2}\sum_{j=1}^{n}\left(\overline{X_i(t)}dX_i(t)-\frac{|X_{i}(t)|^2dX_{n+1}(t)}{X_{n+1}(t)}\right)\\
		&=&2\frac{dX_{n+1}(t)}{X_{n+1}(t)}+\frac{2\left(\sum_{j=1}^{n}\overline{X_j(t)}dX_j(t)-\overline{X_{n+1}(t)}dX_{n+1}(t)\right)}{|X_{n+1}(t)|^2-\sum_{j=1}^{n}|X_{i}(t)|^2}.
	\end{eqnarray*}
	
	We have\begin{eqnarray*}
		|X_{n+1}(t)|^2-\sum_{j=1}^{n}|X_{i}(t)|^2&=&\langle (\frac12-t)P+(\frac12+t)Q,(\frac12-t)P+(\frac12+t)Q \rangle\\ &=&-(\frac12+2t^2)+(\frac12-2t^2)\langle P,Q \rangle,
	\end{eqnarray*}
	and 
	\[\overline{X_j(t)}dX_j(t)=(\frac{1}{2}(\bar{P}_j+\bar{Q}_j)+t(\bar{Q}_j-\bar{P}_j))(Q_j-P_j)dt.\]
	So we can calculate the integral
	\begin{multline*}
		\int_{[P,Q]}A^*(\partial \varphi_0 )=\int_{-1/2}^{1/2}2\frac{dX_{n+1}(t)}{X_{n+1}(t)}\\ +\int_{-1/2}^{1/2}\frac{2(\sum_{j=1}^{n}\overline{X_j(t)}dX_j(t)-\overline{X_{n+1}(t)}dX_{n+1}(t))}{|X_{n+1}(t)|^2-\sum_{j=1}^{n}|X_{i}(t)|^2}.
	\end{multline*}
	
	By symmetry, the second integral equals \begin{eqnarray*}
		\int_{-1/2}^{1/2}\frac{(\sum_{j=1}^{n}((\bar{P}_j+\bar{Q}_j)(Q_j-P_j)-(\bar{P}_{n+1}+\bar{Q}_{n+1})(Q_{n+1}-P_{n+1})))dt}{|X_{n+1}(t)|^2-\sum_{j=1}^{n}|X_{i}(t)|^2}.
	\end{eqnarray*}
	The numerator of the integrand equals \[\langle Q,P\rangle-\langle P,Q\rangle =2\sqrt{-1}\Im(\langle Q,P\rangle)=0.\]
	Thus \[\int_{[P,Q]}A^*(\partial \varphi_0 )=\int_{-1/2}^{1/2}2\frac{dX_{n+1}(t)}{X_{n+1}(t)}=2\log \frac{Q_{n+1}}{P_{n+1}}=2\phi\sqrt{-1}+\log \frac{1-|x|^2}{1-|y|^2}.\]
\end{proof}

\subsection{Holonomy computation}
Let $T$ be a point on the central geodesic $\beta_1=\{w\in X_{\lambda,\vartheta}\big| w'=0,|w_n|=1 \}$. For each oriented loop $\gamma\subset \C^{n-1}\times \C^*$, we define the winding number of $\gamma$ as
\begin{equation}
	\text{wind}(\gamma)=\frac{1}{2\pi i}\int_{\gamma}\frac{1}{w_n}dw_n.
\end{equation}

Let $\beta_m$ be the geodesic loop based at $T$ with winding number $m$ (just winding along $\beta_1$, counterclockwise, $m$ times). Clearly, the holonomy of $\nabla$ along $\beta_m$ does not depend on the choice of $T$. So we denote it by $\text{Hol}_{\nabla}(m\beta_1)$.
\begin{theorem}\label{thm-holonomy-computation}
	Let $p\in X_{\lambda,\vartheta}$, $m\in\Z^*$. Let $\gamma_m$ be the geodesic loop based at $p$ with winding number $m$. Then the holonomy of $\nabla$ along $\gamma_m$ satisfies:
	\begin{equation}
		\arg \text{Hol}_{\nabla}(\gamma_m)-\arg \text{Hol}_{\nabla}(m\beta_1)=-2\arg\bigg(\cosh\!\big(\pi\eta m + \mathrm{i}\,\eta t_n\big)
	- \displaystyle\sum_{j=1}^{n-1}\frac{e^{\mathrm{i}m\vartheta_j}\,|w_j|^2}
	{2\,|w_n|^{\vartheta_j/\pi}}\bigg)
	\end{equation}
\end{theorem}
\begin{proof}

Let $q\in \{w'=0\}\subset X_{\lambda,\vartheta}$ satisfy the following condition:
\[d(p,q)=\min_{q'\in \{w\in X_{\lambda,\vartheta}\big| w'=0\}} d(p,q'). \]
By Lemma \ref{lem-dist-line}, such $q$ is unique.
Let $[p,q]$ be the geodesic segment from $p$ to $q$ with length $d(p,q)$.
We can assume that $m\geq 1$. Let $\tilde{p}$ be a lift of $p$ to $\mathfrak{H}_n$.
Then $[p,q]$ is lifted to a geodesic segment from $\tilde{p}$ to a point $\tilde{q}\in \{z\in \mathfrak{H}^n\big| z'=0\}$. And $\gamma_m$ lifts to the geodesic $\tilde{\gamma}$ from $\tilde{p}$ to $A_{\lambda,\vartheta}^m(\tilde{p})$. $[p,q]$ is also lifted to a geodesic segment from $A_{\lambda,\vartheta}^m(\tilde{p})$ to a point $\tilde{q}'\in \{z\in \mathfrak{H}^n\big| z'=0\}$, which is just $A_{\lambda,\vartheta}^m(\tilde{q})$. 

The image of the geodesic segment $[\tilde{q},\tilde{q}']$ under $\Phi_{\lambda,\vartheta}$ is then a geodesic loop $\sigma_m$ based at $q$ with winding number $m$. Since the restriction of $\nabla$ to the copy of $\C^*_\eta=\{ w\in X_{\lambda,\vartheta}\big| w'=0\}$ is just the Chern connection for $(\C^*_\eta,\cos^2(\eta t_n))$. Then by Proposition \ref{prop-f-twist} and Theorem \ref{thm-holonomy-cylinder}, we have \[\arg \text{Hol}_{\nabla}(\sigma_m)-\arg \text{Hol}_{\nabla}(m\beta_1)=-2\arg(\cosh\!\big(\pi\eta m + \mathrm{i}\,\eta t_n\big)).\]
So, we only need to show that \[\arg \text{Hol}_{\nabla}(\gamma_m)-\arg \text{Hol}_{\nabla}(\sigma_m)=-2\arg\bigg(1
	- \displaystyle\sum_{j=1}^{n-1}\frac{e^{\mathrm{i} m \vartheta_j}\,|w_j|^2|w_n|^{-\vartheta_j/\pi}}
	{2\,\cosh\!\big(\pi\eta m  + \mathrm{i}\,\eta t_n\big)}\bigg) \]


Let $\tilde{q}=(0,z_n)$, then by Proposition \ref{prop-dist-line}, the $n$-th coordinate $\tilde{p}$ is also $z_n$. So $\tilde{p}=(z',z_n)$.

Clearly, $\rho_{k+1}(w)$ is $(S^1)^n$-symmetric, so we can assume that the image of $\tilde{q}$ in $(\B^n,x)$, under the Cayley transform $\ccal_{HB}$, satisfies the condition that $\ccal_{HB}(\tilde{q})$ and $\ccal_{HB}(\tilde{q}')$ are symmetric about the imaginery axis of $x_n$, namely $\ccal_{HB}(\tilde{q})=(0,\alpha)$ and $\ccal_{HB}(\tilde{q}')=(0,-\bar{\alpha})$. Let $\alpha=u+\sqrt{-1}v$. Then, since $\ccal_{HB}$ maps the origin to $(0,1)$ and the infinity to $(0,-1)$, it is clear that $u>0$. From the equations \[z_n=\frac{1-\alpha}{2(1+\alpha)}, \quad \frac{1-2\lambda^{2m}z_n}{1+2\lambda^{2m}z_n}=-u+\sqrt{-1}v, \]
we get that \begin{equation}
	\lambda^{2m}=\frac{(1+u)^2+v^2}{(1-u)^2+v^2}, \quad \Re(z_n)=\frac{1-u^2-v^2}{2((1+u)^2+v^2)}, \quad \Im(z_n)=\frac{-v}{(1+u)^2+v^2}.
\end{equation}

Since $|w_j|^2=\frac{|z_j|^2}{|z_n|}e^{\frac{\vartheta_j}{\pi}t_n}$, for $j\leq n-1$, we have
\begin{eqnarray*}
	\frac{\sum_{j=1}^{n-1}e^{i m  \vartheta_j}|w_j|^2}{2|w_n|^{\frac{\vartheta_j}{\pi}}\cosh(\pi\eta m +\eta t_n i )}&=& \frac{\sum_{j=1}^{n-1}e^{i m  \vartheta_j}|z_j|^2}{2|z_n|\cosh(\pi\eta m +\eta t_n i )}\\
	&=& \frac{\sum_{j=1}^{n-1}e^{i m  \vartheta_j}|z_j|^2}{2(\cosh( m \log\lambda)\Re(z_n)+i \sinh( m \log\lambda)\Im(z_n) )}\\
	&=& \frac{\lambda^ m \sum_{j=1}^{n-1}e^{i m  \vartheta_j}|z_j|^2}{(1+\lambda^{2 m })\Re(z_n)+i(\lambda^{2 m }-1)\Im(z_n) }.
\end{eqnarray*}
So \begin{eqnarray*}
	&&\frac{\lambda^{m}e^{i m \vartheta_j}|z_j|^2}{(1+\lambda^{2m})\Re(z_n)+i(\lambda^{2m}-1)\Im(z_n) }\\
	&=&\frac{2|1+\alpha|^2|1-\alpha|^2\lambda^{m}e^{i m \vartheta_j}|z_j|^2}{(|1+\alpha|^2+|1-\alpha|^2)(1-|\alpha|^2)+((|1+\alpha|^2-|1-\alpha|^2))(\bar{\alpha}-\alpha)}\\
&=&\frac{|1+\alpha|^2|1-\alpha|^2}{1-|\alpha|^4+(\bar{\alpha}^2-\alpha^2)}\lambda^{m}e^{i m \vartheta_j}|z_j|^2=\frac{1-\bar{\alpha}^2}{1+\bar{\alpha}^2}\lambda^{m}e^{i m \vartheta_j}|z_j|^2.
\end{eqnarray*}
In summary, we have 
\begin{equation}\label{e-alpha-bar}
	\displaystyle\sum_{j=1}^{n-1}\frac{e^{\mathrm{i} m \vartheta_j}\,|w_j|^2|w_n|^{-\vartheta_j/\pi}}
	{2\,\cosh\!\big(\pi\eta m  + \mathrm{i}\,\eta t_n\big)}=\frac{1-\bar{\alpha}^2}{1+\bar{\alpha}^2}\lambda^{m}e^{i m \vartheta_j}|z_j|^2.
\end{equation}

 Since $L$ is trivial, for simplicity, we also use $\nabla$ to denote the corresponding $(1,0)$-form of $\nabla$. So we have \[i\big(\arg \text{Hol}_{\nabla}(\gamma_m)-\arg \text{Hol}_{\nabla}(\sigma_m)\big) =\int_{[\tilde p, A_{\lambda,\vartheta}^m(\tilde p)]}\Phi_{\lambda,\vartheta}^*\nabla - \int_{[\tilde q, A_{\lambda,\vartheta}^m(\tilde q)]}\Phi_{\lambda,\vartheta}^*\nabla.\]
By Stokes formula, we have \begin{eqnarray*}
	\int_{[\tilde p, A_{\lambda,\vartheta}^m(\tilde p)]}\Phi_{\lambda,\vartheta}^*\nabla &-& \int_{[\tilde q, A_{\lambda,\vartheta}^m(\tilde q)]}\Phi_{\lambda,\vartheta}^*\nabla+\int_{[A_{\lambda,\vartheta}^m(\tilde p), A_{\lambda,\vartheta}^m(\tilde q)]}\Phi_{\lambda,\vartheta}^*\nabla-\int_{[\tilde p, \tilde q]}\Phi_{\lambda,\vartheta}^*\nabla\\= i\int_{\Sigma}\omega=&&\\
\int_{[\tilde p, A_{\lambda,\vartheta}^m(\tilde p)]}\partial \phi_0& -& \int_{[\tilde q, A_{\lambda,\vartheta}^m(\tilde q)]}\partial \phi_0+\int_{[A_{\lambda,\vartheta}^m(\tilde p), A_{\lambda,\vartheta}^m(\tilde q)]}\partial \phi_0-\int_{[\tilde p, \tilde q]}\partial \phi_0.	
\end{eqnarray*}
Since both $[\tilde p, \tilde q]$ and $[A_{\lambda,\vartheta}^m(\tilde p), A_{\lambda,\vartheta}^m(\tilde q)]$ are lifts of the geodesic segment $[p,q]$, we have \[\int_{[A_{\lambda,\vartheta}^m(\tilde p), A_{\lambda,\vartheta}^m(\tilde q)]}\Phi_{\lambda,\vartheta}^*\nabla-\int_{[\tilde p, \tilde q]}\Phi_{\lambda,\vartheta}^*\nabla=0. \]
And, by Lemma \ref{lem-ball-holonomy}, it is straightforward to check that \[\int_{[A_{\lambda,\vartheta}^m(\tilde p), A_{\lambda,\vartheta}^m(\tilde q)]}\partial \phi_0-\int_{[\tilde p, \tilde q]}\partial \phi_0=0.\]

Thus,
\begin{equation}\label{e-hol-partial-phi}
	i\big(\arg \text{Hol}_{\nabla}(\gamma_m)-\arg \text{Hol}_{\nabla}(\sigma_m)\big)=\int_{[\tilde p, A_{\lambda,\vartheta}^m(\tilde p)]}\partial \phi_0 - \int_{[\tilde q, A_{\lambda,\vartheta}^m(\tilde q)]}\partial \phi_0.
\end{equation}

By Lemma \ref{lem-ball-holonomy} again, we have
\[
\int_{\mathcal{C}_{HB}\big([\tilde p, A_{\lambda,\vartheta}^m(\tilde p)]\big)} \partial\varphi_0
=2i\,\arg\!\Bigg(
1
-\frac{4\lambda^{m}\sum_{j=1}^{n-1} e^{-i m\vartheta_j}\,|z_j|^2}
{(1+2z_n)(1+2\lambda^{2m}\overline{z}_n)}
-\frac{(1-2z_n)(1-2\lambda^{2m}\overline{z}_n)}
{(1+2z_n)(1+2\lambda^{2m}\overline{z}_n)}
\Bigg).
\]
Since $\frac{2}{1+2z_n}=1+\alpha$ and $\frac{2}{1+2\lambda^{2m}\overline{z}_n}=1-\alpha$, we have
\[-\frac{4\lambda^{m}\sum_{j=1}^{n-1} e^{-i m\vartheta_j}\,|z_j|^2}
{(1+2z_n)(1+2\lambda^{2m}\overline{z}_n)}=(1-\alpha^2)\lambda^{m}\sum_{j=1}^{n-1} e^{-i m\vartheta_j}\,|z_j|^2, \]
and $\frac{(1-2z_n)(1-2\lambda^{2m}\overline{z}_n)}{(1+2z_n)(1+2\lambda^{2m}\overline{z}_n)}=-\alpha^2$.
Thus \begin{equation}
	\int_{\mathcal{C}_{HB}\big([\tilde p, A_{\lambda,\vartheta}^m(\tilde p)]\big)} \partial\varphi_0
=2i\,\arg\!\Bigg(
1
-(1-\alpha^2)\lambda^{m}\sum_{j=1}^{n-1} e^{-i m\vartheta_j}\,|z_j|^2
+\alpha^2
\Bigg).
\end{equation}
Similarly 
\begin{equation}
	\int_{\mathcal{C}_{HB}\big([\tilde q, A_{\lambda,\vartheta}^m(\tilde q)]\big)} \partial\varphi_0
=2i\,\arg\!\Big(
1+\alpha^2
\Big).
\end{equation}
Thus \begin{equation}
	\int_{\mathcal{C}_{HB}\big([\tilde p, A_{\lambda,\vartheta}^m(\tilde p)]\big)} \partial\varphi_0-\int_{\mathcal{C}_{HB}\big([\tilde q, \tilde q']\big)} \partial\varphi_0=2i\,\arg\!\Bigg(1
-\frac{1-\alpha^2}{1+\alpha^2}\lambda^{m}\sum_{j=1}^{n-1} e^{-i m\vartheta_j}\,|z_j|^2
\Bigg).
\end{equation}
Since the right hand side of \eqref{e-alpha-bar} is the conjugate of $\frac{1-\alpha^2}{1+\alpha^2}\lambda^{m}\sum_{j=1}^{n-1} e^{-i m\vartheta_j}\,|z_j|^2$, by \eqref{e-hol-partial-phi},
we have proved the Theorem.
\end{proof}

Let $\ell_m$ denote the length of the geodesic segment $[\tilde{p},A_{\lambda,\vartheta}^m(\tilde{p})]$. Let \[\tilde{p}=(z_1,\cdots,z_n),\] then \[A_{\lambda,\vartheta}^m(\tilde{p})=(\lambda^m e^{im\vartheta_1} z_1,\cdots,\lambda^m e^{im\vartheta_{n-1}} z_{n-1},\lambda^{2m} z_n). \] 
So we have \begin{equation}\label{eqn-length-gamma-m}
	\cosh\frac{\ell_m}{2}=\frac{|\lambda^{2m}\bar{z}_n+ z_n-\lambda^m \sum_{j=1}^{n-1}e^{-im\vartheta_j}|z_j|^2|}{\lambda^m(2\Re z_n-\sum_{j=1}^{n-1}|z_j|^2)}.
\end{equation}
\begin{proposition}\label{prop-length-gamma-m}
	\[
	\bigg(\cosh\frac{\ell_m}{2}\bigg)^{-1}
	= \frac{\kappa}{2\left|\,
	\cosh\!\big(\pi\eta m - \mathrm{i}\,\eta t_n\big)
	- \displaystyle\sum_{j=1}^{n-1}\frac{e^{-\mathrm{i}m\vartheta_j}\,|w_j|^2}
	{2\,|w_n|^{\vartheta_j/\pi}}
	\right|}.
	\]
\end{proposition}
\begin{proof}
	The right hand side equals 
	\begin{eqnarray*}
		&&\frac{2|z_n|\cos\left(\frac{\log \lambda}{\pi} t_n\right) - \sum_{j=1}^{n-1} |z_n||z_j|^2 e^{-\frac{\vartheta_j}{\pi} t_n}}{\big|2|z_n|\cosh(\pi\eta m -\eta t_n i )-\frac{\sum_{j=1}^{n-1}e^{-i m  \vartheta_j}|z_n||w_j|^2}{|w_n|^{\frac{\vartheta_j}{\pi}}}\big|}\\
		&=&\frac{2\Re(z_n) - \sum_{j=1}^{n-1} |z_j|^2 }{\big|2\Re(z_n)\cosh(\pi\eta m )-i2\Im(z_n)\sinh(\pi\eta m )-\sum_{j=1}^{n-1}e^{-i m  \vartheta_j}|z_j|^2\big|} \\
		&=&\frac{\lambda^m(2\Re(z_n) - \sum_{j=1}^{n-1} |z_j|^2) }{\big||\Re(z_n)(1+\lambda^{2m})-i\Im(z_n)(\lambda^{2m}-1)-\lambda^m\sum_{j=1}^{n-1}e^{-i m  \vartheta_j}|z_j|^2\big|}\\
		&=&\frac{\lambda^m(2\Re(z_n) - \sum_{j=1}^{n-1} |z_j|^2) }{\big||z_n+\lambda^{2m}\bar{z_n}-\lambda^m\sum_{j=1}^{n-1}e^{-im \vartheta_j}|z_j|^2\big|}.
	\end{eqnarray*}
So the conclusion follows from \eqref{eqn-length-gamma-m}.
\end{proof}

\begin{proof}[Proof of Theorem \ref{thm-cyl-model-main}]
	By Lemma \ref{lem-twist-metric}, it suffices to prove that the Bergman kernel $\rho_{\text{cyl},\alpha,k}$ of $(X_{\lambda,\vartheta},h_{\lambda,\vartheta}^k|w_n|^{-2\mathfrak{m}_k},\frac{\omega_{\lambda,\vartheta}^n}{n!})$ satisfies the conclusion of the main theorem. Since the holonomy of the Chern connection of $h_{\lambda,\vartheta}^{k+1}w_n|^{-2\mathfrak{m}_{k+1}}$ along the geodesic loop $\beta_1$ is just $e^{2\pi \mathfrak{m}_{k+1}i}$, by Formula \ref{e-cyl-twist-bergman}, Theorem \ref{thm-holonomy-computation} and Proposition \ref{prop-length-gamma-m}, we get the conclusion. 
	
\end{proof}

\section{Complex hyperbolic cusp of type I}\label{sec-cusp-I}
By Theorem \ref{thm-parabolic-conjugacy}, we can assume that $P$ is in the standard form for type I or type II. 
If $P=N(0,b)U$ with $U\in U(n-1)$, we can perform a unitary change of coordinates so that $U=\diag(e^{i\vartheta_1},\cdots,e^{i\vartheta_{n-1}})$. So
\begin{equation}\label{eqn-standard-parabolic-1}
	P(z',z_n)=(e^{i\vartheta}z',z_n+bi).
\end{equation}

If $P=N(v,0)U$ with \[U=\begin{bmatrix}
	U'& 0 \\
	0 & I_m
  \end{bmatrix},\quad U'\in U(n-m-1),\] and $ v=(a_1,\cdots,a_{n-1}) \text{ with } a_j=0\text{ for }j\leq n-m-1$, we can perform a unitary change of coordinates so that 
  \begin{itemize}
	\item $a_j=0$ for $j<n-1$;
	\item $a_{n-1}=bi$ for some $b>0$;
	\item $U=\diag(e^{i\vartheta_1},\cdots,e^{i\vartheta_{n-2}},1)$ with $\vartheta_j=0$ for $j\geq n-m$.
  \end{itemize}
    So \begin{equation}\label{eqn-standard-parabolic-2}
	P(z',z_n)=(e^{i\vartheta_1}z_1,\cdots,e^{i\vartheta_{n-2}}z_{n-2},z_{n-1}+bi ,z_n-ibz_{n-1}+\frac12b^2).
\end{equation}

\subsection{Construction of complex hyperbolic cusp of type I}
Given \[\vartheta=(\vartheta_1,\cdots,\vartheta_{n-1}), \quad \vartheta_j\in [-\pi,\pi], j=1,\cdots,n-1,\]and $b>0$, define $A^I_{\vartheta,b}\in \Aut(\mathfrak{H}_n)$ by \[A^I_{\vartheta,b}(z',z_n)=(e^{i\vartheta}z',z_n+b i).\] Let $S^I_{\vartheta,b}$ be the cyclic group generated by $A^I_{\vartheta,b}$. Let $X^I_{\vartheta,b}$ be the quotient space $\mathfrak{H}_n/S^I_{\vartheta,b}$. 

Define a holomorphic map $\Psi^I_{\vartheta,b}:\mathfrak{H}_n\to \C^{n-1}\times \C^*$ by \begin{equation}
	\Psi^I_{\vartheta,b}(z',z_n)=(z_1 e^{-\frac{\vartheta_1}{b}z_n},\cdots,z_{n-1} e^{-\frac{\vartheta_{n-1}}{b}z_n},e^{\frac{2\pi}{b}z_n}).
\end{equation}

\begin{proposition}
	$\Psi^{I}_{\vartheta,b}(x_1,\cdots,x_{n-1},x_n)=\Psi^{I}_{\vartheta,b}(y_1,\cdots,y_{n-1},y_n)$ if and only if $y=(A^{I}_{\vartheta,b})^m(x)$ for some $m\in \Z$.
\end{proposition}
\begin{proof}
	If $y=(A^{I}_{\vartheta,b})^m(x)$, then it is easy to verify that $\Psi^{I}_{\vartheta,b}(x)=\Psi^{I}_{\vartheta,b}(y)$. Conversely, if $\Psi^{I}_{\vartheta,b}(x)=\Psi^{I}_{\vartheta,b}(y)$, then we immediately get that $y_{n}=x_{n}+mbi$ for some $m\in \Z$. Thus \[y_j e^{-\frac{\vartheta_j}{b}y_{n}}=x_j e^{-\frac{\vartheta_j}{b}x_{n}}\] implies that $y_j=x_j e^{\frac{\vartheta_j}{b}mbi}$ for $j\leq n-1$. So we have
	$y=(A^{I}_{\vartheta,b})^m(x)$ for some $m\in \Z$.
\end{proof}
Let $w=(w',w_n)$ be the coordinates of $\C^{n-1}\times \C^*$. Let $\tau_n\eqd \log |w_n|$. Then $\tau_n=\frac{2\pi}{b}\Re z_n$. 
\begin{proposition}\label{prop-biholog-I}
	$\Psi^{I}_{\vartheta,b}$ is a local biholomorphism.
\end{proposition}
\begin{proof}
	We only need to check the Jacobian of $\Psi^{I}_{\vartheta,b}$. It is straightforward to calculate that
	\[
 \frac{dw_n}{w_n}=\frac{2\pi }{b} dz_n,\quad dw_j=e^{-\frac{\vartheta_j}{b}z_n}dz_j-z_je^{-\frac{\vartheta_j}{b}z_n}\frac{\vartheta_j}{b}dz_n,\quad j<n.
\]
So \begin{equation}\label{eqn-jacobian-I}
	\frac{1}{w_n}dw_1\wedge \cdots \wedge dw_n= \frac{2\pi }{b}(\prod_{j=1}^{n-1}e^{-\frac{\vartheta_j}{b}z_n})dz_1 \wedge \cdots \wedge dz_n.
\end{equation}

Thus, $\Psi^{I}_{\vartheta,b}$ is a local biholomorphism.
\end{proof}
Therefore, we can identify $\ccal^{I}_{\vartheta,b}$ with the image of $\Psi^{I}_{\vartheta,b}$. Since 
\[ |z_j|^2 =|w_j|^2e^{2\frac{\vartheta_j}{b}\Re z_n}= |w_j|^2e^{\frac{\vartheta_j}{\pi}\tau_n}, j<n,\]
we have
\[
2\Re z_n-\sum_{j=1}^{n-1}|z_j|^2 =\frac{b}{\pi}\tau_n-\sum_{j=1}^{n-1}|w_j|^2e^{\frac{\vartheta_j}{\pi}\tau_n}.
\]
Let $\kappa_I=\frac{b}{\pi}\tau_n-\sum_{j=1}^{n-1}|w_j|^2e^{\frac{\vartheta_j}{\pi}\tau_n}$, then $\ccal^{I}_{\vartheta,b}$ is the subset of $\C^{n-1}\times \C^*$ defined by $\kappa_I>0$. Let $\omega_{I}$ be the \kahler form on $\ccal^{I}_{\vartheta,b}$ inherited from $\mathfrak{H}_n$. Clearly, we have \begin{equation}
	-2\iddbar \log \kappa_I=\omega_{I}.
\end{equation}
\begin{proposition}
	$-\log \kappa_I+\tau_n$ is a strictly plurisubharmonic exhaustion function on $\ccal^{I}_{\vartheta,b}$. So $\ccal^{I}_{\vartheta,b}$ is a Stein manifold.
\end{proposition}
\begin{proof}
	We only need to show that, for $0<c<1$, the set $E_c=\{w\in \ccal^{I}_{\vartheta,b} \big| -\log \kappa_I+\tau_n\leq c \}$ is compact. The condition is equivalent to $\kappa_I\geq e^{-c}e^{\tau_n}$.
	Similar to the proof of Proposition \ref{prop-X-s-psh}, we define \[F\eqd \{z\in \mathfrak{H}^n \big|0\leq \Im z_n\leq b  \}. \]
	Then the restriction of $\Psi^{I}_{\vartheta,b}$ to $F$ is surjective onto $\ccal^{I}_{\vartheta,b}$. Let $F_c=(\Psi^{I}_{\vartheta,b})^{-1}(E_c)\cap F$. Then \[F_c=\{z\in F\big|2\Re z_n-\sum_{j=1}^{n-1}|z_j|^2 \geq e^{-c}e^{\frac{2\pi}{b}\Re z_n}\}. \]
	So $\Re z_n$ is bounded, hence $|z'|$ being bounded. And since $\Re z_n\geq \frac12e^{-c}$, we conclude that $F_c$ is compact. Thus $E_c$ is also compact.
\end{proof}
And since $\ccal^{I}_{\vartheta,b}$ is clearly homotopic  to $S^1$, every holomorphic line bundle over $\ccal^{I}_{\vartheta,b}$ is trivial. Since $\Psi^{I}_{\vartheta,b}$ can be extended to $\C^n$, it is easy to see that $\ccal^{I}_{\vartheta,b}$ is the subset of $\C^{n-1}\times \C^*$ defined by $\kappa_I>0$.

Denote by $\Psi^{I,-1}_{\vartheta,b}$ a local inverse of $\Psi^{I}_{\vartheta,b}$.
By \eqref{eqn-jacobian-I}, we have \begin{eqnarray*}
	\frac{\omega_{I}^n}{n!}&=& (\Psi^{I,-1}_{\vartheta,b})^*\frac{\omega_{H}^n}{n!} \\
	&=& \frac{2^{n-2}b^2 }{\pi^2(\frac{b}{\pi}\tau_n-\sum_{j=1}^{n-1}|w_j|^2e^{\frac{\vartheta_j}{\pi}\tau_n})^{n+1}}e^{2\Re \frac{z_n}{b}\sum_{j=1}^{n-1}\vartheta_j} \frac{|dw|^2}{|w_n|^2}\\
	&=& \frac{2^{n-2}b^2 }{\pi^2(\frac{b}{\pi}\tau_n-\sum_{j=1}^{n-1}|w_j|^2e^{\frac{\vartheta_j}{\pi}\tau_n})^{n+1}}|w_n|^{ \frac{1}{\pi}\sum_{j=1}^{n-1}\vartheta_j-2}|dw|^2 .
\end{eqnarray*}

\

Let $\hcal_{I,k}$ be the Bergman space consisting of holomorphic functions $f$ on $\ccal^{I}_{\vartheta,b}$ such that
\begin{equation}
	\|f\|^2_{I,k}=\int_{\ccal^{I}_{\vartheta,b}} |f|^2 \kappa_I^{2k} \frac{\omega_{I}^n}{n!}<+\infty.
	\end{equation}
Let $\rho_{I,k}$ be the Bergman kernel function of $\hcal_{I,k}$.	

For a fixed number $\alpha\in [0,1)$, recall that $\mathfrak{m}_k=k\alpha-[k\alpha]$ is the decimal part. We can consider the Bergman space $\hcal_{I,\alpha,k}$ consisting of holomorphic functions $f$ on $X_{\lambda,\vartheta}$ such that  \[\|f\|^2_{I,\alpha,k}\eqd\int_{\ccal^{I}_{\vartheta,b}} |f|^2 \kappa_I^{2k} |w_n|^{-2\mathfrak{m}_k}\frac{\omega_{I}^n}{n!}<\infty. \]
We denote by $\rho_{I,\alpha,k}$ the Bergman kernel function of $\hcal_{I,\alpha,k}$.

The main result of this section is the following theorem.
\begin{theorem}\label{thm-main-cusp-model-I}
	Theorem \ref{thm-main} holds for $(\ccal^{I}_{\vartheta,b},\omega_I)$, with $k_0=n+1$.
\end{theorem}

\subsection{Approximation by complex hyperbolic cylinders}

Define $E^{I}_{\vartheta,\lambda,b}:X_{\lambda,-\vartheta}\to \C^{n-1}\times \C^*$ by 
\begin{equation}
	E^{I}_{\vartheta,\lambda,b}(z_1,\cdots,z_{n-1},z_n)=\big(\sqrt{\frac{b}{2\log\lambda}}e^{-\frac{\vartheta_1\pi}{4\log \lambda}} z_1,\cdots,\sqrt{\frac{b}{2\log\lambda}}e^{-\frac{\vartheta_{n-1}\pi}{4\log \lambda}}z_{n-1},e^{\frac{\pi^2}{2\log\lambda}} z_n\big).
\end{equation}
\begin{proposition}\label{prop-biholom-I}
	$E^{I}_{\vartheta,\lambda,b}$ is a biholomorphism from $X_{\lambda,\vartheta}$ to an open subset of $\ccal^{I}_{\vartheta,b}$.
	\end{proposition}
\begin{proof}
	Clearly, $E^{I}_{\vartheta,\lambda,b}$ is injective and a local biholomorphism. Since $\tau_n=t_n+\frac{\pi^2}{2\log\lambda}$, we have $\frac{\log \lambda}{\pi}t_n=\frac{\log \lambda}{\pi}\tau_n-\frac{\pi}{2}$. Thus, \[\cos(\frac{\log \lambda}{\pi}t_n)=\sin(\frac{\log \lambda}{\pi}\tau_n). \]
	Since $|w_j|^2=\frac{b}{2\log\lambda}e^{-\frac{\vartheta_{j}\pi}{2\log \lambda}}|z_j|^2$, we have that $2\cos(\frac{\log \lambda}{\pi}t_n)-\sum_{j=1}^{n-1}|z_j|^2e^{\frac{\vartheta_j}{\pi}t_n}$ equals
	\begin{eqnarray*}
		&&2\sin(\frac{\log \lambda}{\pi}\tau_n)-\sum_{j=1}^{n-1}\frac{2\log\lambda}{b}e^{\frac{\vartheta_{j}\pi}{2\log \lambda}}|w_j|^2e^{\frac{\vartheta_j}{\pi}\tau_n-\frac{\vartheta_{j}\pi}{2\log \lambda}}\\
		&=&2\sin(\frac{\log \lambda}{\pi}\tau_n)-\frac{2\log\lambda}{b}\sum_{j=1}^{n-1}|w_j|^2e^{\frac{\vartheta_j}{\pi}\tau_n}.
	\end{eqnarray*}
	Thus, \begin{equation}\label{eq-cyl-cusp-I}
		2\cos(\frac{\log \lambda}{\pi}t_n)-\sum_{j=1}^{n-1}|z_j|^2e^{\frac{\vartheta_j}{\pi}t_n}=\frac{2\log\lambda}{b}\big(\frac{b}{\log\lambda}\sin(\frac{\log \lambda}{\pi}\tau_n)-\sum_{j=1}^{n-1}|w_j|^2e^{\frac{\vartheta_j}{\pi}\tau_n} \big).
	\end{equation}
	Since $\tau_n\geq 0$, we have $\frac{b}{\log\lambda}\sin(\frac{\log \lambda}{\pi}\tau_n)\leq \frac{b}{\pi}\tau_n$. So the proposition follows from \eqref{eq-cyl-cusp-I}.

\end{proof}
We denote by \begin{equation}
	\tilde{X}_{\lambda,I} = E^{I}_{\vartheta,\lambda,b}(X_{\lambda,-\vartheta}).
\end{equation}
Recall that the Hermitian metric on the trivial line bundle on $X_{\lambda,-\vartheta}$ is defined by $h_{\lambda,-\vartheta}=\frac14(2\cos(\frac{\log \lambda}{\pi}t_n)-\sum_{j=1}^{n-1}|z_j|^2e^{\frac{\vartheta_j}{\pi}t_n})^2$. On $\tilde{X}_{\lambda,I}$, we define a Hermitian metric on the trivial line bundle by rescaling that on $X_{\lambda,-\vartheta}$ as follows: \begin{equation}
	\tilde{h}_{\lambda,I}=(\frac{\log \lambda}{b})^2(2\cos(\frac{\log \lambda}{\pi}t_n)-\sum_{j=1}^{n-1}|z_j|^2e^{\frac{\vartheta_j}{\pi}t_n})^2.
\end{equation}
By \eqref{eq-cyl-cusp-I}, the following proposition is straightforward.
\begin{proposition}\label{prop-h-lambda-tilde-I}
	On each compact subset $K\subset \ccal^{I}_{\vartheta,b}$, we have that $\tilde{h}_{\lambda,I}$ converges to $h_{I}=\kappa_I^2$ in the $C^5$-topology as $\lambda\to 1^+$.
\end{proposition}
Let $\tilde{\omega}_{\lambda,I}$ be the pull-back of the \kahler form on $X_{\lambda,-\vartheta}$ to $\tilde{X}_{\lambda,I}$ via the inverse of $E^{I}_{\vartheta,\lambda,b}$.
This proposition implies that $\tilde{\omega}_{\lambda,I}$ converges to $\omega_I$ in $C^3(K)$ topology for each compact subset $K\subset \ccal^{I}_{\vartheta,b}$.
From \eqref{e-omega-lambda-n}, one easily get that 
\begin{equation}\label{e-tilde-omega-n}
	\frac{\tilde{\omega}_{\lambda,I}^n}{n!}=\frac{2^{n-2}b^2}{\pi^2\big(\frac{b}{\log\lambda}\sin(\frac{\log \lambda}{\pi}\tau_n)-\sum_{j=1}^{n-1}|w_j|^2e^{\frac{\vartheta_j}{\pi}\tau_n} \big)^{n+1}}|w_n|^{ \frac{1}{\pi}\sum_{j=1}^{n-1}\vartheta_j-2}|dw|^2 .
\end{equation}

\eqref{eq-cyl-cusp-I} also implies \begin{equation}\label{eqn-h-tilde-leq-h-I}
	\tilde{h}_{\lambda,I}\leq h_I \text{ on }\tilde{X}_{\lambda,I}.
\end{equation}
 Let $\tilde{\rho}_{\lambda,I,\alpha,k}$ denote the Bergman kernel function of $(\tilde{X}_{\lambda,I},\tilde{h}_{\lambda,I}^k|w_n|^{-2\mathfrak{m}_k},\frac{\tilde{\omega}_{\lambda,I}^n}{n!})$, and $\rho_{\lambda,I,\alpha,k}$ denote the Bergman kernel function of $(X_{{\lambda,-\vartheta}},h_{\lambda,-\vartheta}^k|z_n|^{-2\mathfrak{m}_k},\frac{\omega_{\lambda,-\vartheta}^n}{n!})$. Then, since rescaling the Hermitian metric by a constant factor does not change the Bergman kernel function, we have \begin{equation}\label{eqn-rho-tilde-rho}
	\tilde{\rho}_{\lambda,I,\alpha,k}(E^{I}_{\vartheta,\lambda,b}(p))= \rho_{\lambda,-\vartheta,\alpha,k}(p),
\end{equation}for each $p\in X_{\lambda,-\vartheta}$.

\begin{theorem}\label{thm-rho-tilde-convergence-I}
	Suppose that $2k\geq n+1$.
	For each point $p\in \ccal^{I}_{\vartheta,b}$, we have 
	\begin{equation}
		\lim_{\lambda\to 0^+}\tilde{\rho}_{\lambda,I,\alpha,k}(p)=\rho_{I,\alpha,k}(p).
	\end{equation}
\end{theorem}
\begin{proof}
	Let $r_p$ be the injectivity radius of $(\ccal^{I}_{\vartheta,b},\omega_I)$ at $p$. Then, for $\lambda-1$ small enough, the injectivity radius of $(\tilde{X}_{\lambda,I},\tilde{\omega}_{\lambda,I})$ at $p$ is greater than $\frac12r_p$. So $\rho_{\lambda,-\vartheta,\alpha,k}(p)$ is bounded. Let $\{\lambda_j\}$ be a sequence such that $\lambda_j\to 1^+$ and \[\lim_{j\to \infty}\tilde{\rho}_{\lambda_j,\vartheta,\alpha,k}(p)\] converges. Let $s_j$ be a peak section of $\hcal_{\lambda_j,\vartheta,\alpha,k}$ at $p$. On each compact subset $K\subset \ccal^{I}_{\vartheta,b}$, by Proposition \ref{prop-h-lambda-tilde}, the $L_2$-norm $\left\|s_j \right\|_{K,I,\alpha,k}$ of $s_j$ with respect to the metric $(h_I^kk|w_n|^{-2\mathfrak{m}_k},\frac{\omega_I^n}{n!})$ is $<2$ for $j$ sufficiently large. So we can find a subsequence of $s_j$ that converges uniformly on $K$. By a diagonal argument, we can find subsequence $s_{j_l}$ that converges uniformly on each compact $K\subset \ccal^{I}_{\vartheta,b}$. Let $s$ be the limit of $s_{j_l}$. Then on each compact subset $K\subset \ccal^{I}_{\vartheta,b}$, we have \[\left\|s \right\|_{K,I,k}=\lim_{l\to \infty}\left\|s_{j_l} \right\|_{K,\tilde{h}_{\lambda_{j_l},\vartheta}^k|w_n|^{-2\mathfrak{m}_k}}\leq 1. \]
	So \[\left\|s \right\|_{\ccal^{I}_{\vartheta,b},\alpha,k}\leq 1.\] Since the norm of $s$ at $p$ is the limit of $\sqrt{\tilde{\rho}_{\lambda,-\vartheta,\alpha,k}(p)}$,
	this implies that \[\rho_{I,\alpha,k}(p)\geq \lim_{l\to\infty}\tilde{\rho}_{\lambda_j,\vartheta,\alpha,k}(p).\]
	Thus, we have \begin{equation}
		\limsup_{\lambda\to 1^+}\tilde{\rho}_{\lambda,-\vartheta,\alpha,k}(p)\leq \rho_{I,\alpha,k}(p).
	\end{equation}

	On the other hand, let $s$ be a peak section of $\hcal_{I,k}$ at $p$. By Proposition \ref{prop-biholom-I} and Formulas \eqref{eqn-h-tilde-leq-h-I} and \eqref{e-tilde-omega-n}, we have $s $ is $L_2$-integrable with respect to $(X_{\lambda,-\vartheta},h_{\lambda,-\vartheta}^k|w_n|^{-2\mathfrak{m}_k},\frac{\omega_{\lambda,-\vartheta}^n}{n!})$ and \[ \left\|s \right\|_{X_{\lambda,-\vartheta},h_{\lambda,-\vartheta}^k|w_n|^{-2\mathfrak{m}_k}}\leq 1.\] So \[\left|s(p)\right|^2_{\tilde{h}_{\lambda,I}^k|w_n|^{-2\mathfrak{m}_k}}\leq \tilde{\rho}_{\lambda,-\vartheta,\alpha,k}(p).\]
	 This implies that \begin{equation}
		\rho_{I,k}(p)=\lim_{\lambda\to 1^+} \left|s(p)\right|^2_{\tilde{h}_{\lambda,I}^k|w_n|^{-2\mathfrak{m}_k}}\leq  \liminf_{\lambda\to 1^+}\tilde{\rho}_{\lambda,-\vartheta,\alpha,k}(p).
		\end{equation}
	 Therefore, the theorem follows.
	
\end{proof}

For each point $p\in \ccal^{I}_{\vartheta,b}$ and each $m\in \Z^*$, it is easy to see that there is exactly one oriented geodesic loop $\gamma_{p,m}^I$, with respect to $(\ccal^{I}_{\vartheta,b},\omega_I)$, based at $p$ such that the winding number of $\gamma$ at $p$ is $m$. More precisely, let $\tilde{p}$ be a lift of $p$ in $\mathfrak{H}_n$, then $\gamma_{p,m}^I$ is the projection of the geodesic segment $[\tilde{p},(A^{I}_{\vartheta,b})^m(\tilde{p})]$ to $\ccal^{I}_{\vartheta,b}$. Similarly, we also denote by $\gamma_{p,m,\lambda}$ the oriented geodesic loop, with respect to $(\tilde{X}_{\lambda,I},\tilde{\omega}_{\lambda,I})$, based at $p$ such that the winding number of $\gamma_{p,m,\lambda}$ at $p$ is $m$.
\begin{theorem}\label{thm:geodesic_loop-I}
	For each point $p\in \ccal^{I}_{\vartheta,b}$ and each $m\in \Z^*$, $\gamma_{p,m,\lambda}$ converges to $\gamma_{p,m}^I$ as $\lambda\to 1^+$.
\end{theorem}
\begin{proof}
	We first show that the lengh of $\gamma_{p,m,\lambda}$ converges to the length of $\gamma_{p,m}^I$ as $\lambda\to 1^+$. Let $p=(w',w_n)$ and let $\tilde{p}=(z',z_n)$ be a lift of $p$ in $\mathfrak{H}_n$. Let $\ell_{I,m}$ denote the length of the geodesic segment $[\tilde{p},(A^{I}_{\vartheta,b})^m(\tilde{p})]$. We have 
	\begin{eqnarray*}
		\cosh\frac{\ell_{I,m}}{2}&=&\frac{\big|z_n+\bar{z}_n-mbi-\sum_{j=1}^{n-1}e^{-im\vartheta_j}|z_j|^2 \big|}{2\Re z_n-\sum_{j=1}^{n-1}|z_j|^2}\\
		&=&\frac{\big|\frac{b}{\pi}\tau_n -mbi-\sum_{j=1}^{n-1}e^{-im\vartheta_j}|w_j|^2e^{\frac{\vartheta_j}{\pi}\tau_n}  \big|}{\frac{b}{\pi}\tau_n-\sum_{j=1}^{n-1}|w_j|^2e^{\frac{\vartheta_j}{\pi}\tau_n} }.
	\end{eqnarray*}
	
	Let $q\in X_{\lambda,-\vartheta}$ be the image of $p$ under the inverse of $E^{I}_{\vartheta,\lambda,b}$. Let $q=(x',x_n)$ and let $\tilde{q}=(z',z_n)$ be a lift of $q$ in $\mathfrak{H}_n$.
	Let $\ell_{m,\lambda}$ denote the length of the geodesic segment $[\tilde{q},A_{\lambda,-\vartheta}^m(\tilde{q})]$. We use the notation $z_n=|z_n|e^{i\theta_n}$ and $t_n=\log |x_n|$.

	By \eqref{eqn-length-gamma-m}, we have 
	\begin{eqnarray*}
		\cosh\frac{\ell_{m,\lambda}}{2}&=&\frac{\big|\lambda^{2m}\bar{z}_n+ z_n-\lambda^m \sum_{j=1}^{n-1}e^{im\vartheta_j}|z_j|^2 \big|}{\lambda^m(2\Re z_n-\sum_{j=1}^{n-1}|z_j|^2)}\\
		&=&\frac{\big|\lambda^{m}e^{-i\theta_n} +\lambda^{-m} e^{i\theta_n}-\sum_{j=1}^{n-1}e^{im\vartheta_j}|x_j|^2e^{\frac{\vartheta_j}{\log \lambda}\theta_n}  \big|}{2\cos(\frac{\log\lambda}{\pi}t_n)-\sum_{j=1}^{n-1}|x_j|^2e^{\frac{\vartheta_j}{\pi}t_n} }\\
		 &=&\frac{|f(\lambda,w)|}{\frac{2\log\lambda}{b}\big(\frac{b}{\log\lambda}\sin(\frac{\log \lambda}{\pi}\tau_n)-\sum_{j=1}^{n-1}|w_j|^2e^{\frac{\vartheta_j}{\pi}\tau_n} \big)},
	\end{eqnarray*}
where $f(\lambda,w)$ equals \[(\lambda^{m}+\lambda^{-m})\cos(\frac{\log\lambda}{\pi}t_n) -i(\lambda^{m}-\lambda^{-m})\sin(\frac{\log\lambda}{\pi}t_n)-\sum_{j=1}^{n-1}e^{im\vartheta_j}|x_j|^2e^{\frac{\vartheta_j}{\pi}t_n},\]
which further transforms to
\begin{equation}\label{eqn-f-lambda-w}
	(\lambda^{m}+\lambda^{-m})\sin(\frac{\log\lambda}{\pi}\tau_n)+i(\lambda^{m}-\lambda^{-m})\cos(\frac{\log\lambda}{\pi}\tau_n)-\frac{2\log\lambda}{b}\sum_{j=1}^{n-1}e^{im\vartheta_j}|w_j|^2e^{\frac{\vartheta_j}{\pi}\tau_n}.
\end{equation}

Then since \[lim_{\lambda\to 1^+}(\lambda^{m}+\lambda^{-m})\frac{b}{2\log\lambda}\sin(\frac{\log\lambda}{\pi}\tau_n)=\frac{b}{\pi}\tau_n, \]
and \[lim_{\lambda\to 1^+}(\lambda^{m}-\lambda^{-m})\frac{b}{2\log\lambda}\cos(\frac{\log\lambda}{\pi}\tau_n)=mb, \]
we have 
\[\lim_{\lambda\to 1^+}\cosh\frac{\ell_{m,\lambda}}{2}=\frac{\big|\frac{b}{\pi}\tau_n +mbi-\sum_{j=1}^{n-1}e^{im\vartheta_j}|w_j|^2e^{\frac{\vartheta_j}{\pi}\tau_n}  \big|}{\frac{b}{\pi}\tau_n-\sum_{j=1}^{n-1}|w_j|^2e^{\frac{\vartheta_j}{\pi}\tau_n} }=\cosh\frac{\ell_{I,m}}{2}.\]

 We fix $m$. For each $\eta>0$, we denote by $v_\lambda$ the initial unit tangent vector of $\gamma_{p,m,\lambda}$. If a sequence $\{\lambda_j\}$ converges to $1^+$, then by compactness of the unit tangent bundle, we can extract a subsequence $\{\lambda_{j_k}\}$ such that $v_{\lambda_{j_k}}$ converges to a unit vector $v$. Since $\tilde{\omega}_{\lambda,I}$ converges to $\omega_I$ in $C^3(K)$ topology for each compact subset $K\subset \ccal^{I}_{\vartheta,b}$, by the standard theory of ODE, $\gamma_{p,m,\lambda_{j_k}}$ converges to a geodesic $\gamma_v$, with respect to $(\ccal^{I}_{\vartheta,b},\omega_I)$, starting at $p$ with initial vector $v$. Since the length of $\gamma_{p,m,\lambda}$ converges, $\{\ell_{m,\lambda}\}_\lambda<\delta$ is bounded. Thus, $\gamma_v$ is a geodesic loop based at $p$.
Clearly $\gamma_v$ also has winding number $m$, hence, $\gamma_v=\gamma_{p,m}^I$. Since the limit is unique, we have that $v_\eta$ converges to the initial unit vector of $\gamma_{w,m}$ as $\eta\to 0$, and that $\gamma_{p,m,\lambda}$ converges to $\gamma_{p,m}^I$. This completes the proof.
\end{proof}

\begin{proof}[Proof of Theorem \ref{thm-main-cusp-model-I}]
	The conclusion of Lemma \ref{lem-twist-metric} also holds for $\ccal^{I}_{\vartheta,b}$. Thus, we can assume that the Hermitian metric on the trivial line bundle over $\ccal^{I}_{\vartheta,b}$ is $h_I|w_n|^{-2\alpha}$ for some $\alpha\in [0,1)$. And for $k\geq 1$, the Hermitian metric on the trivial line bundle over $\ccal^{I}_{\vartheta,b}$ is $h_I^k|w_n|^{-2\mathfrak{m}_k}$, where $\mathfrak{m}_k=k\alpha-[k\alpha]$.

	We adapt the notations in Theorem \ref{thm:geodesic_loop-I}. For each point $p\in \ccal^{I}_{\vartheta,b}$, and for each geodesic loop $\gamma_{p,m}$ (with respect to $\omega_I$), we denote by $e^{2\pi \alpha_{p,m}i}$ the holonomy of the Chern connection of $h_I|w_n|^{-2\alpha}$ along $\gamma_m$. Similarly, for each $\lambda>1$ and for each geodesic loop $\gamma_{p,m,\lambda}$ (with respect to $\tilde{\omega}_{\lambda,I}$), we denote by $e^{2\pi \alpha_{p,m,\lambda}i}$ the holonomy of the Chern connection of $\tilde{h}_{\lambda,I}|w_n|^{-2\alpha}$ along $\gamma_{p,m,\lambda}$. Then by Proposition \ref{prop-h-lambda-tilde-I} Theorem \ref{thm:geodesic_loop-I}, we have \[\lim_{\lambda\to 1^+}e^{2\pi \alpha_{p,m,\lambda}i}=e^{2\pi \alpha_{p,m}i}. \]

	By Theorem \ref{thm-rho-tilde-convergence-I}, we only need to show that \begin{equation}
		\lim_{\lambda\to 1^+}\sum_{m\in \Z^*}\cosh^{-2k}\frac{\ell_{I,m,\lambda}}{2}e^{-2\pi k\alpha_{p,m,\lambda} i}=\sum_{m\in \Z^*}\cosh^{-2k}\frac{\ell_{I,m}}{2}e^{-2\pi k\alpha_{p,m} i}.
		\label{eqn-rho-tilde-series-I}
		\end{equation}

	For $\log\lambda$ small enough, we have $\cos(\frac{\log\lambda}{\pi}\tau_n)>\frac12$.
	Since $\frac{\lambda^{m}-\lambda^{-m}}{2}=\sinh(m\log\lambda)$ and $\sinh x\geq x$ for $x\geq 0$, we have \[\frac{2}{2\log\lambda}\cos(\frac{\log\lambda}{\pi}\tau_n)>\frac{m}{2}. \]Thus, \eqref{eqn-f-lambda-w} implies that there exist $ \lambda_0>0, c>0$ independent of $m$ and $m_0$ such that \[\cosh\frac{\ell_{m,\lambda}}{2}>cm,\] for $m\geq m_0$ and $\lambda\leq \lambda_0$. Then, since $\sum m^{-2k}$ converges, by Theorem \ref{thm:geodesic_loop-I}, dominated convergence theorem implies \eqref{eqn-rho-tilde-series-I}. This completes the proof of Theorem \ref{thm-main-cusp-model-I}.
\end{proof}

\section{Complex hyperbolic cusp of type II}\label{sec-cusp-II}
\subsection{Construction of complex hyperbolic cusp of type II}

Given \[\vartheta=(\vartheta_1,\cdots,\vartheta_{n-2}), \quad \vartheta_j\in [-\pi,\pi], j=1,\cdots,n-2,\] and $b>0$, define $A^{II}_{\vartheta,b}\in \Aut(\mathfrak{H}_n)$ by \begin{equation*}
	A^{II}_{\vartheta,b}(z'',z_{n-1},z_n)=(e^{i\vartheta}z'',z_{n-1}+bi,z_n-ibz_{n-1}+\frac12b^2),
\end{equation*}
where $z''=(z_1,\cdots,z_{n-2})$. Let $S^{II}_{\vartheta,b}$ be the cyclic group generated by $A^{II}_{\vartheta,b}$. Let $\ccal^{II}_{\vartheta,b}$ be the quotient space $\mathfrak{H}_n/S^{II}_{\vartheta,b}$. 
Define $\Psi^{II}_{\vartheta,b}:\mathfrak{H}_n\to \C^{n-2}\times \C^*\times \mathfrak{H}_1$ by 
\begin{equation}
	\Psi^{II}_{\vartheta,b}(z_1,\cdots,z_{n-1},z_n)=(z_1 e^{-\frac{\vartheta_1}{b}z_{n-1}},\cdots,z_{n-2} e^{-\frac{\vartheta_{n-2}}{b}z_{n-1}},e^{\frac{2\pi}{b}z_{n-1}}, z_n+\frac{1}{2}z_{n-1}^2).
\end{equation}
\begin{proposition}
	$\Psi^{II}_{\vartheta,b}(x_1,\cdots,x_{n-1},x_n)=\Psi^{II}_{\vartheta,b}(y_1,\cdots,y_{n-1},y_n)$ if and only if $y=(A^{II}_{\vartheta,b})^m(x)$ for some $m\in \Z$.
\end{proposition}
\begin{proof}
	If $y=(A^{II}_{\vartheta,b})^m(x)$, then it is easy to verify that $\Psi^{II}_{\vartheta,b}(x)=\Psi^{II}_{\vartheta,b}(y)$. Conversely, if $\Psi^{II}_{\vartheta,b}(x)=\Psi^{II}_{\vartheta,b}(y)$, then we immediately get that $y_{n-1}=x_{n-1}+mbi$ for some $m\in \Z$. Thus \[y_j e^{-\frac{\vartheta_j}{b}y_{n-1}}=x_j e^{-\frac{\vartheta_j}{b}x_{n-1}}\] implies that $y_j=x_j e^{\frac{\vartheta_j}{b}mbi}$ for $j\leq n-2$. Finally, from \[y_n+\frac{1}{2}y_{n-1}^2=x_n+\frac{1}{2}x_{n-1}^2,\] we get that $y_n=x_n-imbx_{n-1}+\frac12 (mb)^2$. So we have
	$y=(A^{II}_{\vartheta,b})^m(x)$ for some $m\in \Z$.

\end{proof}
Let $w=(w',w_n)$ be the coordinates of $\C^{n-1}\times \C^*$. Let $\tau_{n-1}\eqd \log |w_{n-1}|$. Then $\tau_{n-1}=\frac{2\pi}{b}\Re z_{n-1}$. 
\begin{proposition}\label{prop-biholog-I}
	$\Psi^{II}_{\vartheta,b}$ is a local biholomorphism.
\end{proposition}
\begin{proof}
	We only need to check the Jacobian of $\Psi^{II}_{\vartheta,b}$. It is straightforward to calculate that
	\[
 \frac{dw_{n-1}}{w_{n-1}}=\frac{2\pi }{b} dz_{n-1},\quad dw_n=dz_n+z_{n-1}dz_{n-1}
\]
and \[dw_j=e^{-\frac{\vartheta_j}{b}z_{n-1}}dz_j-z_je^{-\frac{\vartheta_j}{b}z_{n-1}}\frac{\vartheta_j}{b}dz_n,\quad j<n-1.\]
So \begin{equation}\label{eqn-jacobian-II}
	\frac{1}{w_{n-1}}dw_1\wedge \cdots \wedge dw_n= \frac{2\pi }{b}(\prod_{j=1}^{n-2}e^{-\frac{\vartheta_j}{b}z_{n-1}})dz_1 \wedge \cdots \wedge dz_n.
\end{equation}

Thus, $\Psi^{II}_{\vartheta,b}$ is a local biholomorphism.
\end{proof}

Therefore, we can identify $\ccal^{II}_{\vartheta,b}$ with the image of $\Psi^{II}_{\vartheta,b}$. 

Since $2\Re w_n=2\Re z_n+(\Re z_{n-1})^2-(\Im z_{n-1})^2$, we have 
\[2\Re z_n-|z_{n-1}|^2=2\Re w_n-2(\Re z_{n-1})^2=2\Re w_n-\frac{b^2}{2\pi^2}\tau_{n-1}^2. \]
And since $|z_j|^2=|w_j|^2e^{\frac{2\vartheta_j}{b}\Re z_{n-1}}=|w_j|^2e^{\frac{\vartheta_j}{\pi}\tau_{n-1}}$, we have \begin{equation}
	2\Re z_n-\sum_{j=1}^{n-1}|z_j|^2=2\Re w_n-\frac{b^2}{2\pi^2}\tau_{n-1}^2-\sum_{j=1}^{n-2}|w_j|^2e^{\frac{\vartheta_j}{\pi}\tau_{n-1}}.
\end{equation}
Let $\kappa_{II}=2\Re w_n-\frac{b^2}{2\pi^2}\tau_{n-1}^2-\sum_{j=1}^{n-2}|w_j|^2e^{\frac{\vartheta_j}{\pi}\tau_{n-1}}$, then $\ccal^{II}_{\vartheta,b}$ is the subset of $\C^{n-1}\times \C^*$ defined by $\kappa_{II}>0$.
Let $\omega_{II}$ be the \kahler form on $\ccal^{II}_{\vartheta,b}$ inherited from $\mathfrak{H}^n$.

\begin{proposition}
	$-\log \kappa_{II}+2\Re w_n$ is a strictly plurisubharmonic exhaustion function on $\ccal^{II}_{\vartheta,b}$. So $\ccal^{II}_{\vartheta,b}$ is a Stein manifold.
\end{proposition}
\begin{proof}
	We only need to show that, for $0<c<1$, the set $E_c=\{w\in \ccal^{II}_{\vartheta,b} \big| -\log \kappa_{II}+2\Re w_n\leq c \}$ is compact. The condition is equivalent to $\kappa_{II}\geq e^{-c}e^{2\Re w_n}$.
	We define \[F\eqd \{z\in \mathfrak{H}^n \big|0\leq \Im z_{n-1}\leq b  \}. \]
	Then the restriction of $\Psi^{II}_{\vartheta,b}$ to $F$ is surjective onto $\ccal^{II}_{\vartheta,b}$. Let $F_c=(\Psi^{II}_{\vartheta,b})^{-1}(E_c)\cap F$. Then \[F_c=\{z\in F\big|2\Re z_n-\sum_{j=1}^{n-1}|z_j|^2 \geq e^{-c}e^{2\Re z_n+(\Re z_{n-1})^2-(\Im z_{n-1})^2}\}. \]
	So \[2\Re z_n\leq e^{-c-b^2}e^{2\Re z_n}. \] Thus, $2\Re z_n$ is bounded. So $|z'|$ is also bounded. Finally, since $\Re z_n\geq e^{-c-b^2}$, we conclude that $F_c$ is compact. Thus $E_c$ is also compact.
\end{proof}
And since $\ccal^{II}_{\vartheta,b}$ is clearly homotopic  to $S^1$, every holomorphic line bundle over $\ccal^{II}_{\vartheta,b}$ is trivial. Since $\Psi^{II}_{\vartheta,b}$ can be extended to $\C^n$, it is easy to see that $\ccal^{II}_{\vartheta,b}$ is the subset of $\C^{n-1}\times \C^*$ defined by $\kappa_{II}>0$.

Denote by $\Psi^{II,-1}_{\vartheta,b}$ a local inverse of $\Psi^{II}_{\vartheta,b}$.
By \eqref{eqn-jacobian-II}, we have \begin{eqnarray*}
	\frac{\omega_{II}^n}{n!}&=& (\Psi^{II,-1}_{\vartheta,b})^*\frac{\omega_{H}^n}{n!} = \frac{2^{n-2}b^2 }{\pi^2\kappa_{II}^{n+1}}e^{2\Re \frac{z_{n-1}}{b}\sum_{j=1}^{n-2}\vartheta_j} \frac{|dw|^2}{|w_{n-1}|^2}\\
	&=& \frac{2^{n-2}b^2 }{\pi^2\kappa_{II}^{n+1}}|w_{n-1}|^{ \frac{1}{\pi}\sum_{j=1}^{n-2}\vartheta_j-2}|dw|^2 .
\end{eqnarray*}

Let $\hcal_{II,\alpha,k}$ be the Bergman space consisting of holomorphic functions $f$ on $\ccal^{II}_{\vartheta,b}$ such that
\begin{equation}
	\|f\|^2_{II,\alpha,k}=\int_{\ccal^{II}_{\vartheta,b}} |f|^2 \kappa_{II}^{2k}|w_{n-1}|^{-2\mathfrak{m}_k} \frac{\omega_{II}^n}{n!}<+\infty.
	\end{equation}
Let $\rho_{II,\alpha,k}$ be the Bergman kernel function of $\hcal_{II,\alpha,k}$.

\

The main result of this section is the following theorem.
\begin{theorem}\label{thm-cusp-model-II}
	Theorem \ref{thm-main} holds for $(\ccal^{II}_{\vartheta,b},\omega_{II})$, with $k_0=n+1$.
\end{theorem}

\subsection{Approximation by complex hyperbolic cylinders}
We fix a $\ccal^{II}_{\vartheta,b}$. By abuse of notation, we also denote by $\vartheta=(\vartheta,0)$, i.e. we set $\vartheta_{n-1}=0$.
Let $N_n:\C^{n-1}\times \C^*\to \C^{n-1}\times \C^*$ be the biholomorphism defined by \begin{equation}
	N_n(z_1,\cdots,z_{n-1},z_n)=(\frac{1}{\sqrt{2}}z_1,\cdots,\frac{1}{\sqrt{2}}z_{n-1},z_n).
\end{equation}
We then identify the complex hyperbolic cylinder $X_{\lambda,-\vartheta}$ with $N_n(X_{\lambda,-\vartheta})$, namely, we consider $X_{\lambda,-\vartheta}$ as the subset of $\C^{n-1}\times \C^*$ defined by \begin{equation}
\cos\left(\frac{\log \lambda}{\pi} t_n\right) - \sum_{j=1}^{n-2} |z_j|^2 e^{\frac{\vartheta_j}{\pi} t_n}-|z_{n-1}|^2>0,
	\end{equation}
where $t_n=\log |z_n|$.

Define $E^{II}_{\vartheta,\lambda,b}:X_{\lambda,-\vartheta}\to \C^{n-1}\times \C^*$ by 
\begin{equation}
	E^{II}_{\vartheta,\lambda,b}(z_1,\cdots,z_{n-1},z_n)=(\frac{b}{\log\lambda}z_1,\cdots,\frac{b}{\log\lambda}z_{n-2},z_n,2(\frac{b}{\log\lambda})^2\frac{z_{n-1}+1}{1-z_{n-1}}).
\end{equation}
\begin{proposition}
	$E^{II}_{\vartheta,\lambda,b}$ is a biholomorphism from $X_{\lambda,-\vartheta}$ to an open subset of $\ccal^{II}_{\vartheta,b}$.
	\end{proposition}
\begin{proof}
Since $w_n=2(\frac{b}{\log\lambda})^2\frac{z_{n-1}+1}{1-z_{n-1}}$, we have \[z_{n-1}=\frac{w_n-2(\frac{b}{\log\lambda})^2}{w_n+2(\frac{b}{\log\lambda})^2}=\frac{2w_n}{w_n+2(\frac{b}{\log\lambda})^2}-1.\] 

Notice that $\tau_{n-1}=t_n$, we have \begin{eqnarray*}
	&&\cos\left(\frac{\log \lambda}{\pi} t_n\right) - \sum_{j=1}^{n-2} |z_j|^2 e^{\frac{\vartheta_j}{\pi} t_n}-|z_{n-1}|^2\\
    &=&1-2\sin^2\left(\frac{\log \lambda}{2\pi} \tau_{n-1}\right) -(\frac{\log \lambda}{b})^2 \sum_{j=1}^{n-2} |w_j|^2 e^{\frac{\vartheta_j}{\pi} \tau_{n-1}}-\big|\frac{2w_n}{w_n+2(\frac{b}{\log\lambda})^2}-1 \big|^2\\
	&=& \frac{8(\frac{b}{\log\lambda})^2\Re w_n}{|w_n+2(\frac{b}{\log\lambda})^2|^2}-2\sin^2\left(\frac{\log \lambda}{2\pi} \tau_{n-1}\right) -(\frac{\log \lambda}{b})^2 \sum_{j=1}^{n-2} |w_j|^2 e^{\frac{\vartheta_j}{\pi} \tau_{n-1}},
	\end{eqnarray*}
	where the last line can be rewritten as
	\begin{multline}\label{eq-cyl-cusp-II}
		(\frac{\log \lambda}{b})^2\bigg( \frac{8(\frac{b}{\log\lambda})^4\Re w_n}{|w_n+2(\frac{b}{\log\lambda})^2|^2}-2(\frac{b}{\log \lambda})^2\sin^2\left(\frac{\log \lambda}{2\pi} \tau_{n-1}\right) 
		- \sum_{j=1}^{n-2} |w_j|^2 e^{\frac{\vartheta_j}{\pi} \tau_{n-1}}\bigg).
	\end{multline}
	Let $x=\frac{\Re w_n}{2(\frac{b}{\log\lambda})^2}$ and $y=\frac{\log \lambda}{2\pi} \tau_{n-1}$, then the proposition follows from Lemma \ref{lem-inequality-sin} below.
	\end{proof}
\begin{lemma}\label{lem-inequality-sin}
	For $x\geq 0, y\geq 0$, if $\frac{x}{(1+x)^2}-\sin^2(y)\geq 0$, then \[\frac{x}{(1+x)^2}-\sin^2(y)\leq x-y^2.\]
\end{lemma}

\begin{proof}
	Suppose $\frac{x}{(1+x)^2}-\sin^2(y)=c>0$, then, $\exists c'>c$ such that \[\frac{x-c'}{(1+(x-c'))^2}= \frac{x}{(1+x)^2}-c.\] Replacing $x$ by $x-c'$, we only need to show that if $\frac{x}{(1+x)^2}-\sin^2(y)=0$, then $x-y^2\geq 0$. So we need to show that $\arcsin(\sqrt{\frac{x}{(1+x)^2}})\leq \sqrt{x}$.
	
	Let $f_1=\arcsin(\sqrt{\frac{x}{(1+x)^2}})$ and $f_2=\sqrt{x}$. Since $f_1(0)=f_2(0)$, we can compare their derivatives:
	\begin{eqnarray*}
		f_1'&=&\frac{1}{\sqrt{1-\frac{x}{(1+x)^2}}}\cdot\frac{1}{2\sqrt{\frac{x}{(1+x)^2}}}\cdot\frac{(1+x)^2-2x(1+x)}{(1+x)^4}\\ &=&\frac{1}{\sqrt{1+x+x^2}}\cdot\frac{1-x}{1+x}\cdot\frac{1}{2\sqrt{x}}\leq \frac{1}{2\sqrt{x}}=f_2'.
	\end{eqnarray*}
	So the conclusion follows.
	
\end{proof}
We denote by $\tilde{X}_{\lambda,II}$ the image of $X_{\lambda,-\vartheta}$ under $E^{II}_{\vartheta,\lambda,b}$. Let $\tilde{\omega}_{\lambda,II}$ be the pull back of $\omega_{\lambda,-\vartheta}$ by $(E^{II}_{\vartheta,\lambda,b})^{-1}$. We define \[\tilde{h}_{\lambda,I}=\bigg( \frac{8(\frac{b}{\log\lambda})^4\Re w_n}{|w_n+2(\frac{b}{\log\lambda})^2|^2}-2(\frac{b}{\log \lambda})^2\sin^2\left(\frac{\log \lambda}{2\pi} \tau_{n-1}\right) 
		- \sum_{j=1}^{n-2} |w_j|^2 e^{\frac{\vartheta_j}{\pi} \tau_{n-1}}\bigg)^2.\] 
	By \eqref{eq-cyl-cusp-II} and Lemma \ref{lem-inequality-sin}, the following proposition is straightforward.
\begin{proposition}\label{prop-h-lambda-tilde}
	On each compact subset $K\subset \ccal^{II}_{\vartheta,b}$, we have that $\tilde{h}_{\lambda,II}$ converges to $h_{II}=\kappa_{II}^2$ in the $C^5$-topology as $\lambda\to 1^+$. We also have \[\tilde{h}_{\lambda,II}\leq h_{II}\] at points in $\tilde{X}_{\lambda,II}$.
\end{proposition}

This proposition implies that $\tilde{\omega}_{\lambda,II}$ converges to $\omega_{II}$ in $C^3(K)$ topology for each compact subset $K\subset \ccal^{II}_{\vartheta,b}$.

Under the indentification $N_n$, we have \[\frac{\tilde{\omega}_{\lambda,II}^n}{n!}=\frac{2^{n-2}(\log\lambda)^2e^{\frac{t_n}{\pi}\sum_{j=1}^{n-2}\vartheta_j }}{\pi^2\big(\cos\left(\frac{\log \lambda}{\pi} t_n\right) - \sum_{j=1}^{n-2} |z_j|^2 e^{\frac{\vartheta_j}{\pi} t_n}-|z_{n-1}|^2 \big)^{n+1}|z_n|^2}|dz|^2. \]
Since \[dz_{n-1}=\frac{4(\frac{b}{\log\lambda})^2}{(w_n+2(\frac{b}{\log\lambda})^2)^2}dw_n,\]we have \[|dz_{n-1}|^2=(\frac{\log\lambda}{b})^4\big|1+\frac{(\log\lambda)^2}{2b^2}w_n \big|^{-2}|dw_n|^2 \]
Thus,
by \eqref{eq-cyl-cusp-II}, we have 
 \begin{equation}
	\frac{\tilde{\omega}_{\lambda,II}^n}{n!}=\frac{2^{n-2}b^2e^{\frac{\tau_{n-1}}{\pi}\sum_{j=1}^{n-2}\vartheta_j }}{\pi^2\tilde{h}_{\lambda,I}^{(n+1)/2}}\big|1+\frac{(\log\lambda)^2}{2b^2}w_n \big|^{-2}\frac{|dw|^2}{|w_{n-1}|^2}
\end{equation}

 Let $\tilde{\rho}_{\lambda,II,\alpha,k}$ denote the Bergman kernel function of $(\tilde{X}_{\lambda,II},\tilde{h}_{\lambda,II}^k|w_{n-1}|^{-2\mathfrak{m}_k} ,\frac{\tilde{\omega}_{\lambda,II}^n}{n!})$, and $\rho_{\lambda,II,\alpha,k}$ denote the Bergman kernel function of $(X_{{\lambda,-\vartheta}},h_{\lambda,-\vartheta}^k|z_n|^{-2\mathfrak{m}_k} ,\frac{\omega_{\lambda,-\vartheta}^n}{n!})$. Then,  we have \begin{equation}\label{eqn-rho-tilde-rho-II}
	\tilde{\rho}_{\lambda,II,\alpha,k}(E^{II}_{\vartheta,\lambda,b}(p))= \rho_{\lambda,II,\alpha,k}(p),
\end{equation}for each $p\in X_{\lambda,-\vartheta}$.

\begin{theorem}\label{thm-rho-tilde-convergence-II}
	Suppose that $2k\geq n+1$.
	For each point $p\in \ccal^{II}_{\vartheta,b}$, we have 
	\begin{equation}
		\lim_{\lambda\to 0^+}\tilde{\rho}_{\lambda,II,\alpha,k}(p)=\rho_{II,\alpha,k}(p).
	\end{equation}
\end{theorem}
\begin{proof}
	Since $\Re w_n>0$, we have $\big|1+\frac{(\log\lambda)^2}{2b^2}w_n \big|^{-2}<1$.
	Then the proof of this theorem completely parallel the proof of Theorem \ref{thm-rho-tilde-convergence-I}.
	
\end{proof}

For each point $p\in \ccal^{II}_{\vartheta,b}$ and each $m\in \Z^*$, it is easy to see that there is exactly one oriented geodesic loop $\gamma_{p,m}^{II}$, with respect to $(\ccal^{II}_{\vartheta,b},\omega_{II})$, based at $p$ such that the winding number of $\gamma$ at $p$ is $m$. More precisely, let $\tilde{p}$ be a lift of $p$ in $\mathfrak{H}_n$, then $\gamma_{p,m}^{II}$ is the projection of the geodesic segment $[\tilde{p},(A^{II}_{\vartheta,b})^m(\tilde{p})]$ to $\ccal^{II}_{\vartheta,b}$. Similarly, we also denote by $\gamma^{II}_{p,m,\lambda}$ the oriented geodesic loop, with respect to $(\tilde{X}_{\lambda,II},\tilde{\omega}_{\lambda,II})$, based at $p$ such that the winding number of $\gamma_{p,m,\lambda}$ at $p$ is $m$.
\begin{theorem}\label{thm:geodesic_loop-II}
	For each point $p\in \ccal^{II}_{\vartheta,b}$ and each $m\in \Z^*$, $\gamma^{II}_{p,m,\lambda}$ converges to $\gamma_{p,m}^{II}$ as $\lambda\to 1^+$.
\end{theorem}
\begin{proof}
	We only show that the lengh of $\gamma^{II}_{p,m,\lambda}$ converges to the length of $\gamma_{p,m}^{II}$ as $\lambda\to 1^+$. The rest is the same as the proof of Theorem \ref{thm:geodesic_loop-I}.
	
	Let $p=(w',w_n)$ and let $\tilde{p}=(z',z_n)$ be a lift of $p$ in $\mathfrak{H}_n$. Let $\ell_{II,m}$ denote the length of the geodesic segment $[\tilde{p},(A^{II}_{\vartheta,b})^m(\tilde{p})]$. We have 
	\begin{eqnarray*}
		&&\cosh\frac{\ell_{II,m}}{2}\\
		&=&\frac{\big|z_n+\bar{z}_n+mbi\bar{z}_{n-1}+\frac12m^2b^2-\sum_{j=1}^{n-2}e^{-im\vartheta_j}|z_j|^2 -z_{n-1}(\bar{z}_{n-1}-mbi)\big|}{2\Re z_n-\sum_{j=1}^{n-1}|z_j|^2}\\
		&=&\frac{\big|2\Re z_n -\sum_{j=1}^{n-1}e^{-im\vartheta_j}|z_j|^2 +\frac12m^2b^2+2mbi\Re z_{n-1} \big|}{2\Re z_n-\sum_{j=1}^{n-1}|z_j|^2 }\\ 
		&=&\frac{\big|2\Re w_n-\frac{b^2}{2\pi^2}\tau_{n-1}^2-\sum_{j=1}^{n-2}e^{-im\vartheta_j}|w_j|^2e^{\frac{\vartheta_j}{\pi}\tau_{n-1}} +\frac12m^2b^2+\frac{mb^2}{\pi}\tau_{n-1}i \big|}{\kappa_{II} }.
	\end{eqnarray*}
	
	Let $q\in X_{\lambda,-\vartheta}$ be the image of $p$ under the inverse of $E^{II}_{\vartheta,\lambda,b}$. Let $q=(x',x_n)$ and let $\tilde{q}=(z',z_n)$ be a lift of $q$ in $\mathfrak{H}_n$.
	Let $\ell_{m,\lambda,II}$ denote the length of the geodesic segment $[\tilde{q},A_{\lambda,-\vartheta}^m(\tilde{q})]$. We use the notation $z_n=|z_n|e^{i\theta_n}$ and $t_n=\log |x_n|$.

	By \eqref{eqn-length-gamma-m}, we have 
	\begin{eqnarray*}
		\cosh\frac{\ell_{m,\lambda}}{2}&=&\frac{\big|\lambda^{2m}\bar{z}_n+ z_n-\lambda^m \sum_{j=1}^{n-1}e^{im\vartheta_j}|z_j|^2 \big|}{\lambda^m(2\Re z_n-\sum_{j=1}^{n-1}|z_j|^2)}\\
		&=&\frac{\big|\lambda^{m}e^{-i\theta_n} +\lambda^{-m} e^{i\theta_n}-2\sum_{j=1}^{n-1}e^{im\vartheta_j}|x_j|^2e^{\frac{\vartheta_j}{\log \lambda}\theta_n}  \big|}{2\cos(\frac{\log\lambda}{\pi}t_n)-2\sum_{j=1}^{n-1}|x_j|^2e^{\frac{\vartheta_j}{\pi}t_n} }\\
		 &=&\frac{|f(\lambda,w)|}{\cos(\frac{\log\lambda}{\pi}t_n)-\sum_{j=1}^{n-1}|x_j|^2e^{\frac{\vartheta_j}{\pi}t_n} },
	\end{eqnarray*}
where $f(\lambda,w)$ equals \[\cosh(m\log\lambda)\cos(\frac{\log\lambda}{\pi}t_n)-i\sinh(m\log\lambda)\sin(\frac{\log\lambda}{\pi}t_n)-\sum_{j=1}^{n-1}e^{im\vartheta_j}|x_j|^2e^{\frac{\vartheta_j}{\pi}t_n} .\]
Since \[|x_{n-1}|^2=1-\frac{8(\frac{\log\lambda}{b})^2\Re w_n}{\big| w_n+2(\frac{\log\lambda}{b})^2 \big|^2}, \]
and \begin{eqnarray*}
	\cosh(m\log\lambda)\cos(\frac{\log\lambda}{\pi}t_n)-1=2\sinh^2(\frac{m\log\lambda}{2})-2\cosh(m\log\lambda)\sin^2(\frac{\log\lambda}{2\pi}t_n),
\end{eqnarray*}
we have \begin{multline}\label{eqn-f-lambda-w-II}
	f(\lambda,w)=\frac{8(\frac{b}{\log\lambda})^2\Re w_n}{\big| w_n+2(\frac{b}{\log\lambda})^2 \big|^2}+2\sinh^2(\frac{m\log\lambda}{2})-2\cosh(m\log\lambda) \sin^2(\frac{\log\lambda}{2\pi}\tau_{n-1})\\ -i\sinh(m\log\lambda)\sin(\frac{\log\lambda}{\pi}\tau_{n-1})-(\frac{\log\lambda}{b})^2\sum_{j=1}^{n-1}e^{im\vartheta_j}|w_j|^2e^{\frac{\vartheta_j}{\pi}\tau_{n-1}}. 
\end{multline}
Thus \begin{multline*}
	\lim_{\lambda\to 1^+}\frac{f(\lambda,w)}{(\frac{\log\lambda}{b})^2}=2\Re w_n+\frac12 m^2 b^2-\frac{b^2}{2\pi^2}\tau_{n-1}^2-i\frac{mb^2}{\pi}\tau_{n-1}- \sum_{j=1}^{n-1}e^{im\vartheta_j}|w_j|^2e^{\frac{\vartheta_j}{\pi}\tau_{n-1}}.
\end{multline*}
So, by \eqref{eq-cyl-cusp-II}, we get that the lengh of $\gamma^{II}_{p,m,\lambda}$ converges to the length of $\gamma_{p,m}^{II}$ as $\lambda\to 1^+$.

\end{proof}

\begin{proof}[Proof of Theorem \ref{thm-cusp-model-II}] 
	The theorem follows from Theorems \ref{thm-rho-tilde-convergence-II}, \ref{thm:geodesic_loop-II} and.
Fix a point $p\in \ccal^{II}_{\vartheta,b}$. Since \[\frac{2\sinh^2(\frac{m\log\lambda}{2})}{(\frac{\log\lambda}{b})^2}\geq \frac12 m^2b^2,\]
	 \eqref{eqn-f-lambda-w-II} implies that there exist $ \lambda_0>0, c>0$ independent of $m$ and $m_0$ such that \[\cosh\frac{\ell_{m,\lambda}}{2}>cm^2,\] for $m\geq m_0$ and $\lambda\leq \lambda_0$. Then, since $\sum m^{-4k}$ converges, dominated convergence theorem implies that \[\lim_{\lambda\to 1^+}\cosh^{-2k}\frac{\ell_{m,\lambda}}{2}e^{-2\pi k\alpha_{p,\lambda} i}=\cosh^{-2k}\frac{\ell_{m}}{2}e^{-2\pi k\alpha_{p} i}. \]
\end{proof}

\section{Proof of Theorem \ref{thm-main}}\label{sec-proof-main}
Although Lu-Zelditch \cite{lu2016szegHo} and Ma-Marinescu \cite{ma2015exponential} have proved the formula in Definition and Theorem \ref{def-and-thm-k-gamma} in the compact case, we give a complete proof here since we do not assume compactness, and we need the precise lower bound for $k$. 

Let $(X,\omega,L,h)$ satisfy the setting of Theorem \ref{thm-main}, and let $\nabla$ be the Chern connection of $(L,h)$.
\subsection{Back to the unit ball}\label{subsec-back-unit-disk}
Let $\pi:\B^n\to X$ be a universal holomorphic locally isometric covering map. Let $\Lambda$ be the Fuchsian group such that $X\cong \B^n/\Lambda$. Let $L_x$ be the fiber of $L$ at $x\in X$.
Then $\pi^*L$ is defined as the subset \[ \pi^*L=\{(p,v)\in \B^n\times L\big| v\in L_{\pi(p)}\}.\] The Hermitian metric $\pi^*h$ on $\pi^*L$ is the pull-back of $h$ via the projection $\B^n\times L\to L$.

Since $\pi^*L$ is a holomorphic line bundle on $\B^n$, it is trivial. Let $\bm{e}'$ be a holomorphic frame of $\pi^*L$ on $\B^n$. Then $\parallel \bm{e}'\parallel^2_{\pi^*h}=e^{-\varphi}$ for some smooth function $\varphi$ on $\B^n$ satisfying $\iddbar \varphi=\pi^*\omega=\omega_0$. Let $\phi_0=-2\log(1-|z|^2)$. Then $\varphi-\phi_0$ is pluriharmonic. So there exists a holomorphic function $F$ on $\B^n$ such that $\varphi-\phi_0=\Re F$. Let $\bm{e}=e^{F/2}\bm{e}'$, then \[\parallel \bm{e}\parallel^2_{\pi^*h}=e^{-\phi_0}.\]
Similarly, for any $k\in \mathbb{N}$, we can define a frame $\bm{e}_k$ of $\pi^*L^k$ on $\B^n$ such that $\parallel \bm{e}_k\parallel^2_{\pi^*h}=e^{-k\phi_0}$.

Each $g\in \Lambda$ is lifted to a map \[\tilde{g}:\pi^*L^k\to \pi^*L^k, \quad (x,v)\mapsto (g(x),v).\] 
So for any $s\in H^0(\B^n,\pi^*L^k)$, we can define the pullback map $g^*:H^0(\B^n,\pi^*L^k)\to H^0(\B^n,\pi^*L^k)$ by 
\begin{equation}
	g^*s(x)=\tilde{g}^{-1} s(g(x)).
\end{equation}
If $f\in H^0(\B^n,\pi^*L^k)$, then $f=U\bm{e}_k$ for some holomorphic function $U$. Then we have
\[g^*s(x)=(g^*U\frac{g^*\bm{e}_k}{\bm{e}_k})\bm{e}_k,\]where $g^*U=U\circ g$.

To distinguish different Bergman spaces, we use the following notation: let \[\hcal_{X,k}=\{s\in H^0(X,L^k): \int_X |s|^2_h \frac{\omega^n}{n!}<\infty\},\] and \[\hcal_{\B^n,k}=\{s\in H^0(\B^n,\pi^*L^k): \int_{\B^n} |s|^2_{\pi^*h} \frac{\omega_0^n}{n!}<\infty\}.\] Let $K_{X,k}$ and $K_k$ be the Bergman kernels of $\hcal_{X,k}$ and $\hcal_{\B^n,k}$ respectively. Let $\rho_{X,k}$ and $\rho_{k}$ be the corresponding Bergman kernel functions.
Let $V_k=\pi^*\hcal_{X,k}\subset H^0(\B^n,\pi^*L^k)$. Clearly, for $s\in V_k$, $g^*s=s, \forall g\in \Lambda$. Let $F_\Lambda$ be a fundamental domain of the action of $\Lambda$ on $\B^n$. Then for any $s\in H^0(\B^n,\pi^*L^k)$ that is invariant under the action of $\Lambda$, $s\in V_k$ if and only if \[\int_{F_\Lambda} |s|^2\frac{\omega_0^n}{n!}<\infty.\] 
With the $L^2$-norm defined by integrating over $F_\Lambda$, $V_k$ is then a Hilbert space.
$\pi^*K_{X,k}$ is then the Bergman kernel of $V_k$ in the sense that, for any $s\in V_k$, \[\int_{F_\Lambda} (s(z),\pi^*K_{X,k}(z,w))_{\pi^*h}\frac{\omega_0^n}{n!}(z)=s(w),\]
for any $w\in \B^n$.

\subsection{Summation over $\Lambda$.}

For simplicity, we write $h$ for $\pi^*h$.

Recall that the critical exponent $\delta(\Lambda)$ of a Fuchsian group $\Lambda$ is defined as following:

Let $\ocal_0$ be the orbit of $0$ under the action of $\Lambda$ on $\B^n$.
Let $\lambda_N$ be the number of points $p\in \ocal_0$ in the half open shell of radii in $(N-\frac{1}{2},N+\frac{1}{2}]$. Then,
\[
\delta(\Lambda)\triangleq \limsup_{N\to\infty}\frac{\log\lambda_N}{N}.
\]
So for any $\epsilon>0$, there exists $C_1>0$ and $N_1$ such that for any $N>N_1$,  \begin{equation}\label{e-lambda-n}
	\lambda_N\leq C_1 e^{(\delta(\Lambda)+\epsilon)N}.
\end{equation}

We can give $\Lambda$ an order according to the distance from $0$ to $g(0)$ for $g\in \Lambda$.
\begin{lemma}\label{lem-sum-g-f-0}
	Let $f_0=\bm{e}_k$. Then when $k>\delta(\Lambda)$, $\sum_{g\in \Lambda}g^*f_0$ converges locally uniformly and absolutely.
\end{lemma}
\begin{proof}
	We have $|f_0(z)|_h=(1-|z|^2)^k=\cosh^{-2k}\frac{d(z,0)}{2}$. For a fixed radius $d_0$, we have 
	\[|g^*f_0(w)|_h\leq \cosh^{-2k}\frac{d(g(0),0)-d_0}{2},\]
	for $w\in B_0(d_0)$, when $d(g(0),0)>d_0$.

	Since $\cosh x>\frac12 e^x$, we obtain
	\[|g^*f_0(w)|_h\leq 2^k e^{-2k\frac{d(g(0),0)-d_0}{2}}.\]
	Thus, by \eqref{e-lambda-n},
	\[\sum_{N\leq d(g(0),0)<N+1}|g^*f_0(w)|_h\leq C_2e^{(\delta(\Lambda)+\epsilon)N-kN},\] for $N$ large enough.
	Then it is clear that when $k>\delta(\Lambda)$, $\sum_{g\in \Lambda}g^*f_0$ converges locally uniformly and absolutely.

\end{proof}
So $\sum_{g\in \Lambda}g^*f_0$ is a holomorphic section of $\pi^*L^k$ on $\B^n$. Clearly $\sum_{g\in \Lambda}g^*f_0$ is invariant under the action of $\Lambda$.
\begin{lemma}\label{lem-sum-g-f-0-in-v-k}
	For $k>2\delta(\Lambda)+\frac{n}{2}$ ($k>\delta(\Lambda)+\frac{n}{2}$, when $X$ is compact), we have $\sum_{g\in \Lambda}g^*f_0\in V_k$. 
\end{lemma}
\begin{proof}
	Write $x_0=0$ and $x_g=g(0)$. 
	Recall the Dirichlet fundamental domain $F$ is the intersection of the closed half-spaces
\[
F=\bigcap_{\gamma\in\Lambda}\{x:\;d(x,x_0)\le d(x,\gamma x_0)\}.
\]
To show that $\sum_{g\in \Lambda}g^*f_0\in V_k$, it suffices to show that \[\sum_{g\in \Lambda}(\int_F |g^*f_0|_h^2\frac{\omega_0^n}{n!})^{1/2}<\infty. \]	
Let $r(g)=\inf_{x\in F}d(x_g,x)$.
	Let $H_g$ denote the perpendicular bisector of the geodesic segment $[x_0,x_g]$; by construction
\[
d(x_g,H_g)=\tfrac12\,d(x_0,x_g).
\]
Since $F\subset H_g^+$, the half-space containing $x_0$, we have
\[
r(g)\;\ge\;\inf_{x\in H_g^+}d(x_g,x)=d(x_g,H_g)=\tfrac12\,d(x_0,x_g).
\]
Since \[\int_{d(x,0)>r}|f_0(x)|_h^2\frac{\omega_0^n}{n!}<c\cosh^{-4k+2n}\frac{r}{2}, \] for some constant $c>0$ independent of $r$,
the same argument, using critical exponent, as in the proof of Lemma \ref{lem-sum-g-f-0} shows that, for $k>2\delta(\Lambda)+\frac{n}{2}$,
\[\sum_{g\in \Lambda}(\int_F |g^*f_0|_h^2\frac{\omega_0^n}{n!})^{1/2}<\infty.\]
When $X$ is compact, we can take a compact fundamental domain $F_\Lambda$. Then $\exists C>0$ such that $r_g>d(x_0,x_g)-C$.
Then, the same argument shows that for $k>\delta(\Lambda)$, we have $\sum_{g\in \Lambda}g^*f_0\in V_k$.

\end{proof}
Since $\sqrt{\frac{1}{(4\pi)^n} \frac{(2k-1)!}{(2k-n-1)!}}f_0$ is the peak section at $0$ of $(\B^n,\pi^*L^k,\pi^*h^k)$, Lemmas \ref{lem-sum-g-f-0} and \ref{lem-sum-g-f-0-in-v-k} apply to the peak section $s_p$ for any $p\in \B^n$.

\begin{lemma}\label{lem-int-s-sum-g-s-p}
	Let $s_p$ be a peak section at $p$ of $(\B^n,\pi^*L^k)$. Then for $k>2\delta(\Lambda)+\frac{n}{2}$ ($k>\delta(\Lambda)+\frac{n}{2}$, when $X$ is compact), and for any $s\in V_k$, we have 
	\[\int_{F_\Lambda}(s,\sum_{g\in \Lambda}g^*s_p)_{\pi^*h}\frac{\omega_0^n}{n!}=\frac{s(p)}{s_p(p)}.\]
\end{lemma}
\begin{proof}
	It suffices to show this for $p=0$. By Lemmas \ref{lem-sum-g-f-0} and \ref{lem-sum-g-f-0-in-v-k}, 
	\[ \int_{F_\Lambda}(s,\sum_{g\in \Lambda}g^*s_0)_{\pi^*h}\frac{\omega_0^n}{n!}=\sum_{g\in \Lambda}\int_{F_\Lambda}(s,g^*s_0)_{\pi^*h}\frac{\omega_0^n}{n!}.\]
	On the other hand, we claim that $$(s,s_0)_{\pi^*h}\in L^1(\B^n,\frac{\omega_0^n}{n!}).$$ Assume that this claim holds, then \[\sum_{g\in \Lambda}\int_{F_\Lambda}(s,g^*s_0)_{\pi^*h}\frac{\omega_0^n}{n!}=\int_{\B^n} (s,s_0)_{\pi^*h}\frac{\omega_0^n}{n!}.\]

 Then by taking the Taylor expansion of $\frac{s}{\bm{e}_k}$, we get the right hand side equals $(\frac{1}{(4\pi)^n} \frac{(2k-1)!}{(2k-n-1)!})^{-1/2}\frac{s(0)}{\bm{e}_k(0)}=\frac{s(0)}{s_0(0)}$.

 Now we prove the claim. For each $g\in \Lambda$, we have \[\int_{gF_\Lambda}|(s,s_0)_{\pi^*h}|\frac{\omega_0^n}{n!}\leq (\int_{gF_\Lambda} |s_0|^2_{\pi^*h}\frac{\omega_0^n}{n!})^{1/2}(\int_{F_\Lambda} |s|^2_{\pi^*h}\frac{\omega_0^n}{n!})^{1/2}. \]
 So, to show that $(s,s_0)_{\pi^*h}\in L^1(\B^n,\frac{\omega_0^n}{n!})$, it suffice to show that \[\sum_{g\in \Lambda}(\int_{gF_\Lambda} |s_0|^2_{\pi^*h}\frac{\omega_0^n}{n!})^{1/2}<+\infty. \]Since \[\int_{gF_\Lambda} |s_0|^2_{\pi^*h}\frac{\omega_0^n}{n!}=\int_{F_\Lambda} |g^*s_0|^2_{\pi^*h}\frac{\omega_0^n}{n!},\] we only need to show that 
 \[\sum_{g\in \Lambda}(\int_{F_\Lambda} |g^*s_0|^2_{\pi^*h}\frac{\omega_0^n}{n!})^{1/2}<+\infty, \] which has been proved in the proof of Lemma \ref{lem-sum-g-f-0-in-v-k}.
\end{proof}

Let $K_k(z,w)$ be the Bergman kernel of $\hcal_k$. Then we can write \[K_k(z,w)=\sum_{i=1}^{\infty}f_i(z)\otimes \bar{f}_i(w),\] where $\{f_i\}_{i=1}^\infty$ is an orthonormal basis of $\hcal_{k}$.
For each $g\in \Lambda$, we define \[K^g_k(z,w)=\sum_{i=1}^{\infty}g^*f_i(z)\otimes \bar{f}_i(w).\]
\begin{proposition}
	$K^g_k(z,w)=\sum_{i=1}^{\infty}g^*f_i(z)\otimes \bar{f}_i(w)$ is independent of the choice of orthonormal basis $\{f_i\}_{i=1}^\infty$ of $\hcal_{k}$.
\end{proposition}
\begin{proof}
	For any $f\in \hcal_{k}$, we have $f=\sum_{i=1}^{\infty}c_i g^*f_i$, since $\{g^*f_i\}$ is also an orthonormal basis of $\hcal_{k}$. Then for any fixed $w$, \[\int_\B^n (f,\sum_{i=1}^{\infty}g^*f_i(z)\otimes \bar{f}_i(w))_{\pi^*h}\frac{\omega_0^n}{n!}=\sum_{i=1}^{\infty}c_i f_i(w)=(g^{-1})^*f(w), \]
	namely, $K^g_k(z,w)$ is the representation element for the functional $\val_w\circ (g^{-1})^*$, hence independent of the choice of the orthonormal basis.
\end{proof}
\begin{defandthm}\label{def-and-thm-k-gamma}
	We define
\[K_{k}^\Lambda(z,w)\eqd \sum_{g\in \Lambda}K^g_k(z,w). \]
Then for $k>\delta(\Lambda)$, and for each fixed $w$, it converges locally uniformly and absolutely. Moreover, we have that for $k>2\delta(\Lambda)+\frac{n}{2}$ ($k>\delta(\Lambda)+\frac{n}{2}$, when $X$ is compact), $\sum_{g\in \Lambda}K^g_k(z,w)\in V_k$ for each fixed $w$.
\end{defandthm}
\begin{proof}
	Since $K^g_k(z,w)$ is independent of the choice of orthonormal basis, for each fixed $w$, we can choose the orthonormal basis $\{f_i\}$ such that $f_1=s_w$, the peak section at $w$. So $f_i(w)=0$ for $i>1$. So we have \[K^g_k(z,w)=g^*s_w(z)\otimes \bar{s}_w(w).\]
	So by Lemma \ref{lem-sum-g-f-0}, we have $\sum_{g\in \Lambda}K^g_k(z,w)$ converges locally uniformly and absolutely. And by Lemma \ref{lem-sum-g-f-0-in-v-k}, $\sum_{g\in \Lambda}K^g_k(z,w)\in V_k$ for each fixed $w$.
\end{proof}

\begin{theorem}\label{thm-k-gamma-bergman}
	For $k>2\delta(\Lambda)+\frac{n}{2}$ ($k>\delta(\Lambda)+\frac{n}{2}$, when $X$ is compact), $K_{k}^\Lambda(z,w)$ is the Bergman kernel of $V_k$.
\end{theorem}
\begin{proof}
	By the proof in Definition and Theorem \ref{def-and-thm-k-gamma}, for fixed $w$, we have \[K_{k}^\Lambda(z,w)=(\sum_{g\in \Lambda}g^*s_w(z))\otimes \bar{s}_w(w).\]
	So for each $S\in \hcal_{X,k}$, by Lemma \ref{lem-int-s-sum-g-s-p} we have \[\int_{F_\Lambda} (S,K_{k}^\Lambda(z,w))_{\pi^*h}\frac{\omega_0^n}{n!}=\frac{S(w)}{s_w(w)}s_w(w)=S(w).\]
	We have proved the theorem.

\end{proof}
\subsection{Proof of the main results}
\subsubsection{Proof of Theorem \ref{thm-main}}
\begin{proof}
    For a point \(p\in X\) choose a lift \(\tilde p\in\B^n\); without loss of generality we may take \(\tilde p=0\). By Theorem \ref{thm-k-gamma-bergman} we have the identity
    \[
    \rho_{X,k}(0)=K_{k}^\Lambda(0,0)
    = \frac{1}{(4\pi)^n} \frac{(2k-1)!}{(2k-n-1)!} \;+\; \sum_{g\in\Lambda\setminus\{1\}} K^g_k(0,0).
    \]
	Recall that an element $g\in \Lambda$ is primitive if it can not be written as $g=a^m$ for some $a\in\Lambda$ and $m>1$.
    Let \(\mathcal P\) denote the set of primitive elements of \(\Lambda\). Writing \(\langle g\rangle=\{g^m:\,m\in\Z\}\) for the cyclic subgroup generated by \(g\in\mathcal P\), we may reorganise the sum as
    \[
    \sum_{g\in\Lambda\setminus\{1\}} K^g_k(0,0)
    = \sum_{g\in\mathcal P}\sum_{m=1}^\infty K^{g^m}_k(0,0)
    = \frac12\sum_{g\in\mathcal P}\sum_{g'\in\langle g\rangle\setminus\{1\}} K^{g'}_k(0,0).
    \]

    Fix \(g\in\mathcal P\). The quotient \(\B^n/\langle g\rangle\) is either a complex hyperbolic cylinder $X_{\lambda,\vartheta}$ (if \(g\) is hyperbolic) or a complex hyperbolic cusp $\ccal^I_{\vartheta,b}$ or $\ccal^{II}_{\vartheta,b}$ (if \(g\) is parabolic). Let
    \(\pi_g:\B^n\to \B^n/\langle g\rangle\) be the quotient map and \(\chi_g:\B^n/\langle g\rangle\to X\) the natural covering with \(\pi=\chi_g\circ\pi_g\). For each $g'\in \langle g\rangle$, let $e^{2\pi \alpha_{g'}\sqrt{-1}}$ be the holonomy of $L$ along the geodesic loop based at $p$ corresponding to $g$ and $\tilde{p}$. Since the pull-back of a parallel section is still a parallel section, the holonomy of the connection $\chi_g^*\nabla$ on $\chi_g^*L$ along the geodesic loop based at $\pi_g(0)$ corresponding to $g'$ and $\tilde{p}$ is also $e^{2\pi \alpha_{g'}\sqrt{-1}}$.

    Applying Theorems \ref{thm-cyl-model-main}, \ref{thm-main-cusp-model-I} and \ref{thm-cusp-model-II} to the quotient \(\B^n/\langle g\rangle\) (and using Proposition \ref{prop-correspondece} to identify geodesic loops with group elements) yields, for each \(g\in\mathcal P\),
    \[
    \sum_{g'\in\langle g\rangle\setminus\{1\}} K^{g'}_k(0,0)
    = C(n,k) \sum_{g'\in\langle g\rangle\setminus\{1\}}
    \cosh^{-2k}\!\Big(\frac{d(0,g'(0))}{2}\Big)\cos\!\big(2\pi k\alpha_{g'}\big),
    \]
    where, $C(n,k)=\frac{1}{(4\pi)^n} \frac{(2k-1)!}{(2k-n-1)!}$ and \(d(0,g'(0))\) denotes the hyperbolic distance from \(0\) to \(g'(0)\).

    Summing over all primitive elements \(g\in\mathcal P\) gives the formula in Theorem \ref{thm-main}. Finally, since for our Fuchsian group \(\Lambda\) one has \(\delta(\Lambda)\le1\), the convergence and the lower bound on \(k\) are ensured by the hypotheses in Theorems \ref{thm-cyl-model-main}, \ref{thm-main-cusp-model-I} and \ref{thm-cusp-model-II}. So we may take \(k_0=2n+[\frac{n}{2}]+1 \) in general, and take $k_0=n+[\frac{n}{2}]+1 $ when $X$ is compact.
\end{proof}

\subsubsection{Proof of Theorem \ref{thm-off-diag}}

\begin{proof}
	Let $\tilde{x}$ be a lift of $x$ in $\B^n$. Let $\tilde{y}$ be a lift of $y$ in $\B^n$ such that $d(\tilde{x},\tilde{y})=d(x,y)$. By taking a conjugation of $\Gamma$, we can simply assume that $\tilde{y}=0$. Then by Theorem \ref{thm-k-gamma-bergman}, we have 
	\[|K_{X,k}(x,y)|=|\pi^*K_{X,k}(\tilde{x},\tilde{y})|=|\sum_{g\in \Gamma}g^*s_0(x)\otimes s_0(0)|,\]
	where $s_0=\sqrt{C(n,k)}\bm{e}_k$ is the peak section of $\hcal_{\B^n,k}$ at $0$. Since \[|s_0(x)|=\sqrt{C(n,k)}\cosh^{-k}\frac{d(0,x)}{2},\] the same holds for $s_p$ for any $p\in \B^n$. Then, since $g^*s_0=s_{g^{-1}(0)}$, we have $|g^*s_0(\tilde{x})|=\sqrt{C(n,k)}\cosh^{-k}\frac{d(g^{-1}(0),\tilde{x})}{2}$. Therefore, \[|K_{X,k}(x,y)|\leq C(n,k)\sum_{g\in \Gamma}\cosh^{-k}\frac{d(g(0),\tilde{x})}{2}.
	\]So the first part of the theorem follows. 

	When the minimizing geodesic between $x$ and $y$ is unique, we have $|g^*s_0(\tilde{x})|\leq \sqrt{C(n,k)}\cosh^{-k}\frac{l_2}{2}$, for all $g\neq 1$, where $l_2=\min_{\gamma\in \mathfrak{G}_{x,y}, \ell(\gamma)>d(x,y)} \ell(\gamma)$. So for $k$ large enough, we have $\sum_{g\in \Gamma,d(g(0),\tilde{x})>l_1}\cosh^{-k}\frac{d(g(0),\tilde{x})}{2}\leq c \cosh^{-k}\frac{l_2}{2}$, for some $c>0$. So the second part of the theorem follows.
\end{proof}
\subsubsection{Proof of Theorem \ref{thm-max-min}}
For each point $p\in X$, let $\tilde{p}$ be a lift of $p$ in $\D$. For each $R>0$, let \[N_R(\tilde{p})=\#\{g\in \Gamma: d(\tilde{p},g(\tilde{p}))<R\}.\]
Clearly, $N_R(\tilde{p})$ is independent of the choice of the lift $\tilde{p}$, so we can simply write it as $N_R(p)$.
\begin{lemma}\label{lem-n-r-p}
	Suppose that $X$ is a hyperbolic surface with finite topology and without cusps. Then for each $\epsilon>0$, there exists a constant $C>0$ such that\[N_R(p)\leq Ce^{(1+\epsilon)R},\] for all $p\in X$.
\end{lemma}
\begin{proof}
    First, since X has finite topology and no cusps, every end is a funnel and the injectivity radius is uniformly positive. Set
    \[
    \delta:=\inf_{x\in X}\operatorname{inj}(x)>0.
    \]

    Fix \(p\in X\) and let \(\tilde p\in\D\) be a lift. For \(r>0\) denote by \(B_{\tilde p}(r)\) the hyperbolic disk of radius \(r\) centered at \(\tilde p\). If \(1\neq g\in\Gamma\) then
    \[
    g\bigl(B_{\tilde p}(\delta)\bigr)\cap B_{\tilde p}(\delta)=\varnothing,
    \]
    because distinct deck-translates of the radius-\(\delta\) ball are disjoint by the definition of \(\delta\). Moreover, whenever \(d(\tilde p,g\tilde p)<R\) we have
    \[
    g\bigl(B_{\tilde p}(\delta)\bigr)\subset B_{\tilde p}(R+\delta).
    \]
    Counting disjoint translates yields the area bound
    \[
    N_R(\tilde p)\,\mathrm{Area}\bigl(B_{\tilde p}(\delta)\bigr)
    \le \mathrm{Area}\bigl(B_{\tilde p}(R+\delta)\bigr),
    \]
    so
    \[
    N_R(\tilde p)\le
    \frac{\mathrm{Area}\bigl(B_{\tilde p}(R+\delta)\bigr)}
    {\mathrm{Area}\bigl(B_{\tilde p}(\delta)\bigr)}.
    \]

    For the hyperbolic plane (curvature \(-1\)) the disk area is explicit:
    \[
    \mathrm{Area}\bigl(B_{\tilde p}(r)\bigr)=4\pi\sinh^2\!\frac r2
    =2\pi(\cosh r-1).
    \]
    Hence
    \[
    N_R(\tilde p)\le
    \frac{\sinh^2\!\bigl(\tfrac{R+\delta}{2}\bigr)}
    {\sinh^2\!\bigl(\tfrac{\delta}{2}\bigr)}
    \le C\,e^{R}
    \]
    for a constant \(C=C(\delta)>0\). In particular, for any \(\varepsilon>0\) there exists \(C_\varepsilon>0\) with
    \[
    N_R(\tilde p)\le C_\varepsilon e^{(1+\varepsilon)R},
    \]
    as claimed.

\end{proof}
\begin{proof}[Proof of Theorem \ref{thm-max-min}]

	For a point $p\in X$, let $\tilde{p}$ be a lift of $p$ in $\D$. For each $g\in \Gamma$,
	we write $\gamma_{\tilde{p},g}$ for the geodesic loop based at $p$ corresponding to the geodesic segment $[\tilde{p},g(\tilde{p})]$. Let $\gamma$ be a systole. Let $p\in \gamma$, and let $\tilde{p}\in \D$ be a lift of $p$. $\exists g\in \Gamma$ such that $\gamma_{\tilde{p},g}=\gamma$. Denote by \[\Gamma^g=\Gamma\setminus \{1,g,g^{-1}\}. \]
	Since, by assumption, there is no other systole passing through $p$, we have \[\min_{f\in\Gamma^g}d(f(\tilde{p}),\tilde{p})> l_1.\] The same is true for all point $\tilde{q}$ in the geodesic segment $[\tilde{p},g(\tilde{p})]$. For $x\in \D$, we define a function \[A_g(x)=\min_{f\in\Gamma^g}d(f(x),x).\]
By the proper discontinuity of the action of $\Gamma$ on $\D$, it is not hard to see that $A_g(x)$ is a continuous function on $\D$. So \[l_3(\gamma)\eqd \min_{q\in [\tilde{p},g(\tilde{p})]}A_g(q)>l_1.\]
So there exists an open neighborhood $U_{[\tilde{p},g(\tilde{p})]}$ of $[\tilde{p},g(\tilde{p})]$ such that for any $x\in U_{[\tilde{p},g(\tilde{p})]}$, \[A_g(x)>\frac{l_1+l_3(\gamma)}{2}.\] $U_{[\tilde{p},g(\tilde{p})]}$ is then mapped to an open neighborhood $V_{\gamma}$ of $\gamma$ in $X$ by $\pi$. So for each point $q\in V_{\gamma}$, we have that \[\ell(\gamma_{\tilde{q},f})>\frac{l_1+l_3(\gamma)}{2},\]for any $f\in \Gamma^g$, where $\tilde{q}\in [\tilde{p},g(\tilde{p})]$ is a lift of $q$.
We then let \[l_3=\min_{\gamma \text{ is a systole}} l_3(\gamma).\]
Since there is only finite number of systoles on $X$, we have $l_3>l_1$. Let $V=\cup_{\gamma \text{ is a systole}}V_{\gamma}$. 
Then by the uniform control of the distribution of the orbits of $\Gamma$ in Lemma \ref{lem-n-r-p}, one can see that  $\exists C_1>0,k_1>0$ such that for $k\geq k_1$, we have\[
\sum_{\gamma\in \mathfrak{G}_q, \ell(\gamma)>\frac{l_1+l_3}{2}}\cosh^{-2k}\frac{\ell(\gamma)}{2} \leq C_1 \cosh^{-2k}\frac{l_1+l_3}{4},
\] for any $q\in V$. 

And for each $q\in V$, it is not hard to see that the length $\ell_q$ of the shortest geodesic loop based at $q$ satisfies the equality \begin{equation}\label{e-length-gamma}
\cosh^2\frac{\ell_q}{2}=\cosh^2\frac{l_1}{2}\cosh^2  d(q,\tilde{\gamma})-\sinh^2 d(q,\tilde{\gamma}).
\end{equation}
So we define the function $\nu(q)>0 $ by \[\cosh^2\frac{\nu(q)}{2}=\cosh^2\frac{l_1}{2}\cosh^2  d(q,\tilde{\gamma})-\sinh^2 d(q,\tilde{\gamma}),\]for each $q\in V$.
Then we have \begin{equation}
	\rho_k(q)-\frac{k-\frac12}{2\pi}\leq \frac{k-\frac12}{2\pi}(2\cosh^{-2k}\frac{\nu(q)}{2}+C_1 \cosh^{-2k}\frac{l_1+l_3}{4}),
\end{equation}for each $q\in V$. And if $p$ is in some systole, it is then easy to see that 
\begin{equation}
	\rho_k(p)-\frac{k-\frac12}{2\pi}\geq \frac{k-\frac12}{2\pi}(2\cosh^{-2k}\frac{l_1}{2}-C_1 \cosh^{-2k}\frac{l_1+l_3}{4}),
\end{equation}
We claim that there exists $l_4>l_1$ such that $\ell(\gamma)\leq l_4$, for each $\gamma\in \mathfrak{G}_q$ and each $q\in X\setminus V$.

Notice that $l_1$ is also the minimum of $\ell_g$ for $g\in \Gamma$. Let $l_2$ be the second smallest length among $\{\ell_g|g\in \Gamma\}$. 
If $\ell_g\geq l_2$, then for any lift $\tilde{q}\in \D$ of $q$, we have $\ell(\gamma_{\tilde{q},g})\geq l_2$. If $\ell_g= l_1$, then, since $\exists \delta_1>0$ such that $d(q,\gamma)\geq \delta_1$ for all $q\notin V$ and $\gamma$ a systole, we get that the distance of $\tilde{q}$ to the geodesic along which $g$ translates is $\geq \delta_1$. Therefore, by \eqref{e-length-gamma},
$\exists l_5>l_1$ such that $\ell(\gamma_{\tilde{q},g})\geq l_5$. So we have proved the claim.

Therefore, $\exists C_2>0,k_2>0$ such that for $k\geq k_2$, we have\[
\sum_{\gamma\in \mathfrak{G}_q}\cosh^{-2k}\frac{\ell(\gamma)}{2} \leq C_2 \cosh^{-2k}\frac{l_4}{2},
\] for any $q\notin V$.

Now suppose that $\rho_k(q)>\rho_k(p)$ for some $p$ in a systole. Then for $k$ large enough, we  must have $q\in V$ and \[2\cosh^{-2k}\frac{l_1}{2}\leq 2\cosh^{-2k}\frac{\nu(q)}{2}+2C_1 \cosh^{-2k}\frac{l_1+l_3}{4}.\]

Thus, let $\gamma$ be the systole that is closest to $q$, we have \[1-(1-\tanh^2\frac{l_1}{2}\sinh^2 d(q,\gamma))^k\leq C_1(\cosh^{2k}\frac{l_1}{2})/(\cosh^{2k}\frac{l_1+l_3}{4}). \]
So we have \[\frac{k}{2}\tanh^2\frac{l_1}{2}\sinh^2 d(q,\gamma)\leq C_1(\cosh^{2k}\frac{l_1}{2})/(\cosh^{2k}\frac{l_1+l_3}{4}). \]

When $k$ is large enough, $\frac{2C_1}{k\tanh^2\frac{l_1}{2}}<1$. This completes the proof of the maximum part of the theorem. The proof of the minimum part is almost the same.
\end{proof}

\begin{proposition}\label{prop-kx}
	When $L=K_X$, the holonomy along a simple closed geodesic is $e^{2\pi \alpha\sqrt{-1}}=1$.
\end{proposition}
\begin{proof}
	We continue to use the notations in the proof of Theorem \ref{thm-main}.
	Let $\gamma$ be a simple closed geodesic in $X$. Let $g\in \Gamma$ be an element corresponding to $\gamma$ and the lift $\tilde{p}$. Since $\gamma$ is simple, $g$ is hyperbolic and $\gamma$ is the image of the geodesic along which $g$ translates.  Then
	$\D/\langle g\rangle$ is a hyperbolic cylinder.
	 Then $\chi_g:\D/\langle g\rangle\to X$ is the natural covering map.  Let $\pi_g:\D\to \D/\langle g\rangle$ be the quotient map such that $\pi=\chi_g\circ \pi_g$.
	Since the metric on $K_X$ is defined by $\omega$, the metric $\chi_g^*h$ on $\chi_g^*K_X$ is given by 
	\[
	\chi_g^*h=\frac{idz\wedge d\bar{z}}{\omega_\eta}=\frac{|z|^2\cos^2(\eta\log|z|)}{\eta^2}.
	\]
	Let $\phi=-\log(|z|^2\cos^2(\eta\log|z|)) $. Then on $\{|z|=1\}$, $\partial \phi=-\frac{dz}{z}+0$. So 
	\[
	\int_{|z|=1}\partial \phi=\pm 2\pi i.
	\]
	And we have proved the proposition.
\end{proof}

\bibliographystyle{plain}


\bibliography{references}

\begin{thebibliography}{10}

\bibitem{aryasomayajula2018off}
Anilatmaja Aryasomayajula and Priyanka Majumder.
\newblock Off-diagonal estimates of the bergman kernel on hyperbolic riemann surfaces of finite volume.
\newblock {\em Proceedings of the American Mathematical Society}, 146(9):4009--4020, 2018.

\bibitem{aryasomayajula2020estimates}
Anilatmaja Aryasomayajula and Priyanka Majumder.
\newblock Estimates of the bergman kernel on a hyperbolic riemann surface of finite volume-ii.
\newblock In {\em Annales de la Facult{\'e} des sciences de Toulouse: Math{\'e}matiques}, volume~29, pages 795--804, 2020.

\bibitem{AMM}
Hugues Auvray, Xiaonan Ma, and George Marinescu.
\newblock Bergman kernels on punctured riemann surfaces.
\newblock {\em Mathematische Annalen}, 379(3):951--1002, April 2021.

\bibitem{berman2011fekete}
Robert Berman, S{\'e}bastien Boucksom, and David~Witt Nystr{\"o}m.
\newblock Fekete points and convergence towards equilibrium measures on complex manifolds.
\newblock {\em Acta Mathematica}, 207(1):1--27, 2011.

\bibitem{berman2012sharp}
Robert~J Berman.
\newblock Sharp asymptotics for toeplitz determinants and convergence towards the gaussian free field on riemann surfaces.
\newblock {\em International Mathematics Research Notices}, 2012(22):5031--5062, 2012.

\bibitem{bsz0}
Pavel Bleher, Bernard Shiffman, and Steve Zelditch.
\newblock Poincar\'{e}-{L}elong approach to universality and scaling of correlations between zeros.
\newblock {\em Comm. Math. Phys.}, 208(3):771--785, 2000.

\bibitem{bsz1}
Pavel Bleher, Bernard Shiffman, and Steve Zelditch.
\newblock Universality and scaling of correlations between zeros on complex manifolds.
\newblock {\em Invent. Math.}, 142(2):351--395, 2000.

\bibitem{boutet1975singularite}
Louis Boutet~de Monvel and Johannes Sj{\"o}strand.
\newblock Sur la singularit{\'e} des noyaux de bergman et de szeg{\"o}.
\newblock {\em Journ{\'e}es {\'e}quations aux d{\'e}riv{\'e}es partielles}, pages 123--164, 1975.

\bibitem{Catlin}
David Catlin.
\newblock The {B}ergman kernel and a theorem of {T}ian.
\newblock In {\em Analysis and geometry in several complex variables ({K}atata, 1997)}, Trends Math., pages 1--23. Birkh\"{a}user Boston, Boston, MA, 1999.

\bibitem{Charles2003}
Laurent Charles.
\newblock Berezin-toeplitz operators, a semi-classical approach.
\newblock {\em Communications in Mathematical Physics}, 239(1):1--28, Aug 2003.

\bibitem{chen1974hyperbolic}
Shyan~S Chen and Leon Greenberg.
\newblock Hyperbolic spaces.
\newblock In {\em Contributions to analysis}, pages 49--87. Elsevier, 1974.

\bibitem{dailiuma}
Xianzhe Dai, Kefeng Liu, and Xiaonan Ma.
\newblock On the asymptotic expansion of bergman kernel.
\newblock {\em Journal of Differential Geometry}, 72(1):1--41, 2006.

\bibitem{datar2023kahler}
Ved Datar, Xin Fu, and Jian Song.
\newblock K{\"a}hler--einstein metrics near an isolated log-canonical singularity.
\newblock {\em Journal f{\"u}r die reine und angewandte Mathematik (Crelles Journal)}, 2023.

\bibitem{Deleporte2021}
Alix Deleporte.
\newblock Toeplitz operators with analytic symbols.
\newblock {\em The Journal of Geometric Analysis}, 31(4):3915--3967, Apr 2021.

\bibitem{demailly1997complex}
Jean-Pierre Demailly.
\newblock {\em Complex analytic and differential geometry}.
\newblock Universit{\'e} de Grenoble I Grenoble, 1997.

\bibitem{donaldson2001}
Simon Donaldson.
\newblock Scalar curvature and projective embeddings, {I}.
\newblock {\em Journal of Differential Geometry}, 59(3):479--522, 2001.

\bibitem{Donaldson15}
Simon Donaldson.
\newblock Algebraic families of constant scalar curvature {K}\"{a}hler metrics.
\newblock In {\em Surveys in differential geometry 2014. {R}egularity and evolution of nonlinear equations}, volume~19 of {\em Surv. Differ. Geom.}, pages 111--137. Int. Press, Somerville, MA, 2015.

\bibitem{Donaldson2014Gromov}
Simon Donaldson and Song Sun.
\newblock Gromov-{H}ausdorff limits of {K}\"{a}hler manifolds and algebraic geometry.
\newblock {\em Acta Mathematica}, 213(1):63--106, 2014.

\bibitem{dsz1}
Michael~R. Douglas, Bernard Shiffman, and Steve Zelditch.
\newblock Critical points and supersymmetric vacua. {I}.
\newblock {\em Comm. Math. Phys.}, 252(1-3):325--358, 2004.

\bibitem{dsz2}
Michael~R. Douglas, Bernard Shiffman, and Steve Zelditch.
\newblock Critical points and supersymmetric vacua. {II}. {A}symptotics and extremal metrics.
\newblock {\em J. Differential Geom.}, 72(3):381--427, 2006.

\bibitem{dsz3}
Michael~R. Douglas, Bernard Shiffman, and Steve Zelditch.
\newblock Critical points and supersymmetric vacua. {III}. {S}tring/{M} models.
\newblock {\em Comm. Math. Phys.}, 265(3):617--671, 2006.

\bibitem{Fefferman1974}
Charles Fefferman.
\newblock The bergman kernel and biholomorphic mappings of pseudoconvex domains.
\newblock {\em Inventiones mathematicae}, 26(1):1--65, Mar 1974.

\bibitem{chg-Goldman1999}
William~m. Goldman.
\newblock {\em Complex Hyperbolic Geometry}.
\newblock Oxford University Press, 02 1999.

\bibitem{hezari2021property}
Hamid Hezari and Hang Xu.
\newblock On a property of bergman kernels when the k{\"a}hler potential is analytic.
\newblock {\em Pacific Journal of Mathematics}, 313(2):413--432, 2021.

\bibitem{kerzman1978cauchy}
Norberto Kerzman and Elias~M Stein.
\newblock The cauchy kernel, the szeg{\"o} kernel, and the riemann mapping function.
\newblock {\em Mathematische Annalen}, 236(1):85--93, 1978.

\bibitem{krantz2002primer}
Steven~G Krantz and Harold~R Parks.
\newblock {\em A primer of real analytic functions}.
\newblock Springer Science \& Business Media, 2002.

\bibitem{Lu2000On}
Zhiqin Lu.
\newblock On the lower order terms of the asymptotic expansion of {T}ian-{Y}au-{Z}elditch.
\newblock {\em Amer.J.math}, 122(2):235--273, 2000.

\bibitem{lu2016szegHo}
Zhiqin Lu and Steve Zelditch.
\newblock Szeg{\"{o}} kernels and poincar{\'e} series.
\newblock {\em Journal d'Analyse Math{\'e}matique}, 130(1):167--184, 2016.

\bibitem{MM}
Xiaonan Ma and George Marinescu.
\newblock {\em Holomorphic {M}orse inequalities and {B}ergman kernels}, volume 254 of {\em Progress in Mathematics}.
\newblock Birkh\"{a}user Verlag, Basel, 2007.

\bibitem{ma2015exponential}
Xiaonan Ma and George Marinescu.
\newblock Exponential estimate for the asymptotics of bergman kernels.
\newblock {\em Mathematische Annalen}, 362(3):1327--1347, 2015.

\bibitem{patterson1978dennis}
SJ~Patterson.
\newblock {\em Dennis A. Hejhal, The Selberg trace formula for $PSL\left(2,R \right)$, Volume I}.
\newblock Springer Berlin, Heidelberg, 1978.

\bibitem{Rouby-Sj}
Oph{\'e}lie Rouby, Johannes Sj{\"o}strand, and San~Vu Ngoc.
\newblock {Analytic Bergman operators in the semiclassical limit}.
\newblock {\em Duke Mathematical Journal}, 169(16):3033 -- 3097, 2020.

\bibitem{Shiffman2008}
Bernard Shiffman.
\newblock Convergence of random zeros on complex manifolds.
\newblock {\em Science in China Series A: Mathematics}, 51(4):707--720, Apr 2008.

\bibitem{sz1}
Bernard Shiffman and Steve Zelditch.
\newblock Distribution of zeros of random and quantum chaotic sections of positive line bundles.
\newblock {\em Comm. Math. Phys.}, 200(3):661--683, 1999.

\bibitem{sz2}
Bernard Shiffman and Steve Zelditch.
\newblock Asymptotics of almost holomorphic sections of ample line bundles on symplectic manifolds.
\newblock {\em J. Reine Angew. Math.}, 544:181--222, 2002.

\bibitem{Shiffman2008NV}
Bernard Shiffman and Steve Zelditch.
\newblock Number variance of random zeros on complex manifolds.
\newblock {\em Geometric and Functional Analysis}, 18(4):1422--1475, Dec 2008.

\bibitem{Sun2011Expected}
Jingzhou Sun.
\newblock Expected {E}uler characteristic of excursion sets of random holomorphic sections on complex manifolds.
\newblock {\em Indiana University Mathematics Journal}, 61(3):pages. 1157--1174, 2011.

\bibitem{Punctured}
Jingzhou Sun.
\newblock Estimations of the bergman kernel of the punctured disk.
\newblock {\em arXiv preprint arXiv:1706.01018}, 2017.

\bibitem{sun2024ChengYau}
Jingzhou Sun.
\newblock Bergman kernels of the {C}heng-{Y}au metrics on quasi-projective manifolds.
\newblock {\em arxiv}, arXiv:2407.07483, 2024.

\bibitem{sun2024apde}
Jingzhou Sun.
\newblock Projective embedding of stably degenerating sequences of hyperbolic {R}iemann surfaces.
\newblock {\em Analysis \& PDE}, 17(6):1871--1886, 2024.

\bibitem{sun-ber-abelian}
Jingzhou Sun.
\newblock On the {B}ergman kernel of polarized abelian varieties.
\newblock {\em Preprint}, arxiv 2511.19003, Available at http://arxiv.org/abs/2511.19003.

\bibitem{Tian1990On}
Gang Tian.
\newblock On a set of polarized {K}\"ahler metrics on algebraic manifolds.
\newblock {\em Journal of Differential Geometry}, 32(1990):99--130, 1990.

\bibitem{Zelditch2000Szego}
Steve Zelditch.
\newblock Szego kernels and a theorem of {T}ian.
\newblock {\em International Mathematics Research Notices}, 6(6):317--331, 2000.

\end{thebibliography}

\end{document}